\documentclass[10pt]{amsart} \usepackage{amssymb, verbatim} 

\newcommand{\Mbar}{\overline{M}}
\newcommand\la{\lambda}
\newcommand\pa{\partial}
\renewcommand\k{z}
\newcommand\de{\delta}
\renewcommand\Im{{\operatorname{Im}}}
\renewcommand\Re{{\operatorname{Re}}}

\newcommand\R{{\mathbf R}}
\newcommand\G{{\mathbf{G}}}
\newcommand\e{{\mathbf{e}}}
\newcommand\C{{\mathbf{C}}}
\renewcommand\j{{\mathbf{j}}}
\newcommand\q{{\mathbf{q}}}
\renewcommand\v{{\mathbf{v}}}
\renewcommand\div{{\operatorname{div}}}

\newcommand\Q{{\mathbf{Q}}}
\newcommand\Z{{\mathbf{Z}}}

\newcommand\eps{\varepsilon}
\newcommand\M{{\mathcal M}}
\renewcommand\P{{\mathcal P}}
\newcommand\RR{{\mathcal R}}
\newcommand\A{{\mathcal A}}
\newcommand\F{{\mathcal F}}
\newcommand\N{{\mathcal N}}
\newcommand\Zflux{{\mathcal{Z}}}

\renewcommand\sc{\operatorname{sc}}
\newcommand\Psisc{\Psi_{\sc}} 
\newcommand\Hess{{\operatorname{Hess}}}
\renewcommand\a{\alpha}
\renewcommand\b{\beta}
\newcommand\ga{\gamma}
\newcommand\si{\sigma}
\newcommand\les{\lesssim}
\DeclareMathOperator{\Op}{{Op}}

\def \endprf{\hfill  {\vrule height6pt width6pt depth0pt}\medskip}

\theoremstyle{plain}
  \newtheorem{theorem}[subsection]{Theorem}
  
  \newtheorem{proposition}[subsection]{Proposition}
  \newtheorem{lemma}[subsection]{Lemma}
  \newtheorem{corollary}[subsection]{Corollary}

\theoremstyle{remark}
  \newtheorem{remark}[subsection]{Remark}

\theoremstyle{definition}
  \newtheorem{definition}[subsection]{Definition}

\include{psfig}

\begin{document}

\title[Quantitative limiting absorption]{Effective limiting absorption principles, and applications}

\author{Igor Rodnianski}
\address{Department of Mathematics, Princeton University, Princeton 
NJ 08544}
\email{ irod@math.princeton.edu}

\author{Terence Tao}
\address{Department of Mathematics, UCLA, Los Angeles CA 90095-1555}
\email{ tao@math.ucla.edu}


\vspace{-0.3in}

\begin{abstract} The limiting absorption principle asserts that if $H$ is a suitable Schr\"odinger operator,
and $f$ lives in a suitable weighted $L^2$ space (namely $H^{0,1/2+\sigma}$ for some $\sigma> 0$), 
then the functions $R(\lambda + i\eps) f := (H - \lambda - i\eps)^{-1} f$ converge in a another weighted $L^2$ space $H^{0,-1/2-\sigma}$
to  the unique solution $u$ of the Helmholtz equation $(H-\lambda) u = f$ which obeys the 
Sommerfeld outgoing radiation condition.  In this paper, we investigate more quantitative (or effective) versions of this principle, for the Schr\"odinger operator on asymptotically conic manifolds with short-range potentials, 
and in particular consider estimates of the form
$$ \| R(\lambda+i\eps) f \|_{H^{0,-1/2-\sigma}} \leq C(\lambda, H) \| f \|_{H^{0,1/2+\sigma}}.$$
We are particularly interested in the exact nature of the dependence of the constants $C(\lambda,H)$
on both $\lambda$ and $H$.  It turns out that the answer to this question is quite subtle, with distinctions being 
made between low energies $\lambda \ll 1$, medium energies $\lambda \sim 1$, and large energies $\lambda \gg 1$,
and there is also a non-trivial distinction between ``qualitative'' estimates on a single operator $H$ (possibly obeying
some spectral condition such as non-resonance, or a geometric condition such as non-trapping), and ``quantitative''
estimates (which hold uniformly for all operators $H$ in a certain class).  Using elementary methods (integration by parts
and ODE techniques), we give some sharp answers to these questions.  As applications of these estimates,
we present a global-in-time local smoothing estimate and pointwise decay estimates for
the associated time-dependent Schr\"odinger equation, as
 well as an integrated local energy decay estimate and pointwise decay estimates for
solutions of the corresponding wave equation, under some additional assumptions on the operator $H$.
\end{abstract}

\maketitle

\section{Introduction}

\subsection{The limiting absorption principle for the free Schr\"odinger operator}

The 
\emph{limiting absorption principle} for the free Schr\"odinger operator $H_0 := -\Delta$ on the Euclidean space $\R^n$ describes the behavior of the
 resolvents $R_0(\lambda \pm i\eps) := (H_0 - (\lambda \pm i\eps))^{-1}$ in the limit $\eps \to 0$
 in the weighted Sobolev spaces $H^{s,m}(\R^n)$, defined for $s = 0,1,2$ and $m \in \R$ by
$$\| f\|_{H^{s,m}(\R^n)} := \sum_{j=0}^s \| \langle x \rangle^m |\nabla^j f| \|_{L^2(\R^n)},$$
where $\langle x \rangle := (1 + |x|^2)^{1/2}$.

\begin{proposition}[Limiting absorption principle for $H_0$]\label{prop-h0}  Let $\lambda > 0$, $n \geq 3$, $\epsilon, \sigma > 0$, and $\pm$ be a sign.  
Then for any $f \in H^{0,1/2+\sigma}(\R^n)$ the problem 
$$
(H_0 - (\lambda \pm i\eps)) u = f
$$
has the unique solution $u_{\pm \eps}=R_0(\lambda \pm i\eps) f \in H^2(\R^n)$ and 
we have the estimate\footnote{Here and in the sequel, we allow all absolute constants $C$ to depend on the dimension $n$ and on exponents such as $\sigma$.}
\begin{equation}\label{lap-h0}
 \| R_0(\lambda \pm i\eps) f \|_{H^{0,-1/2-\sigma}(\R^n)} \leq C \lambda^{-1/2} \| f \|_{H^{0,1/2+\sigma}(\R^n)}
\end{equation}
for all $\eps > 0$, and we have the variant estimate
\begin{equation}\label{lap-2} \| R_0(\lambda \pm i\eps) f \|_{H^{2,-1/2-\sigma}(\R^n)} \leq C (1 + \lambda)  \lambda^{-1/2} \| f \|_{H^{0,1/2+\sigma}(\R^n)}.
\end{equation}
Furthermore, as $\eps \to 0$, 
$$
u_{\pm \eps}=R_0(\lambda \pm i\eps) f\to u_\pm = R_0(\lambda \pm i0) f \quad
{\text{in}} \,\,\,H^{2,-1/2-\sigma}(\R^n).  
$$
The function $u_\pm $ is
the unique solution in $H^{2,-1/2-\sigma}(\R^n)$ to the Helmholtz equation $(H_0 - \lambda) u = f$ which obeys
the \emph{Sommerfeld radiation condition}
$$ (\partial_r \mp i \sqrt\lambda) u_\pm \in H^{0,-1/2+\sigma'}(\R^n) \hbox{ for all } \sigma' < \sigma,$$
where $\partial_r := \frac{x}{|x|} \cdot \nabla_x$ is the radial derivative.
Moreover, 
$$ 
(\partial_r \mp i \sqrt{\lambda\pm i\eps}) u_{\pm\eps}\to (\partial_r \mp i \sqrt\lambda) u_\pm
\quad {\text{in}}\,\,\, H^{0,-1/2+\sigma'}(\R^n).
$$
Finally, if $\sigma < 1$, we may replace \eqref{lap-h0}  with the estimate
\begin{equation}\label{eq:h0}
 \| R_0(\lambda \pm i\eps) f \|_{H^{0,-3/2+\sigma}(\R^n)} \leq C (1 + \lambda)^{-1/2}
 \| f \|_{H^{0,1/2+\sigma}(\R^n)}.
\end{equation}
A similar version,  with the factor of $(1 + \lambda)^{-1/2}$ instead of 
$\lambda^{-1/2}$, holds in place of \eqref{lap-2}.
\end{proposition}

The result is well-known for the free Schr\"odinger Hamiltonian $H_0=-\Delta$. It was 
established for a more general class of differential operators with {\it constant} 
coefficients by Agmon \cite{Ag}. We will reprove this result in the course of this paper (see in particular Section \ref{cheap-sec} for an elementary argument). The sharpness of the above estimates can be easily  verified from the behavior of the kernel 
$$
R_0(\lambda \pm i\eps)(x,y)\sim C \frac {e^{i\sqrt{\la\pm i\eps}\, |x-y|}}{|x-y|^{n-2}}
$$
of the resolvent $R_0(\lambda \pm i\eps)$
for $|x-y|\ge C \la^{-1/2}$ and the bound $R_0(\lambda \pm i\eps)(x,y)\le C |x-y|^{2-n}$
for $|x-y|\le C \la^{-1/2}$.

The limiting absorption principle can be viewed as a quantitative formulation of the fact that the operator $H_0$ has no embedded eigenvalues or resonances in its absolutely continuous spectrum $[0,+\infty)$.  As is well known, it is also closely connected to the \emph{limiting amplitude principle} for the wave equation, and the \emph{local smoothing estimate} for the Schr\"odinger equation.  We shall return to these connections later in this introduction.

\subsection{Quantitative limiting absorption}

In this paper we study the problem of extending the limiting absorption principle (and its applications) in a quantitative manner to other Schr\"odinger operators $H$.  More specifically, we shall limit our attention to
operators of the form $H = - \Delta_M + V$, where $M$ is an asymptotically conic manifold and $V$ is a short-range potential; we now make these concepts more precise.  Henceforth $n \geq 3$ and $\sigma, \sigma_0 > 0$ are fixed, and we allow all constants to depend on $n$, $\sigma$, and $\sigma_0$.

\begin{definition}[Asymptotically conic manifold]\label{acm} An \emph{asymptotically conic manifold} is a
smooth connected $n$-dimensional Riemannian manifold $(M,g)$ which, outside of a compact set $K_0 \subset M$, can be parameterized 
as $M \backslash K_0 = [R_0, \infty) \times \partial M  = \{ (r,\omega): r > R_0, \omega \in \partial M \}$
for some $R_0 > 1$, where $(\partial M, h)$
is a smooth compact $n-1$-dimensional Riemannian manifold, and the metric $g$ takes the scattering form\footnote{We use the indices $a,b$
to parameterize the $n-1$-dimensional space $\partial M$, and the indices $i,j$ to parameterize the $n$-dimensional
space $M$.}
\begin{equation}\label{g-polar}
 g = dr^2 + r^2 h[r]_{ab}(\omega) d\omega^a d\omega^b
\end{equation}
where for each $r > R_0$, $h[r]$ is a metric of the form
\begin{equation}\label{hr-def}
h[r] := (h_{ab}(\omega) + r^{-2\sigma_0} e_{ab}(r,\omega))
\end{equation}
where $\sigma_0 > 0$ and the error $e_{ab}$ obeys the first derivative estimates
\begin{equation}\label{metric-decay}
 |e_{ab}(r,y)| + r |\partial_r e_{ab}(r,y)| + |\nabla_y e_{ab}(r,y)|_h \leq C
\end{equation}
for all $(r,y) \in M \backslash K_0$.  We also make the qualitative assumption that the function $e_{ab}(1/s,y)$ extends smoothly to a function on $[0,1/R_0] \times \partial M$.  We define a 
weight $\langle x \rangle$ on $M$ by setting $\langle x \rangle := r$
on $M \backslash K_0$, and extending the weight smoothly to $K_0$ so that it stays between $R_0$ and $R_0/2$.  
We define the space $L^2(M)$ by using the measure $dg(x) := \sqrt{g}\ dx$ induced by the Riemannian metric,
and let $\Delta_M := \nabla^\alpha \nabla_\alpha$ be the usual Laplace-Beltrami operator, where $\nabla_\alpha$ denotes
covariant differentiation with respect to the Levi-Civita connection, raised and lowered in the usual manner.  We define
the weighted Sobolev spaces $H^{s,m}(M)$ for $s=0,1,2$ and $m \in \R$ by the formula
\begin{equation}\label{weight-def}
\| f\|_{H^{s,m}(M)} := \sum_{k=0}^s \| \langle x \rangle^m |\nabla^k f|_g \|_{L^2(M)}.
\end{equation}
\end{definition}

\textbf{Important convention.}  Henceforth, all constants $C$ are allowed to depend on the manifold $M$ and the quantities $R_0, C, \sigma_0$ appearing above, as well as on the dimension $n$. 

\begin{remark} The class of asymptotically conic manifolds includes the class of asymptotically flat manifolds (as studied for instance in \cite{cks}, \cite{doi}, \cite{burq}, \cite{mmt}, \cite{tjoen}) as a sub-class, in which $\partial M$ is the unit sphere
$S^{n-1}$ with the standard metric.  However in general we do not assume any topological flatness of $M$,
nor do we assume that $\partial M$ is topologically a sphere.  In particular it is certainly possible for $M$ to contain trapped geodesics.  The condition \eqref{metric-decay} seems to indicate that we only require $C^1$ control on our metric.  However, this is because of our decision to use normal co-ordinates \eqref{g-polar}.  A metric in a more general long-range form, e.g. $g = g_0 + r^{-2\sigma_0} \tilde e$, where $g_0 = dr^2 + r^2 h_{ab}$ and $\tilde e$ is a smooth function obeying symbol-type estimates of order 0, can be placed in the normal form
\eqref{g-polar} (see \cite[Section 10.5]{htw}) but as is well known, the use of normal forms costs one degree of
regularity, so that one would need $C^2$ type bounds on the metric $g$ in the original co-ordinates in order to get the $C^1$-type control in \eqref{g-polar}.
\end{remark}

We define a \emph{short-range potential} $V: M \to \R$ on $M$ to be any real-valued function $V$ such that
$\langle x \rangle^{1+2\sigma_0} V(x) \in L^\infty(M)$ for some $\sigma_0 > 0$, 
and then define the \emph{Schr\"odinger operator} $H := -\Delta_M + V$.  

\begin{remark}
The spectral theory of the operator $H$ in the short-range case is well understood; indeed, it
is known, see e.g. \cite{hmv},
that $H$ is essentially self-adjoint, so that the resolvents
$R(\lambda \pm i\eps):= (H - (\lambda \pm i\eps))^{-1}$ are well-defined for $\eps \neq 0$
and are bounded operators on $L^2(M)$ (with bounds depending on $\eps$, of course).  
We note that the essential self-adjointness of $H$, without symbol behavior assumptions on the potential $V$ and
the metric coefficients of $g$, follows from a lower bound
\begin{equation}\label{eq:lower}
\int_M Hu\,\overline u\ge -C\int_M |u|^2 
\end{equation}
and uniqueness of solutions of the wave equation 
$$
u_{tt} + Hu=0,\qquad u(0)=u_0,\quad \partial_t u(0)=u_1,
$$
see \cite{berez}, \cite{chern}, \cite{shub} and references therein. The lower bound \eqref{eq:lower} easily follows 
from the $L^\infty$ bound on the potential $V$ and positivity of the Laplace-Beltrami operator $-\Delta_M$. 
The uniqueness statement is a consequence of the energy identity 
$$
\int_M \left (|\pa_t u(t)|^2 + |\nabla u(t)|^2_g + V |u(t) |^2\right) dg= \int_M \left (|\pa_t u(0)|^2 + |\nabla u(0)|^2_g + V |u(0) |^2\right) dg.
$$
Furthermore, the spectrum $\sigma(H)$ consists of the right half-line $[0,+\infty)$ (where it is purely absolutely continuous,
except possibly at 0 where one may have an eigenvalue), together with a finite number of
eigenvalues $\lambda_1 \leq \lambda_2 \leq \ldots \leq \lambda_k \leq 0$ on the negative half-line
$(-\infty,0]$; the finite number of bound states is implied by
the Cwikel-Lieb-Rozenblum inequality  for manifolds, see e.g. \cite{levin}, \cite{liyau}, since $V \in L^{n/2}$.  In particular one can define spectral multipliers $m(H)$ of $H$ for any bounded (or polynomially growing) $m$ in the usual manner.  

\end{remark}

Our first main result establishes a limiting absorption principle with explicit control on the constants in this general setting:

\begin{theorem}[Quantitative limiting absorption principle]\label{main} Let $\lambda > 0$, $n \geq 3$, $\sigma, \sigma_0 > 0$, 
and $\pm$ be a sign.  Let $M$ be an asymptotically conic manifold, let $V: M \to \R$ be a short-range potential obeying
the pointwise bounds
\begin{equation}\label{V-short}
 |V(x)| \leq A \left( \langle x \rangle^{-2-2\sigma_0} + \lambda^{1/2} \langle x \rangle^{-1-2\sigma_0} \right)
\end{equation}
for all $x \in M$, and let $H := -\Delta_M + V$ with resolvents $R(\lambda \pm i\eps):= (H - (\lambda \pm i\eps))^{-1}$.   Let $f \in H^{0,1/2+\sigma}(\R^n)$.

\begin{itemize}
\item (Limiting absorption near infinity) There exists a compact set $K=K(M,A)$ such that for $s=0,1,2$
\begin{equation}\label{aprop-ls}
 \| R(\lambda \pm i\eps) f \|_{H^{s,-1/2-\sigma}(M\setminus K)} \leq \la^{s/2} 
 C(M,A) (\lambda^{-C(M,A)} + 1)
  \| f \|_{H^{0,1/2+\sigma}(M)}.
\end{equation}
\item (Global limiting absorption)  We have the bound
\begin{equation}\label{prop-ls}
 \| R(\lambda \pm i\eps) f \|_{H^{2,-1/2-\sigma}(M)} \leq C(M,A) \left(\lambda^{-C(M,A)} + e^{C(M,A) \sqrt{\lambda} }\right) \| f \|_{H^{0,1/2+\sigma}(M)}.
\end{equation}
\end{itemize}
The constant $C(M,A)$ can be explicitly computed in terms of $M$, $A$, $n$, $\sigma$, $\sigma_0$.  
\end{theorem}

\begin{remark} A key feature of this estimate is that it is \emph{quantitative} (or \emph{effective}), in the sense that the bounds in \eqref{aprop-ls}, \eqref{prop-ls} depend only on the underlying manifold $M$ and on the bounds enjoyed by the potential $V$.  In particular, no spectral assumptions on $H$ (e.g. involving eigenfunctions or resonances at zero) are assumed.  Later on we shall also discuss \emph{qualitative} (or \emph{ineffective}) limiting absorption principles, in which the bounds are obtained indirectly (via compactness arguments or Fredholm theory) and are allowed to depend on the potential $V$ (and in particular on the spectral behavior of $V$).  In order to obtain effective bounds in Theorem \ref{main}, we shall avoid the use of compactness methods, and instead rely on ``elementary'' methods such as integration by parts and ODE analysis.
\end{remark}

\begin{remark}
The condition \eqref{V-short} on the potential demands $1/\langle x \rangle^{2+}$ type decay at low energies, but only $1/\langle x \rangle^{1+}$ decay at high energies.  It appears to be essentially 
optimal in the class of pointwise bounds. It is almost scale invariant
under the transformation $x\to \mu x$, \,$\la\to \mu^{-1/2} \la$.
The limiting absorption principle 
in the Euclidean case $H=-\Delta_{\R^n} + V$ 
was established for $\lambda\ge c>0$
by Agmon \cite{Ag} essentially under the assumption that 
$|V(x)|\le A\langle x\rangle^{-1-2\si_0}$. The 
global-in-time local smoothing estimate and Strichartz estimates 
for the Schr\"odinger group $e^{it(-\Delta_{\R^3}+V)}$
were proved in \cite{BK} and \cite{rs} respectively under the
assumption that $|V(x)|\le A\langle x\rangle^{-2-2\si_0}$.
These results are qualitative as their proof relies on Fredholm theory.  In \cite{BPS} (see also the followup papers \cite{stalker}, \cite{gus}), these results were extended to the class of inverse square potentials $V(x)=A/|x|^2$; due to the explicit nature of the eigenfunctions in this case, Fredholm theory can be avoided, and the results here are quantitative.

It would be natural to replace pointwise conditions of \eqref{V-short} 
corresponding integral assumptions on $V$, i.e., on the scale-invariant
$L^{n/2}$-norm of $V$. In the flat case there is a number of recent 
qualitative results, see \cite{GS}, \cite{IS}.

The second term in \eqref{V-short} suggests  that one could
weaken the decay hypotheses on $V$ in exchange for more regularity.  It is well known however that the limiting absorption principle can fail for potentials decaying 
like $1/|x|$ or slower.  As the standard example with the Wigner-Von-Neumann
potential shows, such potentials can create embedded eigenvalues in the continuous 
spectrum, which would clearly destroy the limiting absorption principle in this case.
\end{remark}

\begin{remark}
The bound \eqref{aprop-ls} shows that (at least for high energy $\la$) 
geometry of the ``black box" compact region $K$ essentially does not affect the behavior
of the resolvent restricted to the complement of $K$. Such bounds for sufficiently large 
$\la$ with the constant $C(M,A,\la)\le C(M,A) \la^{-1/2}$ have been proven to hold
in \cite{cv}, see also \cite{bshape}. It is also interesting to note that in this context, local
in time Strichartz estimates, restricted to the complement of a compact set, 
for solutions of a time-dependent Schr\"odinger equation have been established in
\cite{bt}.
\end{remark}

\begin{remark}  The high-energy case $\lambda \gg 1$ of the above theorem is essentially contained in earlier work of Burq \cite{b} and Cardoso-Vodev \cite{cv}, \cite{cardoso}.  Thus the main novelty here is the ability to treat medium energies $\lambda \sim 1$ and low energies $\lambda \ll 1$, as well as the universality of the technique developed to treat all ranges of energies.
\end{remark}

The above theorem only reprises one component of Proposition \ref{prop-h0}, namely the estimates \eqref{lap-h0}, \eqref{lap-2}.  Using those estimates, however, one can fairly easily obtain convergence properties of the resolvent as $\eps \to 0$.

\begin{proposition}[Limiting values of resolvent]\label{qlap} Let the notation and assumptions be
as in Theorem \ref{main}.  Let $0<\sigma<\min (1,\sigma_0)$ and $f \in H^{0,1/2+\sigma}(M)$. 
Let $\la\pm i\eps=z^2$ with $z=a\pm b$ such that $b>0$,
Then the functions $u_{\pm \eps} :=
R(\lambda \pm i\eps) f$ obey the 
 \emph{Sommerfeld radiation condition}
\begin{equation}\label{sommerfeld}
\begin{split} 
&\| (\partial_r \mp i z) u_{\pm \eps}\|_{H^{0,-1/2+\sigma}(M \backslash K_0)}
+\| \nabla_\omega u_{\pm \eps}\|_{H^{0,-3/2+\sigma}(M \backslash K_0)}\\
\quad &\leq C( M, K_0, A) (\lambda^{-C(M,A)} + e^{C(M,A) \sqrt{\lambda} })
\| f \|_{H^{0,1/2+\sigma}(M)}
\end{split}
\end{equation}
where $|\nabla_\omega u|_g^2=r^2(|\nabla u|^2_g-|\pa_r u|^2)$.
Note that the constants are uniform in the choice of $\eps$. Furthermore,
for a fixed sign $\pm$ and fixed $\la>0$, 
the functions $u_{\pm \eps}$ converge in $H^{0,-1/2-\sigma}(M)$ 
to a limit $u_{\pm}$, which is
the unique solution in $H^{2,-1/2-\sigma}(M)$ to the Helmholtz equation $(H - \lambda) u = f$ such that
$(\partial_r \mp i \lambda^{1/2}) u_{\pm }$ lies in $H^{0,-1/2+\sigma'}(M \backslash K_0)$
for at least one $0 < \sigma' < \sigma$. 
\end{proposition}

We prove Theorem \ref{main} and Proposition \ref{qlap} in Section \ref{sec:lim-abs}.

\subsection{Refinements at high energy}

At high energies $\la \gg 1$, Theorem \ref{main} suffers an exponential loss in the constants.  We now recall the standard quasimode example that shows that this loss is necessary for certain manifolds $M$:

\begin{proposition}[Quasimode construction]\label{pseudo}  Let $C_0 > 0$ and $\sigma > 0$ be arbitrary.
There exists a smooth Riemannian manifold $M$, which is equal to
Euclidean space $\R^n$ outside of a compact set, and a sequence $\lambda_l \to +\infty$ and $\eps_l \to 0^+$ of
real numbers, as well as functions $f_l \in H^{0,1/2+\sigma}(M)$ such that if $H := -\Delta_M$ (i.e. $V \equiv 0$)
then
$$ 
 \| R(\lambda_l \pm i\eps_l) f_l \|_{H^{0,-1/2-\sigma}(M)} > e^{C_0 \sqrt{\lambda_l}} \| f_l \|_{H^{0,1/2+\sigma}(M)}
$$
for all $l$.
\end{proposition}

We give this standard counterexample (based on the trapped geodesics in a sphere) in Section \ref{quasimode-sec}.  It shows that we cannot hope to eliminate the exponential factor $e^{C(M,A) \sqrt{\lambda}}$ from Theorem \ref{main}, at least for manifolds which
contain trapped geodesics.  

However, as is well known, the limiting absorption properties improve substantially at high energies if one assumes a non-trapping condition on the manifold $M$. We give a (standard) result in this direction, with somewhat stronger assumptions on the manifold and potential than is
strictly necessary:  

\begin{theorem}[Quantitative high-energy non-trapping limiting absorption principle]\label{nontrap} Let 
the notation and assumptions be as in Theorem \ref{main}.  Suppose also that $V$ obeys the 
bound\footnote{Alternatively we may assume that $|V(x)|+ \langle x \rangle
|\nabla V(x)|_g\le A \langle x \rangle^{-2\sigma_0}$, which 
allows $V$ to be a long range potential.}
$$ |V(x)| \leq A  \langle x \rangle^{-1-2\sigma_0} $$
for some $A > 0$ and all $x \in M$, and that the metric perturbation error 
function $e_{ab}$ obeys the symbol
estimates
\begin{equation}\label{metric-decay-complete}
 |(r \partial_r)^\alpha \nabla_y^\beta \e_{ab}(r,y)| \leq C_{\alpha \beta}
\end{equation}
for all multi-indices $\alpha, \beta \geq 0$.  Suppose also that the manifold $M$ is \emph{non-trapping}
in the sense that every unit speed geodesic $t \mapsto \gamma(t)$ in $M$ escapes to the infinite boundary
$\partial M$ as $t \to \pm \infty$.  Then, for $\lambda$ sufficiently large depending on $M$ and $A$,
we have
\begin{equation}\label{prop-ls-highenergy}
 \| R(\lambda \pm i\eps) f \|_{H^{s,-1/2-\sigma}(M)} \leq C(M,A) \lambda^{(s-1)/2} \| f \|_{H^{0,1/2+\sigma}(M)}
\end{equation}
for all $0 \leq s \leq 2$.
\end{theorem}

We prove this Theorem in Section \ref{sec:high-energy} 
using the positive commutator method and the
scattering pseudodifferential calculus.  

\begin{remark} Such results for non-trapping metrics have a long history starting with the work \cite{Vainberg}, \cite{lp} for compact perturbations of the free Hamiltonian $H_0=-\Delta_{\R^n}$. In the general case a related result is contained in \cite{Vodev}, see also \cite{b}  for a qualitative version, and \cite{RT}.  For the closely related task of establishing local smoothing estimates, results analogous to those above were established in \cite{cks}.
\end{remark}

\begin{remark}  In between the two extremes of non-trapping manifolds (for which one has the optimal $\lambda^{-1/2}$ decay in the limiting absorption principle at high energies) and the highly trapping counterexamples (for which one has exponential growth in $\sqrt{\lambda}$) there is a complicated range of intermediate behaviour, for instance if there is just one hyperbolic trapped orbit.  We will not attempt to survey the substantial literature on this topic, but see \cite{ikawa}, \cite{ikawa-2}, \cite{nz}, \cite{burq-2004}, \cite{datchev}, \cite{c08}, \cite{c08b}, \cite{cwun} for some work in this area.
\end{remark}

\begin{remark} If we replace the assumption on the potential
by the weaker bound $|V(x)|\le A \la^{1/2} \langle x\rangle^{-1-2\si_0}$  
then one would have to modify the non-trapping condition as the potential term $V$ has the same order
as the kinetic term $-\Delta_M$ and one would have to consider the Hamilton flow associated to the symbol of $-\Delta_M + V$ rather than the geodesic flow.  We omit the (standard) formalization of this fact.

Note that by combining this Theorem with Theorem \ref{main} one can weaken the hypothesis that $\lambda$ is sufficiently
large, instead requiring that $\lambda > \lambda_0$ for some fixed $\lambda_0 > 0$ (but then the constant $C(M,A)$
in \eqref{prop-ls-highenergy} may depend on $\lambda_0$).
\end{remark}

\begin{remark}
The proof of Theorem \ref{nontrap} is microlocal. The result therefore can be strengthened
to include the situation where there is trapping. In this case we consider the conical 
set 
$$
{\mathcal M}_{\operatorname{nontrap}}=\{(x,\xi):\,\,\ga(t)\,\,\ga(0)=x,\,\dot \ga(0)=\xi\,\,\,
\mbox{escapes to}\,\, \partial M,\,t\to\pm\infty\}
$$
invariant under the geodesic flow in the phase space and a corresponding 
$\Psi$DO $P_{\operatorname{nontrap}}$ with support in ${\mathcal M}_{\operatorname{nontrap}}$. Then we can replace 
\eqref{prop-ls-highenergy} with 
$$
 \| R(\lambda \pm i\eps) P_{\operatorname{nontrap}} f \|_{H^{s,-1/2-\sigma}(M)} \leq C(M,A) \lambda^{s-1/2} \| f \|_{H^{0,1/2+\sigma}(M)}
$$
\end{remark}

\subsection{Refinements at low energy}

At the low energy limit $\la \to 0$, the constant in \eqref{prop-ls} (or \eqref{aprop-ls}) blows up like a large negative power of $\lambda$.  For the purposes of \emph{effective} bounds, which are uniform for all potentials $V$ obeying the bound \eqref{V-short}, such blowup is necessary, as the following result shows:

\begin{proposition}[Bessel equation matching construction]\label{bessel} Let $C_0 > 0$ and $\sigma, \sigma_0 > 0$ be arbitrary.
Then there exists a sequence $V_m: \R^n \to \R$ of spherically symmetric potentials obeying the uniform bounds
$$ \sup_m \| \langle x \rangle^{1+\sigma_0} \nabla_x^\alpha V_m \|_{L^\infty(\R^n)} < \infty \hbox{ for all }
\alpha \geq 0,$$
sequences $\lambda_m \to 0^+$ and $\eps_m \to 0^+$ of real numbers with $0 < \eps_m < \lambda_m$,
as well as functions $f_m \in H^{0,1/2+\sigma}(\R^n)$ such that if $H_m := -\Delta + V_m$ then
$H_m$ has no resonance or eigenfunction at zero, and
$$ 
 \| R(\lambda_m \pm i\eps_m) f_m \|_{H^{0,-1/2-\sigma}(M)} > \lambda_m^{-C_0} \| f_m \|_{H^{0,1/2+\sigma}(M)}
$$
for all $m$.  
\end{proposition}

We give this standard example in Section \ref{bessel-matching-sec}.  This proposition shows that the large negative power in \eqref{prop-ls} is necessary.  One can also show that a negative power is necessary in \eqref{aprop-ls} by combining this example with the Carleman estimate (Proposition \ref{ucp}); we omit the details.

Despite this example, one can eliminate (or at least reduce) the low energy blowup in certain cases.  Firstly, in the flat case $M \equiv \R^n$ the the work of Jensen-Kato \cite{jk} and Jensen \cite{Jensen} provides 
small energy resolvent expansions 
\begin{equation}\label{eq:exp-KJ}
(H-\la)^{-1} = \sum_{j=-2}^\ell \la^{j/2} B_j + o(\la^{\ell/2})
\end{equation}
for Hamiltonians 
$H=-\Delta_{\R^n} +V$ in odd\footnote{In even dimensions resolvent expansions involves additional 
terms containing $\log \la$.} dimensions $n$, under appropriate conditions on $V$, dependent in particular 
on the index $\ell$, even in the presence of zero eigenvalues and resonances.
The expansion \eqref{eq:exp-KJ} shows that for a {\it fixed} potential $V$ the limiting 
absorption constant for the resolvent $(H-\la)^{-1}$ blows up at worst as 
$C(M,V) \lambda^{-1}$ as $\lambda\to 0$. Note however that \eqref{eq:exp-KJ} is 
qualitative in a sense that the bounds on the norms $B_j$, as operators between 
weighted Sobolev spaces, depend on the potential $V$ and not just on the bound $A$ for that potential; thus there is no inconsistency between \eqref{eq:exp-KJ} and Proposition \ref{bessel}.  Note however subtle distinction between effective bounds (which hold uniformly for all $V$ obeying \eqref{V-short}) and ineffective bounds (which are not uniform in $V$).  

It is quite likely that these expansions could be extended to asymptotically conic manifolds.  We will not do so in full generality here, but just give the standard qualitative limiting absorption estimate in the case that there are no eigenfunctions or resonances at zero:

\begin{definition}[Resonances at zero]\label{resonance}
Let $\lambda = \kappa^2$ for some $\kappa \geq 0$.  We say that 
 $H$ has an outgoing (resp. incoming) \emph{resonance at $\lambda$} if 
the equation $Hu=\lambda u$ has a solution $u\in H^{0,-3/2+\sigma}(\R^n)$  for some $\si>0$
with the property that $u\not\in \bigcup_{\si>0}H^{0,-1/2+\sigma}(\R^n)$ and
$\pa_r u - i \kappa u\in H^{0,-1/2+\sigma'}$ (resp. $\pa_r u + i \kappa u\in H^{0,-1/2+\sigma'}$)
for some $\sigma'>0$.  Note that when $\lambda = 0$, there is no distinction between incoming and outgoing resonances, and we shall talk simply about resonances at zero.
\end{definition}

\begin{proposition}[Qualitative low-energy limiting absorption principle]\label{lowenergy-jk} 
Let the notation and assumptions be as in Theorem \ref{main}.  Suppose also that $V$ is strongly short-range
in the sense that $\langle x \rangle^{2+2\sigma_0} V \in L^\infty(M)$.  Assume also that the operator $H := -\Delta_M + V$
contains no eigenfunction or resonance at zero.  Then for $\lambda, \eps \neq 0$ sufficiently small
depending on $M$ and $V$, we have
\begin{equation}\label{cmvl}
 \| R(\lambda \pm i\eps) f \|_{H^{2,-1/2-\sigma}(M)} \leq C(M,V) |\lambda|^{-1/2} \| f \|_{H^{0,1/2+\sigma}(M)}.
\end{equation}
Furthermore, if $\sigma > 1/2$, then we can strengthen this further, to
\begin{equation}\label{cmvl2}
 \| R(\lambda \pm i\eps) f \|_{H^{2,-1/2-\sigma}(M)} \leq C(M,V) \| f \|_{H^{0,1/2+\sigma}(M)}.
\end{equation}
\end{proposition}

The proof of this result is essentially an application of the Fredholm alternative\footnote{Strictly speaking, we do not apply the Fredholm alternative directly, as we will need to uniformly invert a \emph{family} $1 + R_0(\lambda \pm i\eps) V$ of compact perturbations of the identity, rather than a single such perturbation, but instead will be using the \emph{proof} of the Fredholm alternative in our arguments.} and the theory for the free case $V=0$; we give this proof at the end of this section.  The estimates here should be compared with \eqref{prop-ls}; the constants do not blow up as fast as $\lambda \to 0$, but on the other hand the bounds depend directly on $V$ and not just on $A$.

\begin{remark} As is well known, in asymptotically Euclidean case in high dimensions $n \geq 5$ no resonances at zero can occur; see \cite{Jensen}. Moreover,
zero eigenfunctions and resonances are non-generic relative to a family of perturbations
$H_\kappa=-\Delta_M + \kappa V$; see \cite{jk}.
\end{remark}

\begin{remark}  The above proposition is only stated for sufficiently small $\lambda$.  But by combining this result with Theorem \ref{main} we see that it in fact holds for all bounded $\lambda$.  If we also assume non-trapping, then by Theorem \ref{nontrap} we can extend \eqref{cmvl} and \eqref{cmvl2} to arbitrary energies $\lambda > 0$.
\end{remark}
\begin{remark}
The subject of qualitative limiting absorption principles has a long and rich history. Most of the results have been focused 
on a high or a fixed non-zero energy regime in various geometric settings, such as the asymptotically Euclidean, conic, hyperbolic and 
obstacle cases; see for instance \cite{lp}, \cite{melrose}, \cite{iu}, \cite{hassell-2}, \cite{hassell} for a representative set of results. The low energy results have been
developed for example in \cite{gh1}, \cite{gh2}, \cite{bou}. 
For the operator $-\Delta_{{\bf R}^n}+V$ resolvent behavior have analyzed
in particular in \cite{Ag}, \cite{jk}, \cite{Jensen}, \cite{GS}. Recently there have been a lot of interest in establishing sub-exponential,
in fact logarithmic,  bounds on the resolvent at high energies for geometries with sufficiently ``small" or filamentary trapped sets,
see \cite{bgh}, \cite{nz}, \cite{wz}, \cite{dv}.

\end{remark}

In some cases we can make the absence of eigenfunctions or resonances at zero quite quantitative.  For instance, when the potential is mostly positive we have the following result.

\begin{proposition}[Quantitative low-energy limiting absorption principle for mostly positive 
potentials]\label{lowenergy-pos} 
Let the notation and assumptions be as in Theorem \ref{main}.  Suppose also that $V$ obeys the bounds
$$ |V(x)| \leq A \langle x \rangle^{-2-2\sigma_0}$$
for some $A > 0$, and suppose also that the negative part $V_- := \max(-V,0)$ of $V$ obeys the smallness condition
$$ \int_M |V_-(x)|^{n/2} \leq \beta(M)$$
for some sufficiently small $\beta > 0$ depending only on $M$.  Then $H$ has no eigenfunctions or resonances at any energy $\lambda \in \R$
(and in particular at $\lambda = 0$).  Furthermore, if $\la, \eps \neq 0$ are
sufficiently small depending on $M$ and $A$, then
\begin{equation}\label{leo}
 \| R(\lambda \pm i\eps) f \|_{H^{2,-1/2-\sigma}(M)} \leq C(M,A) |\lambda|^{-1/2} \| f \|_{H^{0,1/2+\sigma}(M)}.
\end{equation}
Finally, for all 
sufficiently small $|\la|, \eps$ and $\si<\min(\si_0,1)$
we have the following uniform bound 
\begin{equation}\label{prop-lsd}
 \| R(\lambda \pm i\eps) f \|_{H^{2,-3/2+\sigma}(M)} \leq C({\mathcal I}(M),A) \| f \|_{H^{0,1/2+\sigma}(M)},
\end{equation}
where the constant $C({\mathcal I}(M),A)$ depends 
on the manifold $(M,g)$ only through the isoperimetric constant ${\mathcal I}(M)$.
\end{proposition}

We prove this result in Section \ref{sec:low-energy}.

\begin{remark}
Observe that the Euclidean limiting absorption principle, Proposition \ref{prop-h0}, 
now follows from
Theorem \ref{main}, Proposition \ref{qlap}, Theorem \ref{nontrap} (since the Euclidean space $\R^n$ is non-trapping), and Proposition \ref{lowenergy-pos} (with $V \equiv 0$).  
\end{remark}

Note that Proposition \ref{lowenergy-pos} already applies in the free case $V=0$.  From this and Fredholm theory we can now prove Proposition \ref{lowenergy-jk}:

\begin{proof}[Proof of Proposition \ref{lowenergy-jk}]  We begin with the proof of \eqref{cmvl2}.  We write $H = H_0+V$ where $H_0 = -\Delta_M$ is the free Laplacian, and denote the resolvents for $H_0$ by $R_0(\lambda \pm i\eps)$.  We have the resolvent identity
$$ R(\lambda \pm i\eps) = (1 + R_0(\lambda \pm i\eps) V)^{-1}  R_0(\lambda \pm i\eps).$$
From \eqref{prop-lsd} (and duality and interpolation), the resolvent $R_0$ already obeys \eqref{cmvl2} for any $\sigma > 1/2$, and so it suffices to show that the operators $1 + R_0(\lambda \pm i\eps) V$ are uniformly invertible on $H^{2,-1/2-\sigma}(M)$ for $\sigma$ sufficiently close to (but larger than) $1/2$.  From the Rellich embedding theorem and the strongly short-range nature of $V$, we know that $V$ maps $H^{2,-1/2-\sigma}(M)$ compactly to $H^{0,1/2+\sigma}(M)$, and so by \eqref{cmvl2} for the free operator, the operators $R_0(\lambda \pm i\eps) V$ are uniformly compact on $H^{2,-1/2-\sigma}(M)$ for $\lambda$ sufficiently small.  

Suppose that uniform invertibility fails, then we can find a sequence $\lambda_n \pm i \eps_n$ going to zero, functions $u_n$ of unit magnitude in $H^{2,-1/2-\sigma}(M)$, and $f_n$ going to zero strongly in $H^{2,-1/2-\sigma}(M)$ such that 
$$ (1 + R_0(\lambda_n \pm i\eps_n) V) u_n = f_n.$$
If we then write $\tilde u_n := R_0(\lambda \pm i\eps) V u_n = f_n - u_n$ and $\tilde f_n := R_0(\lambda \pm i\eps) V f_n$, then $u_n$ has norm $1+o(1)$ in $H^{2,-1/2-\sigma}(M)$ and lies in a fixed compact subset of that space, $\tilde f_n$ goes to zero strongly in $H^{2,-1/2-\sigma}(M)$, and
$$ (1 + R_0(\lambda_n \pm i\eps_n) V) \tilde u_n = \tilde f_n.$$
Applying $H_0 - (\lambda_n \pm i \eps_n)$ to both sides, we conclude that
$$ H \tilde u_n = (\lambda_n \pm i\eps_n) \tilde u_n + (H_0 - (\lambda_n \pm i \eps_n)) \tilde f_n.$$
By compactness, we may pass to a subsequence such that $\tilde u_n$ converges strongly in $H^{2,-1/2-\sigma}(M)$ to a limit $\tilde u$ of norm $1$.  Taking distributional limits, we conclude that $H \tilde u = 0$ in the sense of distributions.  But this contradicts the hypothesis that $H$ has no resonance or eigenvalue at zero.

The proof of \eqref{cmvl} is similar, but one takes $\sigma$ close to zero instead of to $1/2$, and uses $H^{0,3/2-\sigma}(M)$ instead of $H^{0,1/2+\sigma}(M)$ (and relies on \eqref{leo} and the adjoint of \eqref{prop-lsd}) to mediate between $R_0(\lambda \pm i\eps)$ and $V$; we omit the details.
\end{proof}

\subsection{Applications}

We now give some applications of the above limiting absorption estimates.  We begin with some (well-known) spectral consequences.

\begin{proposition}[Absence of embedded eigenvalues or resonances]\label{absent}  
Let the notation and assumptions be as in Theorem \ref{main}, and let $\lambda > 0$.  Then there are no non-zero eigenfunctions or resonances (incoming or outgoing) at $\lambda$.  
\end{proposition}

\begin{proposition}[Absence of embedded singular continuous spectrum]\label{singular} The spectrum of $H$ on $(0,+\infty)$ is purely absolutely continuous.
\end{proposition}

We prove these results in Section \ref{spectral-sec}.

In Section \ref{rage-sec} we give a version of the celebrated RAGE theorem (cf. \cite{ruelle}, \cite{amrein}, \cite{enss}).

\begin{proposition}[RAGE theorem]\label{rage}  Let the notation and assumptions be as in Theorem \ref{main}.  Let $f \in L^2(M)$ be orthogonal to all eigenfunctions of $H$.  Then for any compact set $K$, we have 
\begin{equation}\label{eq:rage}
\lim_{t \to \pm \infty} \| e^{itH} f \|_{L^2(K)} = 0.
\end{equation}
Similarly, we have
$$\lim_{t \to \pm \infty} \| u(t) \|_{H^1(K)} + \|\partial_t u(t) \|_{L^2(K)} = 0$$ 
for a solution of the wave equation $\pa_t^2 u+Hu=0$ 
with initial data $(u(0), \partial_t u(0))\in H^1(M) \times L^2(M)$, orthogonal to the eigenfunctions of $H$.
\end{proposition}

Next, we use the limiting absorption principle and the RAGE theorem, together with a closely related result that gives H\"older continuity bounds on the resolvent, to derive the \emph{limiting amplitude principle} for the wave equation.

\begin{proposition}[Limiting amplitude principle]\label{lamp}  Let the notation and assumptions be as in Theorem \ref{main}, let $f \in L^2(M)$ be compactly supported, and let $u_0 \in H^1(M)$, $u_1 \in L^2(M)$, and $\mu > 0$.  Assume that $f, u_0, u_1$ are all orthogonal to all the eigenfunctions of $H$.
Let $u: \R \times M \to \C$ be the solution to the inhomogeneous wave equation
\begin{equation}\label{eq:wave}
\pa_t^2 u -\Delta_M u = e^{i\mu t} f,\quad u|_{t=0}=u_0,\,\,\,\,\pa_t u|_{t=0}=u_1.
\end{equation}
Then for any compact set $K \subset M$ we have
$$ \| u(t) - e^{i\mu t} v \|_{H^1(K)} \to 0 \hbox{ as } t \to +\infty$$
where $v$ is the outgoing solution of the Helmholtz problem 
$$
(H - \mu^2) v =f,
$$ 
i.e. $v$ is such that $(\pa_r-i|\mu|) v\in H^{0,-1/2+\sigma'}(\R^n)$.
\end{proposition}

We prove this result in Section \ref{lamp-sec}.

Now we give a global-in-time local smoothing estimate for the Schr\"odinger equation.

\begin{proposition}[Global-in-time local smoothing estimate for $H$]\label{glos}  Let the notation and
assumptions be as in Theorem \ref{main}.  Suppose that $M$ is a non-trapping manifold, that $V$ obeys the bounds
$$ |V(x)| \leq A \langle x \rangle^{-2-2\sigma_0},$$
and
that $H := - \Delta_M + V$ has no eigenfunction or resonance at zero. Let 
$u: \R \times \R^n \to \C$ be a solution to the forced Schr\"odinger equation 
$$ iu_t - H u = F.$$
and let $P_H$ denote the projection on the continuous spectrum of $H$.
Then we have the estimate
\begin{equation}\label{glos-1}
 \int_\R \| H^{1/2} P_H u(t) \|_{H^{0,-1/2-\sigma}(M)}^2\ dt
\leq C(M,V,A) \left( \| u(0) \|_{H^{1/2}(M)}^2 + \int_\R \| F(t) \|_{H^{0,1/2+\sigma}(M)}^2\ dt\right).
\end{equation}
If furthermore $\sigma > 1/2$, then we have the variant estimate
\begin{align*}
\int_\R \left (\| H^{1/4} P_H u(t) \|_{H^{0,-1/2-\sigma}(M)}^2+
\| P_H u(t) \|_{H^{0,-1/2-\sigma}(M)}^2\right)\ dt
\leq C(M,V,A) \times &\\
\left( \| u(0) \|_{L^2(M)}^2 + \int_\R \| F(t) \|_{H^{0,1/2+\sigma}(M)}^2\ dt\right)&.
\end{align*}
If we assume that $V$ obeys the hypotheses of Proposition \ref{lowenergy-pos} then we may 
eliminate the projection $P_H$ and replace
the constants $C(M,V,A)$ here by $C(M,A)$. 

Finally, if we assume symbol estimates on both the metric coefficients $e_{ab}$ and the 
potential $V$ then we  can replace the last estimate with the family of bounds
\begin{align*}
\int_\R \| P_H u(t) \|_{H^{s+1/2,-1/2-\sigma}(M)}^2&\ dt
\leq C(s,M,V,A) \times \\
&\times \left( \| u(0) \|_{H^{s}(M)}^2 + \int_\R \| F(t) \|_{H^{s-1/2,1/2+\sigma}(M)}^2\ dt\right)
\end{align*}
for all $s \geq 0$.
\end{proposition}

We prove this proposition in Section \ref{ls-sec}.

\begin{remark}
The nontrapping assumption of Theorem \ref{glos} can be removed if one replaces the projection $P_H$ 
by the projection $P_H^\Lambda$ on the continuous spectrum with energies $\la<\Lambda<\infty$.
\end{remark}

\begin{remark} The close connection between limiting absorption principles and local smoothing (they are essentially Fourier transforms of each other, with $t$ being the dual variable to $\lambda$) was first observed by Kato \cite{kato1}.  See \cite{simon2} for some further discussion. In Euclidean space, the global-in-time local smoothing estimates were first established in \cite{sjolin}, \cite{vega}, \cite{constantin}.  In order to obtain estimates which are global in time, it is necessary (by the uncertainty principle) to establish limiting absorption principles at very low energies; high-energy analysis alone is only sufficient to establish local-in-time local smoothing estimates.
\end{remark}

\begin{remark} It is very likely that the above global-in-time local smoothing estimate will imply global-in-time Strichartz estimates, by adapting the arguments in \cite{st}, \cite{tataru}, \cite{bt}, \cite{bt2}, \cite{mmt}.  This would allow one to create slightly more quantitative formulations of some of the Strichartz estimates in \cite{bt2} and \cite{mmt}, however these improvements seem to be rather minor and so we will not detail them here.
\end{remark}

The limiting absorption principle for the Hamiltonian $H$ also leads to the
integrated local energy decay for the wave equation. 

\begin{proposition}[Integrated local energy decay]  \label{prop:int-decay}
Suppose that $M$ is a non-trapping manifold, that $V$ obeys the bounds
$$ |V(x)| \leq A \langle x \rangle^{-2-2\sigma_0},$$
and
that $H := - \Delta_M + V$ has no eigenfunction or resonance at zero. Let 
$u: \R^{n+1} \to \C$ be a solution of the wave equation
\begin{equation}\label{eq:w-inh}
\pa_t^2u + H u = F.
\end{equation}
and let $P_H$ denote the projection on the continuous spectrum of $H$.
Then we have the estimate
\begin{equation}\label{int-decay-1}
 \begin{split}
 \int_\R &\left (\| \pa_t P_H u(t) \|_{H^{0,-1/2-\sigma}(M)}^2+
\| \nabla P_H u(t) \|_{H^{0,-1/2-\sigma}(M)}^2\ dt\right)
\leq C(M,V,A) \times\\ &\times \left( \| \nabla u(0) \|_{L^2(M)}^2 + \| u_t(0) \|_{L^2(M)}^2\right.+
\left.\int_\R \| F(t) \|_{H^{0,1/2+\sigma}(M)}^2\ dt\right).
\end{split}
\end{equation}
For $\sigma<\min (1,\si_0)$ 
the retarded solution 
$$ u_{\operatorname{ret}}(t) := \int_{t' < t} \frac{\sin( (t-t') \sqrt{H} )}{\sqrt{H}} F(t')\ dt'$$
of the inhomogeneous problem \eqref{eq:w-inh}
obeys additional bounds
\begin{equation}\label{int-decay-2}
 \begin{split}
 \int_\R \Big(\| (\pa_t-\pa_r) P_H u_{\operatorname{ret}}(t) \|_{H^{0,-1/2+\sigma}(M)}^2&+
\| r^{-1}\nabla_\omega P_H u_{\operatorname{ret}}(t) \|_{H^{0,-1/2+\sigma}(M)}^2\ dt\Big)
\\ &\leq C(M,V,A) \int_\R \| F(t) \|_{H^{0,1/2+\sigma}(M)}^2\ dt.
\end{split}
\end{equation}
A similar estimate holds for the advanced solution 
$$ u_{\operatorname{adv}}(t) := -\int_{t' > t} \frac{\sin( (t-t') \sqrt{H} )}{\sqrt{H}} F(t')\ dt'$$
with $(\pa_t+\pa_r)$ in place of
$(\pa_t-\pa_r)$.

Furthermore, even if $M$ does not satisfy a non-trapping condition there exists a compact 
set $K\subset M$ such that
\begin{align*}
 \int_\R &\left (\| \pa_t P_H u(t) \|_{H^{0,-1/2-\sigma}(M\setminus K)}^2+
\| \nabla P_H u(t) \|_{H^{0,-1/2-\sigma}(M\setminus K)}^2\ dt\right)
\leq C(M,V,A)\\ &\times \left( \| \nabla u(0) \|_{H^1(M)}^2 + \| u_t(0) \|_{H^1(M)}^2\right.+
\left.\int_\R \| F(t) \|_{H^{1,1/2+\sigma}(M)}^2\ dt\right).
\end{align*} 
If we assume that $V$ obeys the hypotheses of Proposition \ref{lowenergy-pos} then we may 
eliminate the projection $P_H$ and replace
the constants $C(M,V,A)$ by $C(M,A)$. 
\end{proposition}

We prove this in Section  \ref{ls-sec}.
\begin{remark}
The statement of integrated local energy decay for $H_0=-\Delta_{{\bf R}^n}$ goes back to Morawetz \cite{morawetz}.
The proof of such estimates for solutions of the wave equation on black hole spacetimes,  with geometries
which are quite different from the ones considered in this paper, have been instrumental in understanding
their global behavior. See \cite{bs}, \cite{dr1}, \cite{mmt} for Schwarzschild, \cite{dr2}, \cite{tt}, \cite{ab}, \cite{dr5}
for slowly rotating Kerr and \cite{dr4} for the general sub-extremal Kerr case.

\end{remark}

Our final results are pointwise decay estimates for the solutions of the Schr\"odinger
and wave equations, obtained by commuting these equations with a Morawetz-type operator, applyign energy estimates, and using an iteration argument to amplify the resulting decay.

\begin{proposition}[Decay for Schr\"odinger]\label{prop:decay-S}
Let $M$ be a non-trapping manifold with metric $g$ given in \eqref{g-polar} with $\si_0>1/2$. 
We assume 
that $h[r]:=h+e$ satisfies the following assumptions.
$$
|(r\pa_r)^k (\nabla_\omega^\alpha) h[r]|\le C_{k\alpha},\qquad k\le 3, \,\,|\alpha|\le 2.  
$$
Let 
$\psi: \R \times \R^n \to \C$ be a solution to the Schr\"odinger equation 
$$ i\pa_t \psi +\Delta_M \psi = 0.$$
Then for any $t \geq 0$ we have the dispersive estimates
\begin{align*}
\|\psi(t)\|_{L^\infty(M)} &\le C_\eps (1+t)^{-\frac 32+\eps} \|\psi(0)\|_{H^{2,1}(M)},\qquad \forall\eps>0,
\quad n=3,\\
\|\psi(t)\|_{L^p(M)} &\le C (1+t)^{-2+\frac 4p} \|\psi(0)\|_{H^{2,1}(M)},\qquad \forall 2\le p<\infty,\quad
n=4,\\
\|\psi(t)\|_{L^{\frac {2n}{n-4}}(M)} &\le C (1+t)^{-2} \|\psi(0)\|_{H^{2,1}(M)},\qquad\qquad n\ge 5.
\end{align*}
\end{proposition}
\begin{remark}
It is well known that 
a solution of the free Schr\"odinger equation $i\pa_t \psi+\Delta_{{\bf R}^n}\psi=0$ satisfies the dispersive estimate 
$$
\|\psi(t)\|_{L^\infty(\R^n)}\leq C t^{-n/2} \|\psi_0\|_{L^1(\R^n)}.
$$
As a consequence, we do not believe the rates of decay given by Proposition \ref{prop:decay-S} to be sharp, especially in higher dimensions.  Nevertheless, they appear to be novel in such general setting, and we give them here to illustrate an application 
of how the limiting absorption principle (or more precisely, the global-in-time local smoothing estimate) can be used to obtain dispersive inequalities. 
\end{remark}

\begin{remark}
Dispersive estimates for solutions of the Schr\"odinger equation with $H=-\Delta_{{\bf R}^n}+V$, 
projected on the continuous spectrum of $H$ and assuming absence of zero eigenvalues and resonances
have been proved in \cite{rauch}, \cite{jss}, \cite{rs}, \cite{bg}; see also the survey \cite{sch} and the references therein.
\end{remark}

\begin{proposition}[Decay for wave]\label{prop:wave-S}
Let $M$ satisfy the assumptions of Proposition \ref{prop:decay-S} and let 
$u: \R^{n+1} \to \C$ be a solution of the wave equation
$$
\pa_t^2u - \Delta_M u = 0.
$$
Then in dimension $n=3$ we have
$$
\|u(t)\|_{L^\infty(M)} \le C_\eps (1+t)^{-1+\eps} \left (\|\nabla u(0)\|_{H^{1,1}(M)}+
\|u_t(0)\|_{H^{1,2}(M)}\right ),$$
for all $\eps>0$.
\end{proposition}

\begin{remark} A modification of our arguments also gives the variant estimate
$$ \|\psi(t)\|_{L^{\frac {2(n-1)}{n-3}}(M)} \le C (1+t)^{-1}\left (\|\nabla u(0)\|_{H^{1,1}(M)}+
\|u_t(0)\|_{H^{1,1}(M)}\right )$$
in higher dimensions $n>3$.
\end{remark}

\begin{remark}
Decay estimates for solutions of the wave equation in Minkowski space, $M={\bf R}^n$, have, 
of course, a very long history, including the Huygens principle in odd dimensions and the uniform 
$t$-decay with the rate of $t^{-(n-1)/2}$. A quantitative approach to decay in Minkowski space 
has been developed by Klainerman, in what is known as the vector field method, \cite{kl}.
Qualitative decay results in non-trapping geometries have been obtained in the pioneering works 
\cite{lmp}, \cite{V}, \cite{mrs}, \cite{mel}. For the problem with $-\Delta_{{\bf R}^n}+V$ see e.g. \cite{georgiev}.
Quantitative decay results for solutions of the wave equation on black hole spacetimes have been
obtained in \cite{bs}, \cite{dr1}, \cite{dr2}, \cite{dr3}, \cite{ab}, \cite{dr5}, \cite{tat1}, \cite{dr4}.
A general approach to the derivation of decay from the integrated local energy decay statements 
have been developed in \cite{dr3}, where the arguments can be in particular adapted to the asymptotically
conic case, and in \cite{tat1} for stationary asymptotically Euclidean spacetimes. See also \cite{yang} for 
applications to nonlinear problems. The results here follow
an earlier approach of \cite{dr1} and serve as an illustration.
\end{remark}

We prove these propositions in Section \ref{schro-decay} and Section \ref{wave-decay} respectively.

\subsection{Acknowledgements}

We thank Andrew Hassell, Jared Wunsch, Rafe Mazzeo and Mihalis Dafermos for helpful discussions, Nicolas Burq, Hans Christianson and Georgi Vodev for corrections and references, and particularly Rowan Killip and Michael Hitrik for conversations on the significance of the limiting absorption principle.  The first author is supported by NSF grant DMS-0702270. The second author is supported by NSF grant CCF-0649473, the NSF Waterman award, and a grant from the MacArthur Foundation.

\section{Key estimates}\label{key-est}

Throughout this paper we use the notation and assumptions of Theorem \ref{main}.  All constants are henceforth allowed to depend on $M, n, \sigma_0, A$, and $\sigma$.

Fix $\lambda, \eps > 0$, and suppose that we have a solution to the \emph{resolvent equation}
\begin{equation}\label{resolvent}
(H - (\lambda \pm i\eps)) u = f
\end{equation}
(or equivalently that $u = R(\lambda \pm i\eps) f$)
and the closely related \emph{Helmholtz equation}
\begin{equation}\label{helmholtz}
 (-\Delta_M - z^2) u = F, \qquad z^2=\la\pm i\eps
\end{equation}
Of course, the former equation can be viewed as a special case of the latter with $F := f - Vu$.  To avoid technicalities, we shall always make the qualitative assumptions $u, f, F \in L^2(M)$.

In this section we lay out the fundamental estimates for these equations which we shall repeatedly use in our analysis.  The proof of these estimates will be deferred to later sections.  We begin with a simple charge estimate.

\begin{lemma}[Charge estimate]\label{charge-est}  Let $\lambda, \eps, \sigma > 0$, let $\pm$ be a sign,
let $f \in H^{0,1/2+\sigma}(M)$, and let $u = R(\lambda \pm i\eps) f \in L^2(M)$.  
Then we have
\begin{equation}\label{eq:eps-charge}
 \eps \int_M |u|^2\ dg \leq \| f \|_{H^{0,1/2+\sigma}(M)} 
\| u \|_{H^{0,-1/2-\sigma}(M)}.
\end{equation} 
Moreover, for any $\alpha \in \R$ we have
\begin{align}
&\eps^2 \int_{\langle x\rangle \ge 2R} \langle x\rangle^{2\alpha} |u|^2\, dg
 \leq  C(\alpha)
 \int_{\langle x\rangle \ge R} 
\left(\langle x\rangle^{2\alpha} |f|^2+\langle x\rangle^{-2+2\alpha}|\nabla u |_g^2\right)\ dg ,
\label{eq:eps-charge-weight}
\end{align}
\end{lemma}

This easy estimate can be established by integration by parts, and is proven in Section \ref{charge-proof}.  Roughly speaking, this estimate allows us to handle most of the terms in our analysis which contain a factor of $\eps$, and which do not have derivatives on $u$.

Another integration by parts gives the following standard elliptic estimates, proven in Section \ref{energy-sec}:

\begin{lemma}[Elliptic regularity]\label{elliptic-regularity}  Let $H = -\Delta_M + V$ where $V$ is a short-range potential obeying \eqref{V-short}. Let $m \in \R$, and let $u, f \in H^{0,m}(M)$ which satisfies the
Helmholtz equation $(H-z^2) u = f$ in the distributional sense (at least) with $z^2=\la\pm i\epsilon$ 
and $\la\ge 0$.
\begin{itemize}
\item We have $u \in H^{2,m}(M)$ with the elliptic regularity estimate
\begin{equation}\label{elliptic-0} \| u \|_{H^{2,m}(M)} \leq C(m) ( \| f \|_{H^{0,m}(M)} + (1+\lambda) \| u \|_{H^{0,m}(M)} ).
\end{equation}
\item For any $R > 2R_0$, we have the localised elliptic regularity estimates
\begin{equation}\label{energy-spiral}
\int_{R \leq \langle x \rangle \leq 2R} |\nabla u|_g^2\ dg
\leq C \int_{R/2 \leq \langle x \rangle \leq 3R} (\lambda + R^{-2}) |u|^2 + \frac{|f|^2}{\lambda + R^{-2}}\ dg,
\end{equation}
and
\begin{equation}\label{energy-spiral'}
\int_{R \leq \langle x \rangle \leq 2R} |\nabla u|_g^2\ dg
\leq C \int_{R/2 \leq \langle x \rangle \leq 3R} (\lambda + R^{-2}) |u|^2 +R^2 
{|f|^2}\ dg
\end{equation}
and
\begin{equation}\label{elliptic-local}
\int_{\langle x \rangle \geq R} |\nabla^2 u|_g^2 \langle x \rangle^{2m}\ dg
\leq C(m,R) [ \|f\|_{H^{0,m}(M)}^2 +
(1+\lambda) \int_{\langle x \rangle \geq R/2} |u|_g^2 \langle x \rangle^{2m}\ dg ].
\end{equation}
\item We have the energy estimate
\begin{equation}\label{eq:la-nab}
\lambda^{\frac 12}  \|u\|_{H^{0,-1/2-\sigma}} \le C \left (\|\nabla u\|_{H^{0,-1/2-\sigma}}
+ \|u\|_{H^{0,-3/2-\sigma}} + \|f\|_{H^{0,1/2+\sigma}}\right).
\end{equation}
\item If $\eps\le C \la$ for some constant $C>0$ then we have the charge-type estimate
\begin{equation}\label{eq:la-eps-nab}
\eps  \|\nabla u \|_{L^2(M)}\|u\|_{L^2(M)} \le C \|f\|_{H^{0,1/2+\sigma}}
\left (\|\nabla u\|_{H^{0,-1/2-\sigma}}
+  \|u\|_{H^{0,-3/2-\sigma}} + \|f\|_{H^{0,1/2+\sigma}}\right).
\end{equation}
as well as the localised estimates
\begin{equation}\label{eq:la-eps-loc}
\begin{split}
\eps^2  \int_{\langle x\rangle\ge 2R} |\nabla u|_g^2 \,  \int_{\langle x\rangle\ge 2R} |u|_g^2
 &\le C \int_{\langle x\rangle\ge R} \langle x\rangle^{1+2\si} |f|^2\\&\times
\int_{\langle x\rangle\ge R} \left (\langle x\rangle^{-1-2\si} |\nabla u|_g^2+
\langle x\rangle^{-3-2\si} |u|^2+\langle x\rangle^{1+2\si} |f|^2\right).
\end{split}
\end{equation}
and
\begin{equation}\label{eq:la-loc}
\la \int_{\langle x\rangle\ge 2R} \langle x\rangle^{-1-2\si}
 |u|^2 \, \le C 
\int_{\langle x\rangle\ge R} \left (\langle x\rangle^{-1-2\si} |\nabla u|_g^2+
\langle x\rangle^{-3-2\si} |u|^2+\langle x\rangle^{1+2\si} |f|^2\right).
\end{equation}
\end{itemize}
\end{lemma}

Very roughly speaking, these estimates allow us to exchange derivatives on $u$ with factors of $\lambda^{1/2}$ or $1 + \lambda^{1/2}$ whenever necessary.

Next, by applying the positive commutator method\footnote{One can also view this method as another application of integration by parts.} to a first-order differential operator (a variant of $\partial_r$) we obtain a useful estimate which allows us to control the portion of $u$ in the far region $r \gg 1$ by the portion of $u$ in the intermediate region $r \sim 1$.

\begin{lemma}[Pohozaev-Morawetz type estimate]\label{morawetz-lemma}  Let the notation and assumptions be as in Theorem \ref{main}.
In addition we require that $\eps\le C\la$ for some positive constant $C$.
If $r_0 \geq R_0$ is a sufficiently large number (depending only on $M$ and $A$) then we have
\begin{align*} \int_{\langle x \rangle \geq 2r_0}& \langle x \rangle^{-3-2\sigma} |u|^2
\ dg \leq \\
&C(r_0) \left(\int_{r_0 \leq \langle x \rangle \leq 2r_0} (1 + \lambda) |u|^2\ dg + \| f \|_{H^{0,1/2+\sigma}(M)}^2\right)
\end{align*}
and
\begin{align*}
 \int_{\langle x \rangle \geq 2r_0} \big(\langle x\rangle^{-1-2\si} |\nabla u|_g^2\ dg
&+ \lambda \langle x \rangle^{-1-2\sigma} |u|^2\big)\ dg\\& \leq 
C(r_0) \left(\int_{r_0 \leq \langle x \rangle \leq 2r_0} (1 + \lambda) |u|^2\ dg + \| f \|_{H^{0,1/2+\sigma}(M)}^2\right).
\end{align*}
\end{lemma}

We prove this result in section \ref{pm}.  The positive commutator method, when applied to a suitable pseudodifferential operator, will also give Theorem \ref{nontrap}; see Section \ref{sec:high-energy}.

We also need the following variant of the above estimate, which we prove in Section \ref{improvement-sec}, using a more refined analysis based on spherical energies rather than on the positive commutator method.

\begin{lemma}\label{lem:improvement}
Let $u \in H^2(M)$ be a solution to the resolvent equation
$$ (H - (\lambda \pm i\eps)) u = f,\qquad \la\pm i\eps=z^2=(a\pm i b)^2$$
for some $\lambda, \eps > 0$ with the Hamiltonian $H=-\Delta_M +V$
satisfying the assumptions of Theorem \ref{main}.
Then for any $0<\sigma<\min (1,\si_0)$ 
and a sufficiently large $r_0$,
\begin{align}
\|u\|_{H^{0,-3/2+\si}(M_{2r_0})}\le C(r_0) \Big (\|f\|_{H^{0,1/2+\si}(M_{r_0})}
&+(1+\la^{1/2})\|u\|_{L^2(M_{r_0}\setminus M_{4r_0})}\Big ),\label{eq:u-better}\\
|z|^{1/2}\|u\|_{H^{0,-1/2-\si}(M_{2r_0})}+\|\nabla 
u\|_{H^{0,-1/2-\si}(M_{2r_0})}&\le C(r_0) \|f\|_{H^{0,1/2+\si}(M_{r_0})}\nonumber
\\ &+C(r_0)(1+\la^{1/2})\|u\|_{L^2(M_{r_0}\setminus M_{4r_0})}\label{eq:unab-better}
\end{align}
with $M_{r_0}=\{x\in M:\, \langle x\rangle\ge r_0\}$. Furthermore,
\begin{equation}\label{eq:urad-better}
\begin{split}
\|r^{-1}\nabla_\omega u\|_{H^{0,-1/2+\sigma}(M_{2r_0})}+
 \| u_r\mp i z u\|_{H^{0,-1/2+\sigma}(M_{2r_0})}&\le C(r_0) \|f\|_{H^{0,1/2+\si}(M_{r_0})}\nonumber
\\ &+C(r_0)(1+\la^{1/2})\|u\|_{L^2(M_{r_0}\setminus M_{4r_0})}.
\end{split}
\end{equation}
\end{lemma}

\begin{remark}
The result of Lemma \ref{lem:improvement} should be compared with Lemma 
\ref{morawetz-lemma}. First, the analog of \eqref{eq:u-better} in 
Lemma \ref{morawetz-lemma} provides control of the 
$\la^{1/2} \|u\|_{H^{0,-3/2-\si}(M_{2r_0})}$ norm. This improvement will be important
in our analysis of the low energy regime. Furthermore, in Lemma 
\ref{morawetz-lemma} validity of the estimate \eqref{eq:unab-better} is restricted 
to the region $\eps\le C\la$. Removal of this condition is crucial for establishing a uniform 
(in the domain $\eps,\la$) 
limiting absorption principle. Finally, \eqref{eq:urad-better} is the key ingredient in establishing
the Sommerfeld radiation condition of Proposition \ref{qlap}.
\end{remark}

We will need to complement Lemma \ref{morawetz-lemma} by the following estimate, which controls the portion of $u$ in the near region $r \ll 1$ by the intermediate region $r \sim 1$.

\begin{proposition}[Unique continuation estimate]\label{ucp} Let $K$ be a compact subset of $M$, and let $K'$ be a 
compact set contained in the interior of $K$.  Suppose that the potential $V$ obeys the bound $|V(x)| \leq A$
for $x \in K$.  Then if $u \in H^2(M)$ solves the Helmholtz equation $(H-\lambda\mp i\eps) u = f$ on $K$, then we have the estimate
$$ \int_K |u|^2 + |\nabla u|_g^2\ dg \leq C(K,K',A) e^{C(K,K',A) \sqrt{\lambda}} 
\left( \int_K |f|^2\ dg + \int_{K \backslash K'} |u|^2 + |\nabla u|_g^2\ dg \right).$$
\end{proposition}

One immediate consequence of this proposition is that any solution of the equation $(H-\lambda)u = 0$ which vanishes near the boundary of $K$, also vanishes on the interior of $K$, which explains the terminology ``unique continuation''.  Proposition \ref{ucp} is essentially due to Burq \cite{burq}, but in the interest of self-containedness we provide a proof in Section \ref{carleman-sec} using the Carleman inequality method (which, again, can be viewed as a type of integration by parts).  Notice the exponential dependence on the energy here, which will eventually lead to the exponential factors in Theorem \ref{main}.

To conclude the proof of limiting absorption for general manifolds, we will need to analyse how solution $u$ decays at infinity.  To this end we introduce the spherical masses
\begin{equation}\label{mass-def}
\M[r] := r^{n-1} \int_{\partial M} |u(r,\omega)|^2\ dh[r](\omega)
\end{equation}
for $r > R_0$, where the angular metric $h[r]$ was defined in \eqref{hr-def}.

\begin{remark} In the Euclidean case, at least, we expect $u$ to decay like $O(r^{-(n-1)/2})$, and so we expect $\M[r]$ to stay bounded.
\end{remark}

\begin{lemma}[Dichotomy]\label{dichotomy}  Let $C_1 \gg R_0$ be a large number, and then let $C_2 \gg C_1$ be an even larger number.  Let $M_{C_1/2} := \{ x \in M: r > C_1/2 \}$.
Let $u \in H^2(M_{C_1/2})$ be a solution to the resolvent equation \eqref{resolvent}
for some $\lambda, \eps > 0$ and sign $\pm$, so that $u$ is also a solution to the Helmholtz equation \eqref{helmholtz} 
with $F := f - Vu \pm i \eps u$.  Suppose also that we have the normalization
\begin{equation}\label{uf}
\| u \|_{H^{0,-1/2-\sigma}(M_{C_1/2})} = 1; \quad \| f \|_{H^{0,1/2+\sigma}(M_{C_1/2})} = \delta
\end{equation}
for some $\delta > 0$.   Assume that $\sigma < 1/2$, that
$\de\le 1$ and that $\eps \le \la$.  If $C_1$ is sufficiently large (but not depending on $\lambda$), and $C_2$ is sufficiently large depending on $C_1$ (but not on $\lambda$), then one of the following must be true:
\begin{itemize}
\item (Boundedness) There exists a radius $C_1 \leq r_0 \leq C(C_1,C_2)$ such that
\begin{equation}\label{bounded-mass}
 \M[r] \leq C(C_1,C_2) (\lambda^{-C(C_1,C_2)} + 1) \delta \hbox{ for all } r_0/2 \leq r \leq 4r_0.
\end{equation}
\item (Exponential growth) For all $C_1 \leq r \leq 10C_1$, we have
\begin{equation}\label{mass-growth}
 \frac{d}{dr} \M[r] \leq - C_2 (1 + \lambda^{1/2}) \M[r].
\end{equation}
\end{itemize}
\end{lemma}

This lemma is proven by an ODE analysis of the equations of motion obeyed by the mass $\M[r]$ and several other related quantities; it is the main technical innovation of this paper and is proven in Section \ref{sec:ODE}, after some important preliminaries in Section \ref{bessel-sec}.  Observe that this is a ``black box'' result, in that no assumptions whatsoever are made concerning $u$ in the region $r \leq C_1/2$.  The dichotomy between \eqref{bounded-mass} and \eqref{mass-growth} may seem strange, but one way to motivate it is by considering the free case $M = \R^n$, $V=0$, $f=0$, $\eps = 0$ with an ansatz $u(r\omega) = r^{-(n-1)/2} v(r) Y_l(\omega)$ for some spherical harmonic $Y_l$ of order $l$, in which case \eqref{resolvent} simplifies to the Bessel equation
$$v_{rr} + (\lambda - \frac{L(L-1)}{r^2}) v = 0$$
where $L := l + \frac{n-1}{2}$; the quantity $\M[r]$ is essentially just $|v(r)|^2$ in this case.  The solutions to this ODE can be described in terms of Bessel and Hankel functions.  By analysing the asymptotics of such functions one can observe that all solutions are either bounded or grow exponentially as $r$ decreases from infinity\footnote{Of course, exponential growth as $r$ decreases from infinity is the same thing as exponential decay as $r \to \infty$.  But in our application, it is best to think of $r=\infty$ as the ``initial'' value of $r$ for equations such as the Bessel equation.}, which helps explain the above dichotomy.

\section{Proof of main theorem}\label{sec:lim-abs}

We now show how the estimates in the previous section can quickly be used to deduce the main limiting absorption principles in our paper, namely Theorem \ref{main} and Proposition \ref{qlap}.  (The remaining limiting absorption principles will require slight modifications of the above estimates, which we shall address in later sections.)

The strategy will be as follows.  Proposition \ref{ode} shows us that the spherical mass $\M[r]$ exhibits either boundedness \eqref{bounded-mass} or exponential growth \eqref{mass-growth}.  The exponential growth scenario \eqref{mass-growth} will turn out to be incompatible with the unique continuation estimate in Proposition \ref{ucp}, and so we in fact have boundedness for $u$ in the intermediate region $r \sim 1$.  We will then use Lemma \ref{morawetz-lemma}, Proposition \ref{ucp}, and Lemma \ref{elliptic-regularity} to recover control of $u$ in other regions of space (and with higher derivatives), thus leading to the full limiting absorption principle.

\subsection{Proof of Theorem \ref{main}}\label{main-proof}  

We begin by proving Theorem \ref{main}.  Let the notation and assumptions be as in the theorem.  We allow all constants to depend on $M, A$.  Write
$u := R(\lambda \pm i \eps) f$, thus 
$$ (H - (\lambda \pm i \eps)) u = f$$
and $u \in H^2(M)$ (by \eqref{eq:eps-charge} and \eqref{elliptic-0}).

We begin by proving \eqref{aprop-ls}.  We wish to show that
\begin{equation}\label{usu}
\| u \|_{H^{s,-1/2-\sigma}(M \backslash K)} \leq C (\lambda^{-C} + \lambda^{s/2}) \|f\|_{H^{0,1/2+\sigma}(M)}
\end{equation}
for $s=0,1,2$ and for sufficiently large $K$.  By interpolation it suffices to establish the cases $s=0,2$.  From \eqref{elliptic-local} we see that the $s=2$ case follows from the $s=0$ case, so it suffices to take $s=0$.

Henceforth $s=0$.  We may now assume $\eps < \lambda$, since the claim \eqref{usu} follows from \eqref{eq:eps-charge} otherwise.  

Let $r_0$ be a large constant, and take $M_{r_0/2} := \{ x \in M: \langle x\rangle \ge r_0/2\}$.
We then normalize
\begin{equation}\label{normalization}
 \| u \|_{H^{0,-1/2-\sigma}(M_{r_0/2})} = 1; \quad \| f \|_{H^{0,1/2+\sigma}(M)} = \delta
\end{equation}
We may assume that $\de \le 1$, since the claim \eqref{usu} is immediate otherwise. 

Let $r_0 \ll C_1 \ll C_2$ be large numbers (depending on $M$ and $A$, but independent of $\lambda$) to be chosen later.  If one then applies Lemma \ref{dichotomy}, we conclude that we are either in the boundedness scenario \eqref{bounded-mass} or the exponential growth scenario \eqref{mass-growth}.

Suppose we are not in the boundedness scenario \eqref{bounded-mass}, so that we are necessarily in the exponential growth scenario \eqref{mass-growth}.  Then we see in particular from Gronwall's inequality that
$$ \int_{C_1 \leq \langle x \rangle \leq 2 C_1} |u|^2\ dg \geq
e^{C_2 (1+\sqrt{\lambda})} \int_{3 C_1 \leq \langle x \rangle \leq 10 C_1} |u|^2\ dg.$$
Using the energy inequality \eqref{energy-spiral'} and the normalization \eqref{normalization} we have
$$ \int_{4 C_1 \leq \langle x \rangle \leq 8 C_1} |\nabla u|^2\ dg
\leq C(C_1) \left(\int_{3 C_1 \leq \langle x \rangle \leq 10 C_1} (1+\la)|u|^2\ dg + \delta^2\right).$$
We may assume that the first-term on the right-hand side dominates the second, since we are not in the
boundedness scenario \eqref{bounded-mass}.  We then conclude that
\begin{equation}\label{scoot}
 \int_{C_1 \leq \langle x \rangle \leq 2 C_1} |u|^2\ dg \geq
 e^{\frac 12C_2 (1+\sqrt{\lambda})} \int_{4 C_1 \leq \langle x \rangle \leq 8 C_1} 
 \left (|u|^2 + |\nabla u|_g^2\right)\ dg.
 \end{equation}
The constant $C_2$ is sufficiently large compared to $C_1$.
Then using Proposition \ref{ucp} with $K=\{x: \langle x\rangle \le 8C_1\}$ and 
$K'=\{x: \langle x\rangle \le 4C_1\}$ we obtain 
\begin{align*}
 \int_{\langle x \rangle \leq 8 C_1} 
 \left (|u|^2 + |\nabla u|_g^2\right)\ dg &\le C(C_1) e^{C(C_1) \sqrt{\lambda}} 
 \int_{\langle x \rangle \leq 8 C_1} |f|^2\ dg\\ 
 &\le C(C_1) e^{C(C_1) \sqrt{\lambda}} \de^2.
\end{align*} 
Applying \eqref{scoot} again we conclude (again taking $C_2$ large compared to $C_1$)
$$
 \int_{4C_1 \leq \langle x \rangle \leq 8 C_1} |u|^2\ dg \leq e^{-\frac{1}{4} C_2(1+\sqrt{\lambda})} \delta^2.$$
Applying Lemma \ref{morawetz-lemma} we conclude
$$
 \int_{\langle x \rangle \leq 8 C_1} \lambda \langle x \rangle^{-1-2\sigma} |u|^2\ dg \leq C(C_1) (1+\lambda) e^{-\frac{1}{4} C_2(1+\sqrt{\lambda})} \delta^2$$
and the claim \eqref{usu} follows from \eqref{normalization}. 
 
Now suppose instead that we are in the boundedness half of the dichotomy, thus there exists a radius $C_1 \leq r_0 \leq C(C_1,C_2)$
such that
$$ \int_{r_0/2 \leq \langle x \rangle \leq 4r_0} |u|^2\ dg \leq 
C(C_1,C_2) (\lambda^{-C(C_1,C_2)} + 1) \delta.$$
Using Lemma \ref{morawetz-lemma} as before, we conclude
$$ \int_{\langle x \rangle \geq 4r_0} \lambda \langle x \rangle^{-1-2\sigma} |u|^2\ dg \leq C(r_0,C_1,C_2) (1+\lambda)
(\lambda^{-C(C_1,C_2)} + 1) \delta$$
and \eqref{usu} follows.

Finally, we prove \eqref{prop-ls}. In view of \eqref{elliptic-0}, it suffices to show that
$$
\| u \|_{H^{0,-1/2-\sigma}(M)} \leq C(\lambda^{-C} + e^{C\sqrt{\lambda}}) \|f\|_{H^{0,1/2+\sigma}(M)}.
$$
But this follows from \eqref{aprop-ls} and Proposition \ref{ucp}.  This completes the proof of Theorem \ref{main}.  

\subsection{Proof of Proposition \ref{qlap}}\label{qlap-proof}  

We now turn to the proof of Proposition
\ref{qlap}.  

The Sommerfeld radiation condition \eqref{sommerfeld} follows immediately from 
Theorem \ref{main} and \eqref{eq:urad-better}.
Their combination gives the estimate
$$
\|r^{-1}\nabla_\omega u\|_{H^{0,-1/2+\sigma}(M_{2r_0})}+
 \| u_r - i \sqrt{\lambda} u\|_{H^{0,-1/2+\sigma}(M_{2r_0})}\leq
 C (\lambda^{-C} + e^{C \sqrt{\lambda}}) \| f \|_{H^{0,1/2+\sigma}(M)}.$$
and completes the proof of the first part of Proposition
\ref{qlap}.
 
To show that for fixed $\la>0$ the the functions ${u_{\pm}}_\eps= R(\la\pm i\eps) f$ 
converge in $H^{0,-1/2-\si}$ to a unique solution ${u_\pm}$ of the 
Helmholtz equation $(H-\la) u=f$ selected by the requirement that 
$u\in H^{2,-1/2-\si}(M)$ and 
$(\pa_r\mp i\la^{1/2}) u\in H^{0,-1/2+\si'}(M\setminus K_0)$ for 
some $\si'>0$ we consider the difference 
$$
w(x)= e^{i(z_1-z_2)s(x)} {u_+}_{\eps_2}(x) - {u_+}_{\eps_1}(x) 
$$
 with $z_{1,2}^2 = \la+i\eps_{1,2}$ and 
 $\Im (z_{1,2})\approx \eps_{1,2} \la^{-1/2}>0$ (we assume that $\la\gg\eps_{1,2}$) and 
 $\eps_1>\eps_2$ (convergence for ${u_-}_\eps$ can be treated in a
 similar fashion.) The function $s(x)$ is assumed to be smooth and 
 obeys the requirement that $s(x)=r$ for $x\in M\setminus K_0$.
 The function $w$ verifies the equation
 $$(H-\la-i\eps_1) w=  G$$
 where $G$ is given by the formula
 $$G:=i (z_2-z_1)e^{i(z_1-z_2)s(x)} 
 \Big ( 2( \nabla_k s \nabla^k - i z_2\,|\nabla s|^2 ) 
 +\Delta_M s+ i(z_1+z_2)(|\nabla s|^2-1) \Big ){u_+}_{\eps_2}.
 $$
 For values of $\eps_{2}<\eps_1\ll \la$ we have that
 $$
 |z_1-z_2|\le 4 \Im (z_1-z_2)\approx \la^{-1/2} (\eps_1-\eps_2).
 $$
Therefore the first part of Proposition \ref{qlap} and the choice of function $s(x)$, which 
in particular gives that $|\Delta_M s(x)|\le C \langle x\rangle^{-1}$, imply
that for any $0<\si'<\si$ and any $\alpha > 0$ we have the estimates
\begin{align*}
\|e^{i(z_1-z_2)s(x)} ( \nabla_k s \nabla^k &- i z_2\,|\nabla s|^2 ) {u_+}_{\eps_2}  \|_{H^{0,1/2+\si'}(M)}
\\ &\le C
|z_1-z_2|^{-1+\si-\si'}  (\lambda^{-C} + e^{C \sqrt{\lambda}}) \| f \|_{H^{0,1/2+\sigma}(M)},\\
\|e^{i(z_1-z_2)s(x)}(|\Delta_M s|&+ |z_1+z_2|\,|\nabla s|^2-1)  {u_+}_{\eps_2}\|_{H^{0,1/2+\si'}(M)}
\\&\le C \|e^{i(z_1-z_2)s(x)} \langle x\rangle^{-1} {u_+}_{\eps_2}\|_{H^{0,1/2+\si'}(M)}\\ &\le 
C |z_1-z_2|^{-\si'-\alpha}  (\lambda^{-C} + e^{C \sqrt{\lambda}}) \| f \|_{H^{0,1/2+\sigma}(M)}.
\end{align*}
 Thus for any $0<\si'<\si< 1-\alpha$ we have
 $$
 \|G\|_{H^{0,1/2+\si'}(M)}\le (\eps_1-\eps_2)^{\si-\si'}  C (\lambda^{-C} + e^{C \sqrt{\lambda}}) 
 \| f \|_{H^{0,1/2+\sigma}(M)},
 $$
 By Theorem \ref{main} and the first part of Proposition \ref{qlap}, the
 solution of the problem $(H-\la-i\eps_2) w=G$ obeys 
 \begin{equation}\label{eq:converg}
 \|(\pa_r-\la^{1/2}) w\|_{H^{0,-1/2+\si'}(M\setminus K_0)}+
 \|w\|_{H^{2,-1/2-\si}(M)}\le (\eps_1-\eps_2)^{\si-\si'}  C (\lambda^{-C} + e^{C \sqrt{\lambda}}) 
 \| f \|_{H^{0,1/2+\sigma}(M)}.
 \end{equation}
 Since by Theorem \ref{main},
 \begin{align*}
  \|(e^{i(z_1-z_2)s(x)}-1) {u_+}_{\eps_2}\|_{H^{0,-1/2-\si}(M)}&\le 
C   \|\langle x\rangle^{-\si+\si'} (e^{i(z_1-z_2)s(x)}-1) {u_+}_{\eps_2}\|_{H^{0,-1/2-\si'}(M)}
 \\ &\leq C (\eps_1-\eps_2)^{\si-\si'}    \| {u_+}_{\eps_2}\|_{H^{0,-1/2-si'}(M)}\\ &\le
 (\eps_1-\eps_2)^{\si-\si'} C (\lambda^{-C} + e^{C \sqrt{\lambda}}) 
 \| f \|_{H^{0,1/2+\sigma}(M)}.
 \end{align*}
 we obtain
 \begin{align*}
 \| {u_+}_{\eps_2} - {u_+}_{\eps_1}\|_{H^{0,-1/2-\si}(M)}&\le 
\|w\|_{H^{0,-1/2-\si}(M)}+  \|(e^{i(z_1-z_2)s(x)}-1) {u_+}_{\eps_2}\|_{H^{0,-1/2-\si}(M)}
\\ &\le (\eps_1-\eps_2)^{\si-\si'} C (\lambda^{-C} + e^{C \sqrt{\lambda}}) 
 \| f \|_{H^{0,1/2+\sigma}(M)},
 \end{align*}
 which implies that ${u_+}_\eps$ is a Cauchy sequence in ${H^{2,-1/2-\sigma}(M)}$ 
 converging to the limit ${u_+}\in {H^{2,-1/2-\sigma}(M)}$.  
 Moreover, \eqref{eq:converg} also implies
 $(\pa_r-i\la^{1/2}) {u_+}\in {H^{0,-1/2+\si'}(M\setminus K_0)}$ for any $\si'<\si$.
 
 To prove uniqueness of solutions of the Helmholtz equation 
 $$
 (H-(\la\pm i\eps)) u =f
 $$
 with the property that $u\in H^{2,1/2-\si}(M)$ and 
$(\pa_r\mp i\la^{1/2}) u\in H^{0,-1/2+\si'}(M\setminus K_0)$ for 
some $\si'>0$ we simply repeat the argument above with the functions 
$$
v(x)= e^{\pm i(\sqrt{\la \pm i\eps}-\sqrt\la)s(x)} {u}(x) - {u_{\pm}}_{\eps}(x) 
$$
and conclude that ${u_\pm}_\eps$ converges to $u$.

This concludes the proof of Theorem \ref{main} and Proposition \ref{qlap}, except for the proofs of the estimates in Section \ref{key-est}.  These estimates will be the focus of the next few sections.

\section{Conservation laws}\label{sec:conserv}

All of the results stated in Section \ref{key-est} rely, fundamentally, on integration by parts arguments.  It will be convenient to present these arguments in a unified framework, namely that of exploiting the conservation laws for the Helmholtz equation
\eqref{helmholtz} using the \emph{abc method} of Friedrichs.  While this framework is in practice
too general to use directly, it does provide a convenient way to display the interrelationships
between the more specialized methods we study below.  Thus we shall devote this section to the
general abc method for the Helmholtz equation \eqref{helmholtz}.  As a quick application, we will be able to establish the charge estimate, Lemma \ref{charge-est}.

\subsection{Densities and currents}

Let $u \in H^2(M)$ be a (complex-valued) solution to the Helmholtz equation \eqref{helmholtz} for 
some $\lambda > 0$, so that $F \in L^2(M)$.
We introduce a number of (real) tensor fields on $M$, namely
\begin{itemize}
\item the \emph{charge density} $\q := |u|^2$,
\item the \emph{energy density} 
$\e := |\nabla u|_g^2 := \Re( \overline{\nabla^k u} \nabla_k u )$,
\item the \emph{charge current} $\j^k := \Im( \overline{u}\, \nabla^k u )$,
\item the \emph{charge gradient} $\v^k := \Re( \overline{u}\, \nabla^k u )$, 
\item the \emph{energy-momentum 
tensor} $\Q^{mk} := \Re( \overline{\nabla^m u} \nabla^k u )
- \frac{1}{2} g^{mk} (\e - |z|^2 \q)$.
\end{itemize}

Note that from the assumption $u \in H^2(M)$ that all of these tensor fields 
are absolutely integrable. A direct computation (recalling that the Levi-Civita covariant 
derivative $\nabla$ respects the metric $g$
and is torsion free) yields the \emph{charge gradient identity}
\begin{equation}\label{charge-gradient}
\nabla_k \q = 2 \v_k,
\end{equation}
the \emph{charge conservation law}
\begin{equation}\label{charge-current}
\nabla_k \j^k + \Im (z^2) \q= - \Im( \overline{u} F )
\end{equation}
the identity
\begin{equation}\label{bochner-identity}
\nabla_k \v^k = \frac{1}{2} \Delta_M \q = \e - \Re (z^2) \q - \Re( \overline{u} F )
\end{equation}
and the \emph{conservation of energy-momentum}
\begin{equation}\label{energy-conservation}
\nabla_k \Q^{m k} +\Re \left((\overline z^2-|z|^2) {\nabla^m u}\,\overline u \right)= - 
\Re( \overline{\nabla^m u} F).
\end{equation}
The energy density $\e$ does not directly obey a useful equation of the above type (the 
expression $\nabla_\alpha \q$ contains expressions involving other double derivatives of $u$ than
the Laplacian $\Delta_M u$).  However, this energy density is clearly related to the other tensor fields, for instance
we have the Cauchy-Schwarz inequalities $|\j|_g, |\v|_g \leq \q^{1/2} \e^{1/2}$ and also 
that $\Q = O( \e ) + O( \lambda \q )$.  

The conservation laws \eqref{charge-current}, \eqref{bochner-identity} and \eqref{energy-conservation}
are directly related to the conservation law of the energy-momentum tensor 
$$
\Q^{\a\b} := \Re( \overline{\nabla^\a u} \nabla^\b u )
- \frac{1}{2} {\bf g}^{\a\b} (\overline{\nabla^\mu u}\,\nabla_\mu u)
$$
associated with the wave equation $\square_{\bf g} u=0$ with the metric 
${\bf g}_{\a\b} dx^\a dx^\b=-dt^2 + g_{ij} dx^i dx^j$. Conservation of energy-momentum 
is expressed in the form 
$$
\nabla_m \Q^{km}=0.
$$
The relationship between conservation laws
for the Helmholtz and wave equations becomes apparent under the formal rule 
$$
\frac {d}{dt}=iz
$$ 

\subsection{Proof of Lemma \ref{charge-est}}\label{charge-proof}

As a quick application of the charge conservation law \eqref{charge-current} to the
resolvent equation \eqref{resolvent}, we now prove Lemma \ref{charge-est}.

Let the notation and assumptions be as in that lemma.  Then $u \in L^2(M)$.  From the resolvent equation \eqref{resolvent} (and the boundedness of $V$) we deduce that $\Delta_M u \in L^2(M)$ also. From elliptic regularity we conclude the qualitative fact that $u \in H^2(M)$.  In particular from Cauchy-Schwarz we see that the charge current $\j^\alpha$ is absolutely integrable.  Also if we set $F := f - Vu$
then $F \in L^2(M)$ and so the charge source term $\Im( \overline{u} F )$ is also absolutely integrable.  We may  thus integrate  \eqref{charge-current} (by inserting a large cutoff function $\chi_R$, adapted 
to the ball of radius $R$, and integrating by parts) to obtain the identity
$$ \pm i \eps \int_M \chi_R |u|^2\ dg+ \int_M \chi_{R}
 \Im( \overline{u} F )\ dg = \int_{M} \nabla \chi_R \Im (\nabla u \overline u).$$
Expanding out the definition of $F$ and taking into account that $V$ is real, 
we thus obtain
\begin{equation}\label{charger}
 \pm i \eps \int_M \chi_R |u|^2\ dg = - \int_M \chi_{R} \Im( \overline{u} f) \ dg
+\int_{M} \nabla \chi_R \Im (\nabla u \overline u).
\end{equation}
The first claim now follows after letting $R\to \infty$, applying 
the Cauchy-Schwarz inequality and using the assumption 
on $V$.

Similarly, integrating by parts with the function $\eta_R$, 
where $\eta_R$ is a cut-off function adapted to the shell of radius $R$,
we obtain
$$ \pm i \eps \int_M \eta^2_R |u|^2\ dg = - \int_M \eta^2_{R} \Im( \overline{u} f) \ dg
+2 \int_{M} \eta_R  \nabla \eta_R  \Im (\nabla u \overline u).$$
Therefore,
$$
\eps^2 \int_M \eta^2_R |u|^2\ dg \le C \int_{R\le\langle x\rangle\le 4R} 
(|f|^2 + R^{-2} |\nabla u|_g^2)\ dg. 
$$
The estimate \eqref{eq:eps-charge-weight} now follows 
after multiplying the above   inequality by $R^{2\alpha}$ and summing over all dyadic 
$R\ge R_0$.  This concludes the proof of Lemma \ref{charge-est}.

\subsection{The abc method}

Now we consider how to exploit the other conservation laws \eqref{charge-gradient}, \eqref{bochner-identity},
\eqref{energy-conservation}.  The \emph{abc method} of Friedrichs is a general way to exploit these identities.  
It proceeds by introducing arbitrary tensor fields $a^k$, $b$, $c^k$, and applying Stokes' theorem to 
the (momentum) vector field
\begin{equation}\label{x-form}
 P^k = a^k \q + b \v^k + c_m \Q^{km}
\end{equation}
to obtain (for a compactly supported cutoff function $\chi$)
\begin{equation}\label{stokes}
 \int_M (\div P) \chi\ dg = - \int_M P^k \partial_k \chi\ dg,
\end{equation}
and in particular (if $P$ has sufficient decay)
\begin{equation}\label{stokes-free}
 \int_M \div P \ dg = 0.
\end{equation}
From \eqref{charge-gradient}, \eqref{bochner-identity}, \eqref{energy-conservation} we observe that
the divergence $\div P = \nabla_k P^k$ can be computed explicitly as
\begin{equation}\label{divx-form}
\begin{split}
\div P &=
\nabla_k a^k \q + (2 a^k + \nabla^k b) \v_k
+ (b - \frac{1}{2} \nabla_k c^k) (\e-|z|^2 \q) + (\nabla_k c_m) 
\Re( \overline{\nabla^k u} \nabla^m u)\\
&\quad + b(|z|^2-\la) \q+c_k\Re \left((z^2-|z|^2) \overline{\nabla^k u}\, u \right)- 
b \Re(\overline{u} F) - c_k \Re( \nabla^k u \overline{F}).
\end{split}
\end{equation}
One then hopes to select $a^k$, $b$, $c^k$, $\chi$ so that many of the terms in \eqref{stokes}
or \eqref{stokes-free} carries a useful sign, so that one obtains a non-trivial estimate on $u$.

In practice, the full generality of abc method is difficult to use, as there are too many possibilities
for $a^k$, $b$, $c^k$, $\chi$ which are available, and too many ways in which one could hope to exploit 
positivity.  Instead, one typically specializes to a sub-case of the Friedrichs method which has fewer parameters.
For instance, if one sets $a^\alpha = c^\alpha = 0$ and $b = 1$ one obtains the 
\emph{Lagrangean
 identity}\footnote{The energy-momentum tensor $\Q^{mk}$ is associated with the Lagrangean 
 ${\mathcal L}=\int_M (|\nabla u|_g^2-|z|^2 |u|^2)$.}
\begin{equation}\label{energy-identity}
\int_M \e \chi = \int_M \lambda \q \chi - \int_M \v^k \nabla_k \chi + \int_M \Re( \overline{u} F ) \chi.
\end{equation}
A further integration by parts using \eqref{charge-gradient} yields
\begin{equation}\label{energy-identity-2}
\int_M \e \chi = \int_M \lambda \q \chi + \frac{1}{2} \int_M \q \Delta_M \chi + \int_M \Re( \overline{u} F ) \chi
\end{equation}
(or alternatively one can apply \eqref{stokes-free} with $a^k = -\frac{1}{2} \nabla^k \chi$, 
$b = \chi$, $c^\alpha = 0$).  We exploit this  identity in Section \ref{energy-sec}.

\subsection{Pohozaev-Morawetz type identities}

Another way to reduce the parameters is to introduce the ansatz
$$ c^k := \nabla^k W; \quad b = \frac{1}{2} \Delta_M W$$
for some scalar real-valued function $W$.  Then the formula \eqref{divx-form} simplifies (because
the $(\e-|z|^2 \q)$ term disappears) to yield
\begin{equation}\label{divx-form-2}
\begin{split}
\div P &=
\nabla_k a^k \q + (2 a^k + \frac{1}{4} \nabla^k \Delta_M W) \v_k
+ \Hess_{mk}(W)
\Re( \overline{\nabla^m u} \nabla^k u)\\
&\quad  + \frac 12 \Delta W(|z|^2-\la) \q+\nabla_kW\Re \left((z^2-|z|^2) \overline{\nabla^k u}\, u \right) \\
&\quad - \frac{1}{2} (\Delta_M W) \Re(\overline{u} F) - (\nabla_k W) \Re(\nabla^k u  \overline{F})
\end{split}
\end{equation}
where $\Hess_{mk}(W) :=\nabla^2_{mk}  W$ is the Hessian of $W$.
If one then sets\footnote{To continue our analogy with the wave equation, 
the above choice of $a^k, b, c^k$ is somewhat reminiscent of the construction
of a modified momentum, associated with a conformal Killing vectorfield, for solutions of the 
wave equation.} $a^k := -\frac{1}{4} \nabla^k \Delta_M W$ to cancel the $\v_k$ term,
and introducing the modified inhomogeneous term 
$$
G:= F+\la - |z|^2
$$
one obtains the \emph{Pohozaev-Morawetz identity}\footnote{Particular cases of this identity have been
used by Pohozaev to prove non-existence of solutions of certain nonlinear elliptic equations, \cite{pohozaev}, and by Morawetz in the study of long time behavior of solutions of a wave equation,
\cite{morawetz}.}
\begin{equation}\label{Morawetz-identity}
\begin{split}
&\int_M \Hess_{mk}(W) (\nabla^m u) (\nabla^k u) - \frac{1}{4}
(\Delta_M^2 W) \q + \frac{1}{2} (\Delta_M W) \Re(\overline{u} G) - 
(\nabla^k W) (\Re( \overline{\nabla^k u} G))
\chi\ dg \\
&= \int_M (- \frac{1}{4} (\nabla^k \Delta_M W) \q + \frac{1}{2} (\Delta_M W) \v^k + 
(\nabla_m W) \Q^{mk}) \nabla_k \chi\ dg.
\end{split}
\end{equation}
This identity is particularly\footnote{We should note that this identity is most applicable 
when $\la=|z|^2$, i.e. $\eps=0$, in which case it leads to interesting statements even 
for nonlinear $F=F(u)$.} useful when $W$ is geodesically convex (i.e. the Hessian $\Hess(W)$ is positive
definite) and the quantity $\Delta_M^2 W$ is non-positive, as this implies that the first two terms
on the left-hand side of \eqref{Morawetz-identity} become non-negative.  We will also re-interpret this
identity in terms of the positive commutator method, and give some consequences of this identity,
in Section \ref{pos-sec}.

To account for the the difference between $\la$ and $|z|^2$ we can modify the momentum $P$
as follows
$$
P^k=- \frac{1}{4} (\nabla^k \Delta_M W) \q + \frac{1}{2} (\Delta_M W) \v^k + 
(\nabla_m W) \Q^{mk}+  \Im z\, \v^k - \Re z\, \j^k.
$$
Then 
\begin{align*}
\div P &=
\Hess_{mk}(W)
\Re( \overline{\nabla^m u} \nabla^k u)-\frac 14\Delta_M^2W \q+ \frac 12 \Delta_M W(|z|^2-\la) \q
 \\ &+\Im z (\e-\la \q)+\Re z \Im (z^2) \q+
2 \Im z\, \nabla_kW\Re \left(iz\, \overline{\nabla^k u}\, u \right) \\ &- 
\frac{1}{2} (\Delta_M W) \Re(\overline{u} F) - (\nabla_k W) \Re(\nabla^k u  \overline{F})-
\Im (z u \overline F) 
\end{align*}
After some calculations we obtain
\begin{align*}
\div P &=
\Hess_{mk}(W)
\Re( \overline{\nabla^m u} \nabla^k u)-\frac 14\Delta_M^2 W \q+ \frac 12 \Delta_M W(|z|^2-\la) \q\\
&\quad +\Im z \left(\e-|\nabla^kW\nabla_k u|^2+|\nabla_k W\nabla^k u+iz u|^2\right)\\ &- 
\frac{1}{2} (\Delta W) \Re(\overline{u} F) - (\nabla_k W) \Re(\nabla^k u  \overline{F})-
\Im (z u \overline F), 
\end{align*}
which after integration over $M$ becomes
\begin{equation}\label{eq:radiat}
\begin{split}
\int_M&\left (
\Hess_{mk}(W)
\Re( \overline{\nabla^m u} \nabla^k u)-\frac 14\Delta^2_MW \q+ \frac 12 \Delta_MW(|z|^2-\la) \q
+\Im z \left(\e-|\nabla^kW\nabla_k u|^2\right)\right.\\&\left.+\Im z |\nabla_k W\nabla^k u+iz u|^2- 
\frac{1}{2} (\Delta W) \Re(\overline{u} F) - (\nabla_k W) \Re(\nabla^k u  \overline{F})-
\Im (z u \overline F)\right)\, \chi\\ &= -\int_M
\left(- \frac{1}{4} (\nabla^k \Delta_M W) \q + \frac{1}{2} (\Delta_M W) \v^k + 
(\nabla_m W) \Q^{mk}+  \Im z\, \v^k - \Re z\, \j^k\right) \nabla_k\chi
\end{split}
\end{equation}
This identity is useful if the Hessian of $W$ is positive definite, $\Delta_M^2 W\le 0$, 
$\Delta_M W\ge 0$ and $|\nabla W|\le 1$. For instance in Euclidean space $\R^n$ 
with the standard metric,
the choice of $W=|x|-(1+|x|)^{1-\de}$ satisfies all the 
above requirements and effectively leads to the proof of the limiting absorption principle.
We will use a version of \eqref{eq:radiat} in Section  \ref{pos-sec} and later in the derivation 
of the Sommerfeld radiation condition.

\subsection{Carleman type identities}

Returning to \eqref{divx-form-2},
the choice $a^k := \frac{1}{4} \nabla^k \Delta_M W$ is not the only possible choice for $a^k$.
If we set $W := e^{2w}$ in the above ansatz, and then set
$$ a^k := \frac{1}{2} (\nabla^k w) \Delta_M e^{2w} - \frac{1}{2} (\nabla^k \Delta_M w) e^{2w},$$
then after some computation using \eqref{divx-form-2} we obtain\footnote{Note that no curvature terms appear here, because
the Levi-Civita connection is torsion free, and at no stage do we need to commute the Laplacian $\Delta_M$ with
a derivative.}
\begin{align*}
\div P &= ( (\Delta_M w)^2  + 4 |\nabla w|_g^2 \Delta_M w + 4 \Hess_{mk}(w) \nabla^m w \nabla^k w
+ 4 |\nabla w|_g^4 - \frac{1}{2} \Delta_M^2 w ) e^{2w} \q \\
&\quad + (4 (\nabla_k w) \Delta_M w + 4 \Hess_{mk}(w) \nabla^m w + 8 |\nabla w|_g^2 \nabla_k w)
e^{2w} \v^k \\
&\quad + (2 \Hess_{mk}(w) + 4 \nabla_m w \nabla_k w) e^{2w} \nabla_m u \nabla_k u \\
&\quad - \frac{1}{2} (\Delta e^{2w}) \Re(\overline{u} G) - (\nabla_k e^{2w}) \Re( \overline{\nabla^k u} G)
\end{align*}
which after completion of the square becomes 
\begin{align*}
\div P &= e^{2w} | 2 \nabla_k w \nabla^k u + (2 |\nabla w|_g^2 + \Delta_M w) u |^2 
+ 2 \Hess_{mk}(w) \nabla^m(e^w u) \nabla^k(e^w u)\\ 
&\quad + 2 \Hess_{mk}(w) (\nabla^m w) (\nabla^k w) e^{2w} |u|^2 
- \frac{1}{2} (\Delta^2_M w) e^{2w} |u|^2\\
&\quad - \frac{1}{2} (\Delta e^{2w}) \Re(\overline{u} G) - (\nabla_k e^{2w}) \Re( \overline{\nabla^k u} G).
\end{align*}
This identity may seem like an algebraic miracle, but we can interpret it as another positive commutator
estimate in the more general setting of Carleman inequalities, see Section \ref{carleman-sec}.  If $u$ is
compactly supported, then one can integrate this identity on $M$ (as in \eqref{stokes-free}) to obtain
the \emph{Carleman identity}\footnote{It may seem remarkable at first that the energy $\lambda$ makes
no appearance in this identity.  But note that if one increments $G$ by $\lambda u$ then the resulting changes
in the two terms involving $F$ cancel each other out, thanks to \eqref{charge-gradient} and integration by parts.
In practice, the requirement that $u$ be compactly supported means that we will have to apply a cutoff function to
truncate $u$, and this will cause $\lambda$ to enter the identity in a non-trivial manner.}
\begin{align*}
\int_M | 2 \nabla_k w \nabla^k u + (2 |\nabla w|_g^2 + \Delta_M w) u |^2 e^{2w}\ dg \quad &\\
+ 2 \int_M \Hess_{mk}(w) \nabla^m(e^w u) \nabla^k(e^w u)\ dg \quad & \\
+ 2 \int_M \Hess_{mk}(w) (\nabla^m w) (\nabla^k w) |u|^2 e^{2w}\ dg
&= \int_M \frac{1}{2} (\Delta^2_M w) |u|^2 e^{2w}\ dg \\
+ \int_M \frac{1}{2} (\Delta_M e^{2w}) \Re(\overline{u} G) &+ 
(\nabla_k e^{2w}) \Re( \overline{\nabla^k u} G)\ dg.
\end{align*}
If we write $U := (2 \nabla_k w \nabla^k u + (2 |\nabla w|_g^2 + \Delta_M w) u) e^{w}$
and note that
$$ \frac{1}{2} \Delta_M e^{2w} u + (\nabla_k e^{2w}) \nabla^k u = e^w U$$
we can rewrite the previous expression as
\begin{align*}
\| U \|_{L^2(M)}^2 + 2 \int_M \Hess_{mk}(w) \nabla^m(e^w u) \nabla^k(e^w u)\ dg \quad & \\
+ 2 \int_M \Hess_{mk}(w) (\partial^m w) (\nabla^k w) |u|^2 e^{2w}\ dg&\\
= \int_M \frac{1}{2} (\Delta^2_M w) |u|^2 e^{2w}\ dg &+ \langle e^w F, U \rangle_{L^2(M)}
\end{align*}
and hence by the Cauchy-Schwarz inequality $\langle e^w F, U \rangle_{L^2(M)}
\leq \frac{1}{4} \| e^w F \|_{L^2(M)} + \| U \|_{L^2(M)}$, we conclude
\begin{equation}\label{carleman-identity}
\begin{split}
\int_M \Hess_{mk}(w) \nabla^m(e^w u) \nabla^k(e^w u)\ dg \quad & \\
+ \int_M \Hess_{mk}(w) (\nabla^m w) (\nabla^k w) |u|^2 e^{2w}\ dg
&\leq \int_M \frac{1}{4} (\Delta^2_M w) |u|^2 e^{2w}\ dg + \frac{1}{8} \| e^w F \|_{L^2(M)}^2.
\end{split}
\end{equation}
This identity
is useful when $w$ is geodesically convex (so that both terms on the left-hand side is positive)
and large (so that one can absorb some of the right-hand side terms into the left-hand side).  We shall
exploit it in Section \ref{carleman-sec}.

\section{Elliptic regularity}\label{energy-sec}

In this section we prove the various claims in Lemma \ref{elliptic-regularity}.  One of the key tools here will be the energy identities \eqref{energy-identity}, \eqref{energy-identity-2}.

\subsection{Proof of \eqref{elliptic-0}}

Let the notation and assumptions be as in Lemma \ref{elliptic-regularity}.  We begin with the proof of the elliptic regularity estimate \eqref{elliptic-0}.  By absorbing $z^2 u$ into the $f$ term we may assume that $\lambda = 0$; similarly, by absorbing $Vu$ into the $f$ term we may assume that $V \equiv 0$.  From classical elliptic regularity we know that $u$ is locally
in $H^2$. Our task is thus to show that
$$ \int_M \langle x \rangle^{2m} (|\nabla u|_g^2 + |\nabla^2 u|_g^2)\ dg
\leq C_m \int_M \langle x \rangle^{2m} (|\Delta_M u|^2 + |u|^2)\ dg.$$
Let us first verify this claim for the first derivatives $\nabla u$.  Let $R \gg 1$ be a large number,
and let $\varphi_R$ to the region $\langle x \rangle \leq R$ which equals one for $\langle x \rangle \leq R/2$.
From \eqref{energy-identity-2} we have
$$
\int_M |\nabla u|_g^2 \langle x \rangle^{2m}
\varphi_R = \frac{1}{2} \int_M |u|^2 \Delta_M (\langle x \rangle^{2m} \varphi_R) + 
\int_M \Re( \overline{u} \Delta_M u) \langle x \rangle^{2m} \varphi_R;$$
applying Cauchy-Schwarz (and the very crude estimate $\Delta_M (\langle x \rangle^{2m} \varphi_R) = O_m(\langle x \rangle^{2m})$) and then letting $R \to \infty$ we obtain the estimate for first derivatives.  
Thus it will suffice to show that
$$  \int_M \langle x \rangle^{2m} |\nabla^2 u|_g^2\ dg
\leq C_m \int_M \langle x \rangle^{2m} |\Delta_M u|^2\ dg + C_m \| u \|_{H^{1,m}(M)}^2.$$
The standard Bochner integration by parts argument to give this estimate would require boundedness
of the Riemann curvature tensor.  Let 
us present a slight variant of that argument that only requires boundedness on first derivatives of the metric.  
To illustrate the method let us first work locally, using a compactly supported bump function $\chi$.  We claim the estimate
$$ \| |\nabla^2 u| \chi \|_{L^2(M)} \leq C_m( \| \Delta_M u \|_{H^{0,m}(M)} + \| u \|_{H^{1,m}(M)} ).$$
To see this, we integrate by parts, computing
\begin{align*}
\| |\nabla^2 u| \chi \|_{L^2(M)}^2
&= \langle \nabla_k \nabla^\ell u, \chi^2 \nabla_\ell \nabla^k u \rangle_{L^2(M)}\\
&= - \langle \nabla^\ell u, \chi^2 \nabla_k \nabla_\ell \nabla^k u \rangle_{L^2(M)}
- 2 \langle \nabla^\ell u, \chi (\nabla_k \chi) \nabla_\ell \nabla^k u \rangle_{L^2(M)} \\
&= - \langle \nabla^\ell u, \chi^2 \nabla_\ell \Delta_M u \rangle_{L^2(M)}
- \langle  \nabla^\ell u, \chi^2 [\nabla_k, \nabla_\ell] \nabla^k u \rangle_{L^2(M)}\\
&+ O\left(\| |\nabla \chi|_g |\nabla u|_g \|_{L^2(M)} \| |\nabla^2 u|_g \chi \|_{L^2(M)}\right)  \\
&= O_m(\| \Delta_M u \|_{H^{0,m}(M)}^2) + O( \| u\|_{H^{1,m}(M)} \| \Delta_M u \|_{H^{0,m}(M)} )\\ &
+ O( \| |\nabla^2 u|_g \chi \|_{L^2(M)} \| u \|_{H^{1,m}(M)} )
-
\langle  \nabla^\ell u, \chi^2 [\nabla_k, \nabla_\ell] \nabla^k u \rangle_{L^2(M)}.
\end{align*}
Thus to verify the claim, it will suffice to show that
$$ 
|\langle  \nabla_\ell u, \chi^2 [\nabla_k, \nabla_\ell] \nabla_k u \rangle_{L^2(M)}|
\leq C_m( \| u \|_{H^{1,m}(M)}^2 + \||\nabla^2 u| \chi\|_{L^2(M)} \| u \|_{H^{1,m}(M)} ).$$
If we exploited the boundedness of the Riemann curvature tensor at this point (recalling that $g$ was assumed
to be smooth), then we would be done (and we would not need the second term on the right-hand side).  However, 
one can instead work in local co-ordinates, writing\footnote{Schematically, what we are doing is observing that
any integral expression of the form $\nabla u \nabla^2 g \nabla u$ can be rewritten using integration by parts as $\nabla u \nabla g \nabla^2 u$ plus lower order terms.} 
$$ ([\nabla_k, \nabla_\ell] X)^\gamma = (\partial_k( \Gamma^{s}_{\ell p} X^p )
+ \Gamma^s_{k p} (\nabla_\ell X)^p)
- (\partial_\ell (\Gamma^s_{k p} X^p ) + \Gamma^s_{\ell p} (\nabla_k X)^p )$$
for any vector field $X$, where $\Gamma$ are the Christoffel symbols.  Applying this identity and integrating
by parts as necessary to prevent any derivatives falling on the Christoffel symbols, 
we obtain the above claim, where we only needed the boundedness of the Christoffel symbols $\Gamma$ (which
are essentially one derivative of the metric $g$, in contrast to the Riemann curvature which is two derivatives).

One can perform a similar argument with $\chi$ replaced by $\langle x \rangle^m \varphi_R (1-\chi)$, where $1-\chi$
localizes to the asymptotic region $r \gg R_0$; the point being
that the conditions \eqref{metric-decay} ensures the appropriate boundedness of the Christoffel symbols in this region.  
Indeed, there is some additional decay and vanishing properties of these symbols arising from the normal form
co-ordinate structure, which we shall simply discard; similarly we shall discard any gains of $\langle x \rangle$
when a derivative hits the $\langle x\rangle^{2m}$ factor.  
We omit the standard details.  This concludes the proof of \eqref{elliptic-0}.

\subsection{Proof of \eqref{elliptic-local}}

We now localise \eqref{elliptic-0} to prove \eqref{elliptic-local}.  Let $\chi$ be a smooth cutoff to the region $\{ \langle x \rangle \geq R/2 \}$ which equals $1$ when $\langle x \rangle \geq R$.  Applying \eqref{elliptic-0} to the function $\chi^3 u$ we have
$$ \| \chi^3 u \|_{H^{2,m}(M)} \leq C(m,\chi) ( \| \chi^3 f \|_{H^{0,m}(M)}
+ \| \chi^2 |\nabla u| \|_{H^{0,m}(M)} + (1+\lambda) \| \chi^3 u \|_{H^{0,m}(M)} )$$
which (by the product rule) implies
$$ \| \chi^3 |\nabla^2 u| \|_{H^{0,m}(M)} \leq
C(m,\chi) ( \| \chi^3 f \|_{H^{0,m}(M)}
+ \| \chi^2 |\nabla u| \|_{H^{0,m}(M)} + (1+\lambda) \| \chi^3 u \|_{H^{0,m}(M)} ).$$
By using the interpolating inequality
$$ \| \chi^2 |\nabla u| \|_{H^{0,m}(M)}^2  \leq C(m,\chi) \| \chi^3 |\nabla^2 u| \|_{H^{0,m}(M)} 
\| \chi u \|_{H^{0,m}(M)}$$
(which is easily proven by integration by parts and Cauchy-Schwarz) we conclude that
$$ \| \chi^3 |\nabla^2 u| \|_{H^{0,m}(M)} \leq
C(\chi) ( \| f \|_{H^{0,m}(M)}
+ (1+\lambda) \| \chi u \|_{H^{0,m}(M)} ).$$
The claim \eqref{elliptic-local} follows.

\subsection{Proof of \eqref{energy-spiral}, \eqref{energy-spiral'}, \eqref{eq:la-nab}}

If we apply the identity \eqref{energy-identity-2} with $\chi$ equal to a smooth cutoff
function adapted to the annulus $\frac{1}{2} R \leq \langle x \rangle \leq 3R$ which equals one when
$R \leq \langle x \rangle \leq 2R$, we obtain the estimate
$$
\int_{R \leq \langle x \rangle \leq 2R} |\nabla u|_g^2\ dg
\leq C \int_{R/2 \leq \langle x \rangle \leq 3R} (\lambda + R^{-2}) |u|^2 + |u|^2 |V| + |u| |f|\ dg,$$
and hence by Cauchy-Schwarz (and \eqref{V-short}) we establish \eqref{energy-spiral} and \eqref{energy-spiral'}.
Similarly, applying \eqref{energy-identity-2} with $\chi=\langle x\rangle^{-1-2\si}$ we obtain \eqref{eq:la-nab}.

\subsection{Proof of \eqref{eq:la-eps-nab}, \eqref{eq:la-eps-loc}, \eqref{eq:la-loc}}

Interpolating \eqref{eq:la-nab} $\|u\|_{H^{0,-3/2-\sigma}}$ we have 
$$
\lambda^{\frac 14}  \|u\|_{H^{0,-1-\sigma}} \le C(A)\left (\|\nabla u\|_{H^{0,-1/2-\sigma}}
+ \|u\|_{H^{0,-3/2-\sigma}} + \|f\|_{H^{0,1/2+\sigma}}\right).
$$
Combining \eqref{eq:la-nab} with the charge estimate \eqref{eq:eps-charge} we have that 
$$
\eps\lambda^{\frac 12}  \|u\|^2_{L^2(M)} \le C(A) \|f\|_{H^{0,1/2-\sigma}}
\left (\|\nabla u\|_{H^{0,-1/2-\sigma}}
+  \|u\|_{H^{0,-3/2-\sigma}} + \|f\|_{H^{0,1/2+\sigma}}\right).
$$
At this point we require the condition $\eps\le C \la$ for some constant $C>0$. Then
$$
\eps^{\frac 32}  \|u\|^2_{L^2(M)} \le C(A) \|f\|_{H^{0,1/2+\sigma}}
\left (\|\nabla u\|_{H^{0,-1/2-\sigma}}
+  \|u\|_{H^{0,-3/2-\sigma}} + \|f\|_{H^{0,1/2+\sigma}}\right).
$$
The identity  \eqref{energy-identity-2} (with $\chi=1$) implies that
$$
\|\nabla u \|_{L^2(M)}\le \lambda^{\frac 12}  \|u\|_{L^2(M)} + 
C(A)\left(\la^{-\frac 14} \|f\|_{H^{0,1/2+\sigma}} +  \la^{\frac 14} \|u\|_{H^{0,-1/2-\sigma}}
+ \|u\|_{H^{0,-1-\sigma}}\right).
$$
Combining the above four inequalities and using that $\eps \le C\la$ we conclude \eqref{eq:la-eps-nab}.

By inserting smooth cutoffs localised to the region $\{ x: \langle x \rangle \gtrsim R \}$ to the above argument we also obtain \eqref{eq:la-eps-loc} and \eqref{eq:la-loc}; we omit the standard details.

The proof of Lemma \ref{elliptic-regularity} is now complete.

\section{The positive commutator method}\label{pos-sec}

We continue our study of the Helmholtz equation \eqref{helmholtz}, and recall the well-known \emph{positive commutator method} to analyze this equation.  We let $S$ be an arbitrary pseudo-differential operator, and consider the 
expression $\langle i[-\Delta_M,S] u, u \rangle_{L^2(M)}$, where $i[S,B] := i(SB - BS)$ is the usual 
Lie commutator and $\langle u, v \rangle := \Re \int_M \overline{u} v\ dg$ is
the real inner product.  Then from the self-adjointness of $-\Delta_M - \lambda$ and
\eqref{helmholtz} we have the \emph{commutator identity}\footnote{In applications, $u$ will be in $H^2(M)$, $F$ will
be in $L^2(M)$, and $S$ will be first order, and so there is no difficulty justifying the manipulations below.}
\begin{align*}
\langle i[-\Delta_M,S] u, u \rangle_{L^2(M)} &=
\langle i[-\Delta_M-\lambda,S] u, u \rangle_{L^2(M)}  \\
&= \langle i Su, (-\Delta_M - \lambda) u \rangle_{L^2(M)}
- \langle i (-\Delta_M - \lambda) u, S^* u \rangle_{L^2(M)} \\
&= \langle i Su, F\pm i\eps u\rangle_{L^2(M)} - \langle F\pm i\eps u, (iS)^* u \rangle_{L^2(M)}
\end{align*}
and in particular by Cauchy-Schwarz
\begin{equation}\label{positive-commutator}
\begin{split}
\langle i[-\Delta_M,S] u, u \rangle_{L^2(M)} &\leq \left (\| S u \|_{H^{0,-1/2-\sigma}(M)}+
\| S^* u \|_{H^{0,-1/2-\sigma}(M)}\right)\| F \|_{H^{0,1/2+\sigma}(M)}\\ &
+2 \eps \| S u \|_{L^2(M)} \| u \|_{L^2(M)}.
\end{split}
\end{equation}
This identity is useful when $S$
is chosen so that $i[-\Delta_M,S]$ is a positive operator (at least to top order), which is why the use of this identity is known as the positive commutator method.

There are a number of ways to generate an operator $S$ whose commutator $i[-\Delta_M,S]$ is positive.
One is to make $A$ itself equal to a commutator $S := i[-\Delta_M,W]$, where $W$ is a real-valued scalar weight function,
interpreted as a pointwise multiplication operator $u \mapsto Wu$.  Then $S$ is the self-adjoint first-order
differential operator
$$ S u := - 2i (\nabla^k W) \nabla_k u - i (\Delta_M W) u$$
and $i[-\Delta_M, S]$ can be computed to be the second-order operator\footnote{This can either be seen by
expanding everything out using the Leibnitz rule, and then exploiting the symmetry properties of the Riemann curvature 
tensor, or alternatively noting that both sides of this identity are second-order self-adjoint operators with
real coefficients and agree both at top order and at the constant term, and thus must be identical.
See e.g. \cite{htw:moraw}.}
$$ i[-\Delta_M,S]u = - 4\nabla_k (\Hess^{km}(W) \nabla_m u) - (\Delta_M^2 W) u$$
The positive commutator estimate \eqref{positive-commutator} then becomes
\begin{equation}\label{eq:pos-com}
\begin{split}
& \int_M 4 \Hess^{km}(W) \Re( \overline{\nabla_k u} \nabla_m u)
- (\Delta_M^2 W) |u|^2\ dg \\
&\quad \leq 2 \| 2 (\nabla^k W) \nabla_k u + (\Delta_M W) u \|_{H^{0,-1/2-\sigma}(M)}
\| F\|_{H^{0,1/2+\sigma}(M)}\\ &\quad+
2 \eps\,\| 2 (\nabla^k W) \nabla_k u + (\Delta_M W) u \|_{L^2(M)}  \| u \|_{L^2(M)},
\end{split}
\end{equation}
which should be compared with the $\chi=1$ case of \eqref{Morawetz-identity}.  Indeed \eqref{Morawetz-identity}
can be interpreted as the positive commutator identity applied to the operator $S := \chi i[-\Delta_M,W]$; we omit
the details of this calculation.

\subsection{A cheap proof of limiting absorption in the free case}\label{cheap-sec}

Using this inequality we can easily prove the limiting absorption principle\footnote{This 
``cheap'' proof of the limiting absorption principle applies to the region where $\eps\le C\la$
for some positive constant $C$. Analysis of the remaining region requires the use of 
an additional conservation law \eqref{eq:radiat}. This issue will be addressed in Section
\ref{sec:add-motion} in a more general context.} (or at least the main estimates
\eqref{lap-h0}, \eqref{lap-2}) for the operator $H_0=-\Delta$ in Euclidean space. 
We choose the function 
$$
W(x)=|x|-(1+|x|)^{1-2\si}.
$$
Direct calculation shows that for dimensions $n\ge 3$
$$
|x|\,|\Delta W|+|\nabla W|\le C,\quad -\Delta^2 W\ge c (1+|x|)^{-3-2\si}, \quad
\Hess_{km}(W)\ge (1+|x|)^{-1-2\si} \delta_{km}.
$$
We then obtain from \eqref{eq:pos-com} 
\begin{align*}
 \int_{\R^n} \left(\langle x\rangle^{-1-2\si} |\nabla u|^2 + \langle x\rangle^{-3-2\si} |u|^2\right)
&\leq \left (\| \nabla u \|_{H^{0,-1/2-\sigma}} + \| |x|^{-1} u \|_{H^{0,-1/2-\sigma}}\right)
\| f\|_{H^{0,1/2+\sigma}}\\ &\quad+
2 \eps\,\left(\| \nabla u\|_{L^2}+  \| |x|^{-1}u \|_{L^2}\right)  \| u \|_{L^2},
\end{align*}
Now combining Hardy type inequalities
$$
\| |x|^{-1}u \|_{L^2}\le C \| \nabla u\|_{L^2},\qquad \| |x|^{-1} u \|_{H^{0,-1/2-\sigma}}\le
C\left (\| \nabla u \|_{H^{0,-1/2-\sigma}}+ \|  u \|_{H^{0,-3/2-\sigma}}\right)
$$
with \eqref{eq:la-eps-nab} and \eqref{eq:la-nab} we obtain
$$
\| \nabla u \|_{H^{0,-1/2-\sigma}}+\la^{\frac 12}\|u\|_{H^{0,-1/2-\sigma}} + 
\|  u \|_{H^{0,-3/2-\sigma}}\le C \| f\|_{H^{0,1/2+\sigma}}
$$
with second derivative estimates following from the elliptic regularity estimate \eqref{elliptic-0}.

\begin{remark}
A proof of the Sommerfeld radiation condition and the estimate \eqref{eq:h0} will require the use
of the identity \eqref{eq:radiat} in place of \eqref{Morawetz-identity}. We will return to this point in 
a more general context when we establish Lemma \ref{lem:Out}.
\end{remark}

\subsection{Proof of Lemma \ref{morawetz-lemma}}\label{pm}

Now we can prove the Pohozaev-Morawetz type estimate in Lemma \ref{morawetz-lemma}.  Let the notation and assumptions be as in that Lemma.

For this argument, the function $W$ is chosen
to be 
$$
W=\chi(\langle x\rangle x) (|x|-|x|^{1-2\si}),
$$
where $\chi(r)$ is a smooth cut-off function vanishing for $r\le r_0$ and equal to 1
for $r\ge 2r_0$ for some $r_0\ge R_0$ to be chosen later. 
On the support of $W$
the metric $g$ has the form 
$$
 g=dr^2 + r^2 (h_{ab}(\omega) + r^{-2\sigma_0} e_{ab}(r,\omega)) d\omega^a d\omega^b
$$
The second fundamental form $\Theta_{ab}$ of the hypersurfaces $r=$const for a metric
in this form is equal to the expression
$$
\Theta_{ab}=r h_{ab} + \frac 12 \pa_r\left[ r^{2-2\sigma_0} e_{ab}(r,\omega)\right]=
rh_{ab} + O(r^{1-2\si_0})
$$
and its mean curvature $\Theta$ (i.e. the trace of $\Theta_{ab}$) has the asymptotic
$$
\theta =  r^{-2} (h(\omega) + r^{-2\sigma_0} e(r,\omega))^{ab} \Theta_{ab}=
 \frac {n-1}r + O(r^{-1-2\si_0}).
$$
The Laplacian $\Delta_M$ can be written in this coordinate system as
\begin{equation}\label{deltam}
\Delta_M=\pa_r^2 + \Theta \,\pa_r + \Delta_{\pa M_r},
\end{equation}
where  $\Delta_{\pa M_r}$ is the Laplace-Beltrami operator of the hypersurface $r=$const.
In particular, $\Delta_M$ applied to the function $W(r)$ can be computed as
$$
\Delta_M W = W_{rr}+ \Theta \,W_r= W_{rr}+ \frac{n-1}r \,W_r + O(r^{-1-2\si_0}) W_r.
$$
Furthermore,
$$
\Delta_M^2 W = (\pa_r^2+\frac{n-1}r\pa_r)^2 W + \Delta_M\left
(O(r^{-1-2\si_0}) W_r\right).
$$
Finally the Hessian of $W$,
$$
\Hess_{rr}(W) = W_{rr}, \quad \Hess_{ar}(W)=0,\qquad \Hess_{ab}(W)=\Theta_{ab} W_r.
$$
From this we easily conclude that 
\begin{align*}
r|\Delta_M W|+|\nabla W|_g&\le C,\\
-\Delta_M^2&\ge C \chi r^{-3-2\si} +\Delta_M\left
(O(r^{-1-2\si_0}) W_r\right)+\zeta O(1),\\
\Hess_{mk}(W)&\ge C r^{-1-2\si} + \zeta O(1).
\end{align*}
Here $\zeta$ is a smooth cut-off function with support in the region $r_0\le \langle x\rangle\le 2r_0$
introduced to account for the derivatives falling on $\chi$ in the expression for $W$. Let 
$M_r$ denote the set $\{x\in M:\,\langle x\rangle\ge r\}$. Then
\begin{align*}
 c \int_{M_{2r_0}} 
 &\left(\langle x\rangle^{-1-2\si} |\nabla u|_g^2 + \langle x\rangle^{-3-2\si} |u|^2\right)
\leq  \int_M \Delta_M\left
(O(r^{-1-2\si_0}) W_r\right) |u|^2\\ &+
\int_{M_{r_0}\setminus M_{2r_0} }\left (|\nabla u|_g^2 +|u|^2\right)+
\left (\| \nabla u \|_{H^{0,-1/2-\sigma}(M_{r_0})} + \| u \|_{H^{0,-3/2-\sigma}(M_{r_0})}\right)
\| f\|_{H^{0,1/2+\sigma}(M_{r_0})}\\ &\quad+
 \eps\,\left(\| \nabla u\|_{L^2(M_{r_0})}+  \| |x|^{-1}u \|_{L^2(M_{r_0})}\right)  \| u \|_{L^2(M_{r_0})},
\end{align*}
Arguing as in the proof of the limiting absorption principle in Euclidean space above we can, 
with the help of \eqref{eq:la-eps-loc} and \eqref{eq:la-loc}, reduce the inequality further, 
\begin{align*}
 c \int_{M_{2r_0}} 
 &\left(\langle x\rangle^{-1-2\si} |\nabla u|_g^2 +\langle x\rangle^{-1-2\si} \la |u|^2 + \langle x\rangle^{-3-2\si} |u|^2\right)
\leq  \int_M \Delta_M\left
(O(r^{-1-2\si_0}) W_r\right) |u|^2\\ &+
\int_{M_{{r_0}/2}\setminus M_{2r_0} }\left (|\nabla u|_g^2 +|u|^2\right)+
\| f\|^2_{H^{0,1/2+\sigma}(M_{r_0/2})}.
\end{align*}
It remains to deal with the term $\int_M \Delta_M\left
(O(r^{-1-2\si_0}) W_r\right) |u|^2$. We avoid applying the Laplacian $\Delta_M$ to the 
term $O(r^{-1-2\si_0}) W_r$, as this would require higher differentiability of the metric
$g$ than required in Theorem \ref{main}. Instead we integrate by parts to obtain
\begin{align*}
\int_M \Delta_M\left
(O(r^{-1-2\si_0}) W_r\right) |u|^2&= -2\int_M \left
(O(r^{-1-2\si_0}) W_r\right) \left(\Re(\Delta_M u \overline u)+ |\nabla u|_g^2\right)\\ &=
-2\int_M \left
(O(r^{-1-2\si_0}) W_r\right) \left(\lambda |u|^2+\Re (F\overline u)+ |\nabla u|_g^2\right)
\end{align*}
It then follows that 
\begin{align*}
 c \int_{M_{2r_0}} 
 &\left(\langle x\rangle^{-1-2\si} |\nabla u|_g^2 +\langle x\rangle^{-1-2\si} \la |u|^2 + \langle x\rangle^{-3-2\si} |u|^2\right)\\
\leq &
\int_{M_{{r_0}/2}\setminus M_{2r_0} }\left (|\nabla u|_g^2 +|u|^2\right)+
\| f\|^2_{H^{0,1/2+\sigma}(M_{r_0/2})}.
\end{align*}
provided that $\si<\si_0$ and $r_0$ is sufficiently large. The desired conclusion now 
is a consequence of  \eqref{energy-spiral'}.  This concludes the proof of Lemma \ref{morawetz-lemma}.

\begin{remark}
From the Hessian bound
$$
\Hess_{ab}(W)=\Theta_{ab} W_r= (r h_{ab}+O(r^{1-2\si_0}))W_r 
$$
we observe that a more precise estimate is available on the angular part 
$|\nabla_\omega u|_g^2=r^2(|\nabla u|_g^2- |\pa_r u|^2)$ of the the gradient of $u$:
\begin{align*}
 c \int_{M_{2r_0}} \langle x\rangle^{-3} |\nabla_\omega u|_g^2 
\leq
\int_{M_{{r_0}/2}\setminus M_{2r_0} }\left (|\nabla u|_g^2 +|u|^2\right)+
\| f\|^2_{H^{0,1/2+\sigma}(M_{r_0/2})}.
\end{align*}
We will revisit this analog of the angular Morawetz estimate for solutions of the time-dependent
wave and Schr\"odinger equations in Lemma \ref{lem:Out}, where we will prove an even stronger
result.
\end{remark}

\begin{remark}
As Proposition \ref{pseudo} shows, it is not possible to remove the error term from Lemma \ref{morawetz-lemma} in general.  However, under the non-trapping assumption it is possible. For examples of such results
see \cite{Vainberg}, \cite{lp}, \cite{Vodev}, \cite{b}, \cite{RT}.
\end{remark}

\section{High energy limiting absorption for non-trapping metrics}\label{sec:high-energy}

We are now ready to prove Theorem \ref{nontrap}.  This case resembles
the local-in-time theory of Craig-Kappeler-Strauss \cite{cks} and Doi \cite{doi}, and indeed our main tool here will be the positive commutator method applied to a certain pseudo-differential operator, exploiting the non-trapping hypothesis to ensure that the symbol of the pseudo-differential operator increases along geodesic flow.
As we shall now be working in the high frequency setting, we will not need to take as much care with lower order terms as in previous sections.  
It will be convenient
to use the \emph{scattering pseudo-differential calculus}, which is an extension of the standard pseudo-differential calculus which keeps track of the decay of the symbol at infinity.  We briefly summarize
the relevant features of this calculus here, referring the reader to \cite{cks} for more complete details. (This material will not be used outside of this section.)

For any $m,l \in \R$, we define a \emph{symbol $s: T^* M \to \C$ of order $(m,l)$} to be any
smooth function obeying the bounds
$$ |\nabla_x^\alpha \nabla_\xi^\beta s(x,\xi)| \leq 
C_{\alpha,\beta} \langle \xi \rangle^{m-|\beta|} \langle x \rangle^{-l-|\alpha|};$$
the function $s(x,\xi) = \langle x \rangle^{-l} \langle \xi \rangle^m$ is a typical example of such a symbol.
Note that we assume that each derivative in $x$ gains a power of $\langle x \rangle$,
in contrast to the standard symbol calculus in which no such gain is assumed.  
We let $S^{m,l}(\Mbar)$ denote the space of such symbols.  Given any such symbol $s \in S^{m,l}(\Mbar)$, 
we can define an associated pseudo-differential operator $S = \Op(s)$ by the usual 
Kohn-Nirenberg quantization formula
$$
\Op(a) u(x) := (2\pi)^{-n} \int e^{i\langle x-y, \xi \rangle} s(x,\xi) u(y)\, dy \, d\xi.
$$
We sometimes denote $s$ by $\sigma(S)$ and refer to it as the \emph{symbol} of $S$.  Heuristically speaking,
we have $S = \sigma(s)(x, \frac{1}{i} \nabla_x)$.  We refer to the class of pseudo-differential operators
of order $(m,l)$ as $\Psisc^{m,l}$.
Also, if $h: \R \to \C$ is any spectral symbol of order $m/2$, the corresponding spectral multiplier $h(H)$ is
a pseudo-differential operator of order $(m, 0)$.  In particular, 
$(1+H)^{m/2}$ has order $(m,0)$, and the Littlewood-Paley type 
operators $P_{lo}$, $P_{med}$, $P_{hi}$ have order $(0,0)$.  We caution however that the Schr\"odinger 
propagators $e^{-itH}$ are not pseudo-differential operators.

The composition of an operator $S = \Op(s)$ of order $(m,l)$ with an operator of $B = \Op(b)$ order $(m',l')$ is an
operator $SB$ of order $(m+m', l+l')$, whose symbol $\sigma(SB)$ is equal to $\sigma(S) \sigma(B)$ plus an error
of order $(m+m'-1,l+l'+1)$; note the additional gain of 1 in the decay index $l$, which is not present in the
classical calculus.  
Similarly, the commutator $i[S,B]$ will be an operator of order $(m+m'-1,l+l'-1)$ with
symbol $\sigma(i[S,B])$ equal to the Poisson bracket
$$ \{ \sigma(S), \sigma(B) \} := \nabla_x \sigma(S) \cdot \nabla_\xi \sigma(B) - \nabla_\xi \sigma(S) \cdot \nabla_x \sigma(B),$$
plus an error of order $(m+m'-2, l+l'+2)$. We shall write the above facts schematically as
\begin{align*}
&\sigma(AB) = \sigma(A) \sigma(B) + S^{m+m'-1,l+l'+1};\\
&\sigma( i[A,B] ) = \{ \sigma(A), \sigma(B) \} + S^{m+m'-2,l+l'+2}
\end{align*}
or equivalently as
\begin{align*} 
&\Op(s) \Op(b) = \Op(sb) +  \Psisc^{m+m'-1,l+l'+1},\\
&i[\Op(s), \Op(b)] = \Op( \{s,b\} ) + \Psisc^{m+m'-2,l+l'+2}.
\end{align*}
In particular, since $H$ has order $(2,0)$ and has principal symbol $\frac{1}{2} |\xi|_{g(x)}^2$
plus lower order terms of order $(1,1)$ and $(0,2)$, we see that if $a \in S^{m,l}$, then we have
$$ i[H, \Op(s)] = \Op( X s ) + \Psisc^{m,l+2},$$
where $X a$ denotes the derivative of $s$ along geodesic flow in the cotangent bundle $T^* M$.  

Associated with the scattering calculus are the weighted Sobolev spaces $H^{m,\ell}(M)$ defined (for instance)
by
$$ \| u \|_{H^{m,\ell}(M)} := \| \langle x \rangle^\ell (1 + H)^{m/2} u \|_{L^2(M)}$$
(many other equivalent expressions for this norm exist, of course); when $\ell=0$ this corresponds to the usual Sobolev space $H^m(M)$ and for $m=0,1,2$ the space $H^{m,\ell}(M)$
coincides with the weighted Sobolev spaces previously defined in \eqref{weight-def}.  
It is easy to verify that a
scattering pseudo-differential operator of order $(m,l)$ maps $H^{m',l'}(M)$ to $H^{m'-m,l'+l}(M)$ for any $m', l'$.

In \cite{cks} (see also \cite{doi}) it was shown that the non-trapping hypothesis on $M$ allows one to
construct a real-valued symbol $s \in S^{1,0}$ (depending on $\varphi$) which was non-decreasing along
geodesic flow, $Xs \geq 0$, and in fact obeyed the more quantitative estimate
$$ Xs(x,\xi) = \varphi(x) |\xi|_g^2 + |b|^2$$
for some symbol $b$ of order $(1,1/2+\sigma)$.  The function $\phi(x)$ belongs to the class 
$S^{0,1+2\si}$ and can be chosen to obey a lower bound
$$
\phi(x)\ge \langle x \rangle^{-1-2\si}.
$$
In Euclidean space, an example of such a symbol $a$ is
$C_\varphi \frac{x}{\langle x \rangle-\alpha \langle x \rangle^{1-2\si}} \cdot \xi$ for some sufficiently large constant $C_\varphi$ and a small constant $\alpha$.  Quantizing
this, we obtain
$$ i[H,S] = \nabla^j \varphi(x) \nabla_j + B^* B + \Psisc^{1,2+2\sigma},$$
where $S := \Op(s)$ is a symbol of order $(1,0)$, and $B := \Op(b)$ is a symbol of order $(1,1/2-\sigma)$.
Applying \eqref{positive-commutator}, discarding the positive term $B^* B$ and using
that $\Psisc^{1,2+2\sigma}$ maps $H^{1,-1/2-\sigma}$ to $H^{0,-3/2-\sigma}$ we obtain
\begin{align*}
 \int_M \varphi |\nabla u|_g^2\ dg&\leq \left (\| S u \|_{H^{0,-1/2-\sigma}(M)}+
\| S^* u \|_{H^{0,-1/2-\sigma}(M)}\right)\| F \|_{H^{0,1/2+\sigma}(M)}\\ &
+\eps \| S u \|_{L^2(M)} \| u \|_{L^2(M)} +\| \nabla u \|_{H^{0,-1/2-\sigma}(M)}
\| u \|_{H^{0,-3/2-\sigma}(M)}
\end{align*}
Since $S$ is of order $(1,0)$, it maps $H^{1,-1/2-\sigma}\to H^{0,-1/2-\sigma}$ 
and $H^1(M)\to L^2(M)$. Recalling that 
$F=f+Vu$ with $V$ satisfying the bound 
$|V(x)|\le \Gamma \langle x \rangle^{-1-2\si_0}$ and $\phi(x)\ge \langle x \rangle^{-1-2\si}$,
we have 
$$
\| \nabla u \|^2_{H^{0,-1/2-\sigma}(M)}\le C(A)\left 
(\| f \|^2_{H^{0,1/2+\sigma}(M)}+\| u \|^2_{H^{0,-1/2-\sigma}(M)}+\eps \| u \|_{H^1(M)} \| u \|_{L^2(M)}\right).
$$
The term $\eps \| \nabla u \|_{L^2(M)} \| u \|_{L^2(M)}$ can be 
controlled with the help of \eqref{eq:la-eps-nab},
while the term $\eps \| u \|^2_{L^2(M)}$ can be bounded from the charge estimate
\eqref{eq:eps-charge}. As a consequence,
$$
\| \nabla u \|_{H^{0,-1/2-\sigma}(M)}\le C(A)\left 
(\| f \|_{H^{0,1/2+\sigma}(M)}+\| u \|_{H^{0,-1/2-\sigma}(M)}\right).
$$
To prove the desired result it would be sufficient to show that 
$$
\| u \|_{H^{0,-1/2-\sigma}(M)}\le c \left (\| \nabla u \|_{H^{0,-1/2-\sigma}(M)}+
\| u \|_{H^{0,-3/2-\sigma}(M)}\right) + \| f \|_{H^{0,1/2+\sigma}(M)}
$$
with a sufficiently small constant $c$. This follows immediately from \eqref{eq:la-nab}
with $c=C(\Gamma) \la^{-1/2}$ provided that $\la$ is sufficiently large. We now have that
$$
\| \nabla u \|_{H^{0,-1/2-\sigma}(M)}+\la^{\frac 12} \| u \|_{H^{0,-1/2-\sigma}(M)}
\le C(\Gamma) \| f \|_{H^{0,1/2+\sigma}(M)}.
$$
The remaining estimate for the second derivative follows from the elliptic regularity 
estimate \eqref{elliptic-0}.

The proof of Theorem \ref{nontrap} under an alternative assumption that the potential 
$V$ satisfies $|V(x)|+\langle x\rangle |\nabla V(x)|_g\le A \langle x\rangle^{-2\si_0}$ 
follows  simply requires running the positive commutator argument with $H=-\Delta_M+V$
in place of $-\Delta_M$. We omit the details.

\section{Carleman inequalities}\label{carleman-sec}

In this section we prove the unique continuation estimate in Proposition \ref{ucp}.  Let the notation and assumptions be as in that proposition.

The standard way to prove such unique continuation estimates is via \emph{Carleman inequalities}, which are inequalities of the form
\begin{equation}\label{carleman-1}
 \int_K (t^3 |u|^2 + t |\nabla u|^2) e^{2tw}\ dg \leq C(K,w) \int_K |f|^2 e^{2tw}\ dg
\end{equation}
for some suitable smooth weight function $w: K \to \R$ and some large real parameter $t$ (typically $t \geq C(K,w) (1 + \sqrt{\lambda})$).  However, as observed in Burq \cite{burq}, such inequalities are not available when $K$ has a non-trivial topology, due to the fact that $w$ can be forced to contain stationary points.  However, it is still possible
to obtain a \emph{two-weight} Carleman inequality of the form
\begin{equation}\label{carleman-2}
 \int_K (t^3 |u|^2 + t |\nabla u|^2) (e^{2tw_1} + e^{2tw_2})
\ dg \leq C(K,w_1,w_2) \int_K (|(-\Delta_M - \lambda) u|^2) (e^{2tw_1} + e^{2tw_2})\ dg,
\end{equation}
the point being that we can choose $w_1$ and $w_2$ to have stationary points at different locations (and furthermore
that $w_2 > w_1$ at the stationary points of $w_1$, and $w_1 > w_2$ at the stationary points of $w_2$).  Such an
inequality will allow us to obtain the above proposition.

We now turn to the details.  We begin by reviewing the standard approach to Carleman inequalities
in the literature; for a more detailed survey see \cite{tataru-carleman}.
Let $u$ be a solution to the Helmholtz equation
\eqref{helmholtz} which is supported on the compact set $K$.  We rewrite the equation in the 
form 
$$
(-\Delta_M-\la) u=F:=f\pm i\eps u-V u
$$
If we multiply this equation by a weight
$e^w$ for some smooth real-valued $w$, we obtain
$$ (e^w (-\Delta_M-\lambda) e^{-w}) e^w u = e^w F$$
We now split $e^w \Delta_M e^{-w}$ into real and imaginary parts
$$ e^w (-\Delta_M-\lambda) e^{-w} := \Re(e^w (-\Delta_M-\lambda) e^{-w}) + i \Im(e^w (-\Delta_M-\lambda) e^{-w}),$$
where $\Re(A) := \frac{A+A^*}{2}$ and $\Im(A) := \frac{A-A^*}{2i}$.  Using the
general identity 
$$\| A v \|_{L^2(M)}^2 = \| \Re(A) v \|_{L^2(M)}^2 + \| \Im(A) v \|_{L^2(M)} + 
\langle i[\Re(A), \Im(A)] v, v \rangle_{L^2(M)}$$
we thus obtain
\begin{equation}\label{ci-2}
\begin{split}
 \| \Re(e^w (-\Delta_M-\lambda) e^{-w}) e^w u\|_{L^2(M)}^2 
+ \| \Im(e^w (-\Delta_M-\lambda) e^{-w}) e^w u\|_{L^2(M)}^2 &\\
+ \langle i[\Re(e^w (-\Delta_M-\lambda) e^{-w}), \Im(e^w (-\Delta_M-\lambda) e^{-w})] e^w u, e^w u \rangle_{L^2(M)}
&\\
= \|e^w f \|_{L^2(M)}^2&;
\end{split}
\end{equation}
this identity is closely related to \eqref{carleman-identity}.  In particular we have the inequality
$$
\langle i[\Re(e^w (-\Delta_M-\lambda) e^{-w}), \Im(e^w (-\Delta_M-\lambda) e^{-w})] e^w u, e^w u \rangle_{L^2(M)} \leq
\|e^w f \|_{L^2(M)}^2.$$
Since $w$ was arbitrary, we may replace $w$ by $tw$ for some arbitrary real parameter $t$ (which we shall
think of as being large and positive) to obtain
$$
\langle i[\Re(e^{tw} (-\Delta_M-\lambda) e^{-tw}), \Im(e^{tw} (-\Delta_M-\lambda) e^{-tw})] e^{tw} u, e^{tw} u \rangle_{L^2(M)} \leq
\|e^{tw} f \|_{L^2(M)}^2.$$
The strategy is then to select the weight $w$ so that the commutator 
$i[\Re(e^{tw} (-\Delta_M-\lambda) e^{-tw}), \Im(e^{tw} (-\Delta_M-\lambda) e^{-tw})]$ is positive definite\footnote{Actually,
one only needs this commutator to be positive definite on the region of phase space 
where the operators $\Re(e^{tw} (-\Delta_M-\lambda) e^{-tw})$ and $\Im(e^{tw} (-\Delta_M-\lambda) e^{-tw})$ vanish, 
thanks to the other two terms in the identity \eqref{ci-2}.}, at least to top order
in $t$, to obtain a useful inequality such as \eqref{carleman-1}.  

It is certainly possible to adapt the above scheme to prove a two-weight inequality such as \eqref{carleman-2},
and this was essentially what was done in \cite{burq}.  Let us however pursue a slightly different (though closely
related) approach, based on the inequality \eqref{carleman-identity} as a substitute for \eqref{ci-2}, to emphasize
the fact that certain applications of the Carleman method can be viewed as special cases of the general abc method\footnote{One advantage of doing so is that the abc method does not require more than one degree of regularity on the metric; in particular, the Riemann curvature tensor does not appear here, whereas this tensor will appear
when computing the above commutator.}.  Substituting $w$ by $tw$ as before, the identity \eqref{carleman-identity} becomes
\begin{equation}\label{carleman-identity-t}
\begin{split}
t \int_K \Hess_{\alpha \beta}(w) \partial^\alpha(e^{tw} u) \partial^\beta(e^{tw} u)\ dg \quad & \\
+ t^3 \int_K \Hess_{\alpha \beta}(w) (\partial^\alpha w) (\partial^\beta w) |u|^2 e^{2tw}\ dg
&= t \int_K \frac{1}{4} (\Delta^2_M w) |u|^2 e^{2tw}\ dg + \frac{1}{8} \| e^{tw} F \|_{L^2(M)}^2.
\end{split}
\end{equation}

It is easiest to apply this inequality when $w$ is strictly geodesically convex (i.e. $\Hess(w) > 0$)
and non-stationary (i.e. $\nabla w \neq 0$) on the support $K$ of $u$, as the left-hand side is then positive,
and the first term on the right-hand side can be absorbed into the second term on the left-hand side
if $t$ is large enough.  One would then obtain a \emph{Carleman inequality} of the form
$$ t \| \nabla(e^{tw} u) \|_{L^2(M)}^2 + t^3 \| e^{tw} u \|_{L^2(M)}^2 \leq C(w,K) \| e^{tw} F \|_{L^2(M)}^2$$
for sufficiently large $t$, where the key point is that the constant $C(w,K)$ is independent of $t$.
For instance, in the Euclidean setting $M = \R^n$ one could take $w = \langle x \rangle = (1+|x|^2)^{1/2}$
to obtain an estimate of this form.

For more general manifolds $M$, however, it is not possible to find a weight function $w$ which is both
geodesically convex and non-stationary, for two reasons.  Firstly, if $K$ contains a closed geodesic\footnote{It is natural to conjecture that this implication can be reversed, i.e. that if $K$ was geodesically non-trapping and topologically trivial then there exists a weight function $w$ which is geodesically convex and non-stationary.  This would allow one to replace the pseudo-differential calculus considerations in the previous section by a more elementary integration by parts argument.  Unfortunately, there exist manifolds which are geodesically non-trapping, but for which no such weight function $w$ exists; this can be seen by a minor modification of the example in \cite[Section 6]{ht}.  The point is that geodesic convexity is equivalent to the assertion that the function $\xi_\alpha \nabla^\alpha w$
is increasing with respect to geodesic flow on the cotangent bundle $T^*K = \{ (x, \xi_\alpha): \xi_\alpha \in T^*_x 
M \}$, but this function $\xi_\alpha \nabla^\alpha w$ is necessarily linear in $\xi$, which places additional constraints on solvability beyond
the mere non-existence of trapped geodesics.  If one uses pseudo-differential operators instead then one does
not have this geometrically unnatural linearity constraint.}, then it is clearly impossible for $w$ to be strictly geodesically convex on this geodesic.  Secondly, if $K$ contains a non-trivial topology (e.g. $K$ contains a handle), then from Morse theory we see that $w$ must contain at least one stationary point.  
As mentioned earlier, the latter difficulty will be overcome by considering a pair of weight functions $w_1,w_2$
rather than a single weight $w$.  To overcome the former difficulty, we can replace the weight functions $w_1, w_2$
by $h(w_1)$, $h(w_2)$ for ``sufficiently convex'' functions $h: \R \to \R$, however this only gives us convexity of
the Hessian in the gradient directions $\nabla w_1$, $\nabla w_2$ respectively (note though
that this is the most important direction that we need convexity in, as one can already see from the second term on
the left-hand side of \eqref{carleman-identity-t}).  To deal with the possible failure
of convexity in the non-gradient directions, we will use energy identity \eqref{energy-identity-2}.  More precisely, we have

\begin{lemma}[Preliminary Carleman inequality]\label{carleman-prelim}  Let $w: K \to \R$ be a smooth function and let $B \subset K$ be a open subset of $K$.  Suppose
that $w$ is non-stationary in the set $K \backslash B$, so that the unit vector field $n^\alpha := (\nabla^\alpha w) / |\nabla^\alpha w|_g$ is well defined.  Suppose that $w$ is convex in the direction of $n^\alpha$, and more precisely
$$
\Hess_{nn}(w) := \Hess_{\alpha \beta}(w) n^\alpha n^\beta > 0 \hbox{ on } K \backslash B
$$
for some $c > 0$.  Suppose also that for any unit vector field $X^\alpha$ we have the estimate
$$ \Hess_{\alpha \beta}(w) X^\alpha X^\beta > - \frac{1}{10} \Hess_{nn}(w) |X|_g^2
\hbox{ on } K \backslash B.$$
Then for any solution $u \in H^2(K)$ to the Helmholtz equation \eqref{helmholtz} which is supported in $K$,
and any $t \geq C(K,B,w) (1 + \sqrt{\lambda})$ we have
$$
 \int_K (t^3 |u|^2 + t |\nabla u|^2) e^{2tw}\ dg \leq C(K,B,w) 
\left(\int_K |F|^2 e^{2tw}\ dg + \int_B (t^3 |u|^2 + t |\nabla u|^2) e^{2tw}\ dg\right).$$
\end{lemma}

\begin{proof}  Suppose that we are in the region $K \backslash B$.  Then by hypothesis, we have
$$\Hess_{\alpha \beta}(w) (\partial^\alpha w) (\partial^\beta w) |u|^2 e^{2tw}
= \Hess_{nn}(w) |\nabla w|_g^2 |u|^2 e^{2tw} \geq 0$$
and
\begin{align*}
 \Hess_{\alpha \beta}(w) \partial^\alpha(e^{tw} u) \partial^\beta(e^{tw} u)
&\geq - \frac{1}{10} \Hess_{nn}(w) |\nabla(e^{tw} u)|_g^2\\
&\geq - \frac{1}{5} \Hess_{nn}(w) (t^2 |\nabla w|_g^2 e^{2tw} |u|^2 + |\nabla u|_g^2 e^{2tw})
\end{align*}
and hence
\begin{align*}
 t \Hess_{\alpha \beta}(w) \partial^\alpha(e^{tw} u) \partial^\beta(e^{tw} u)
&+ t^3 \Hess_{\alpha \beta}(w) (\partial^\alpha w) (\partial^\beta w) |u|^2 e^{2tw} \\
&
\geq \Hess_{nn}(w) (\frac{4}{5} t^3 |\nabla w|_g^2 e^{2tw} |u|^2 - \frac{1}{8} t|\nabla u|_g^2 e^{2tw}).
\end{align*}
Integrating this on $K \backslash B$ and using \eqref{carleman-identity-t}, we see that
\begin{align*}
\frac{4}{5} t^3 \int_{K \backslash B} \Hess_{nn}(w) |\nabla w|_g^2 e^{2tw} |u|^2\ dg
&\leq C t \int_K (\Delta^2_M w) |u|^2 e^{2tw}\ dg + C \| e^{tw} F \|_{L^2(M)}^2\\
&\quad + \frac{1}{5} t \int_K \Hess_{nn}(w) \varphi(w) |\nabla u|_g^2 e^{2tw}\ dg\\
&\quad + C(K,B,w) \int_B (t^3 |u|^2 + t |\nabla u|^2) e^{2tw}\ dg,
\end{align*}
where $\varphi$ is a cutoff function that equals 1 on $K \backslash B$ and vanishes near the stationary
points of $w$.  But if we apply the energy identity \eqref{energy-identity-2} with
$\chi = \Hess_{nn}(w) \varphi(w) e^{2tw}$, and using the Cauchy-Schwarz inequality
$\Re(\overline{u} F) \leq t |u|^2 + t^{-1} |F|^2$, we obtain
\begin{align*}
\frac{1}{5} t \int_K \Hess_{nn}(w) \varphi(w) |\nabla u|_g^2 e^{2tw}\ dg
&\leq \frac{1}{5} \lambda t \int_K \Hess_{nn}(w) \varphi(w) |u|^2 e^{2tw}\ dg\\
&\quad + \frac{2}{5} t^3 \int_K \Hess_{nn}(w) \varphi(w) |\nabla w|_g^2 |u|^2 e^{2tw}\ dg \\
&\quad + C(K,B,w) t^2 \int_K |u|^2 e^{2tw}\ dg \\
&\quad + \| e^{tw} F \|_{L^2(M)}^2.
\end{align*}
The contribution of the first two terms on $K \backslash B$ can be bounded
by $$\frac{3}{5} t^3 \int_{K \backslash B} \Hess_{nn}(w) |\nabla w|_g^2 e^{2tw} |u|^2\ dg$$
term since $t^2$ is assumed to be large compared with $\lambda$.  Thus we have
\begin{align*}
\frac{1}{5} t^3 \int_{K \backslash B} \Hess_{nn}(w) |\nabla w|_g^2 e^{2tw} |u|^2\ dg
&\leq  C \| e^{tw} F \|_{L^2(M)}^2 \\
&+ C(K,B,w) t^2 \int_K |u|^2 e^{2tw}\ dg \\
& + C(K,B,w) \int_B (t^3 |u|^2 + t |\nabla u|^2) e^{2tw}\ dg.
\end{align*}
(absorbing the $\Delta^2_M w$ term into the $C(K,B,w) t^2 \int_K |u|^2 e^{2tw}\ dg$
error).  But since $\Hess_{nn}(w) |\nabla w|_g^2$ is non-zero on the compact set $K \backslash B$, and
$t$ is assumed large compared with $C(K,B,w)$, we can
absorb the $C(K,B,w) t^2 \int_K |u|^2 e^{2tw}\ dg$ error into the right-hand side (plus the error on $B$)
and conclude that
$$
t^3 \int_K e^{2tw} |u|^2\ dg
\leq C(K,B,w) \|e^{tw} F \|_{L^2(M)}^2
  + C(K,B,w) \int_B (t^3 |u|^2 + t |\nabla u|^2) e^{2tw}\ dg.
$$
Applying \eqref{energy-identity-2} again (with $\chi = e^{2tw}$, and using the assumption that $t$ is large
compared with $\sqrt{\lambda}$) we obtain the result.
\end{proof}

The above lemma has the drawback that there is an error term involving $u$ on the right hand side.  However, it
is localized to a smaller set $B$ than $K$.  We can exploit this localization by using two weights instead of one, whose
critical points are at different locations.
We begin with a Morse theory lemma (first observed by Burq \cite{burq}):

\begin{lemma}[Construction of Morse function pair] There exist smooth functions $a_1: K \to \R$ and $a_2: K \to \R$ 
which each have finitely many critical points on $K$, which are all in the interior of $K$.  Furthermore, 
whenever $x$ is a critical point of $a_1$ we have $a_2(x) > a_1(x)$, and whenever $x$ is a critical 
point of $a_2$ we have $a_1(x) > a_2(x)$.  (In particular, the critical points of $a_1$ are at distinct locations
from the critical points of $a_2$).
\end{lemma}

\begin{proof}  By enlarging $K$ if necessary we may assume that the boundary $\partial K$ of $K$ is a sphere 
$\partial K = \{ (r,y): r = R\}$ for some $R \gg 1$; in particular, $K$ is now connected (since $M$ is connected by
hypothesis).  
We construct $a_1$ to be the solution to the Dirichlet problem
$$ \Delta_M a_1 = 1 \hbox{ on } K; \quad a_1 = 0 \hbox{ on } \partial K$$
which can be constructed for instance by a standard variational procedure (or the spectral theory of the Laplacian on a compact manifold with Dirichlet boundary conditions).  By standard elliptic theory, this function
is smooth on $K$, and all critical points lie in the interior of $K$ and are non-degenerate
(since $\Hess(a_1)$ is clearly non-vanishing), so in particular the number of critical points is finite.
Furthermore none of the critical points can be local maxima, since $\Delta_M a_1$ is positive.  In particular
if we enumerate the critical points of $a_1$ as $x_1, \ldots, x_m$, and let $\eps > 0$ be a radius so small
that the closed balls $\overline{B(x_1,\eps)}, \ldots, \overline{B(x_m,\eps)}$ are disjoint from each other
and from the boundary $\partial K$, then we can find $x'_1, \ldots, x'_m$ such that $d_M(x_i,x'_i) < \eps$ and 
$a_1(x'_i) > a_1(x_i)$.

We now let $\phi: K \to K$ be a diffeomorphism which is equal to the identity outside of the balls $\overline{B(x_1,\eps)}, \ldots, \overline{B(x_m,\eps)}$, and which swaps $x_i$ and $x'_i$ for each $i$.
If we then set $a_2 := a_1 \circ \phi$ then it is clear that $a_2$ has critical points precisely at $x'_1,\ldots,x'_m$,
and the claims follow.
\end{proof}

We can now combine these two lemmas to obtain a two-weight Carleman inequality which avoids the $B$ errors.

\begin{corollary}[Two weight Carleman inequality]\label{cor-carleman} There exist smooth functions $w_1: K \to \R$ and $w_2: K \to \R$
with the property that for any solution $u \in H^2(K)$ to the Helmholtz equation \eqref{helmholtz} which is 
supported in $K$, and any $t \geq C(K,w_1,w_2) (1 + \sqrt{\lambda})$ we have
$$
 \int_K (t^3 |u|^2 + t |\nabla u|_g^2) (e^{2tw_1} + e^{2tw_2})\ dg \leq C(K,w_1,w_2) 
\int_K |F|^2 (e^{2tw_1} + e^{2tw_2})\ dg.$$
\end{corollary}

\begin{proof}  Let $a_1, a_2$ be constructed by the previous lemma, and let $x_1, \ldots, x_m$
be the critical points of $a_1$ and $x'_1,\ldots, x'_{m'}$ be the critical points of $a_2$.  If 
$\eps = \eps(K,a_1,a_2) > 0$
is a sufficiently small radius, then we have $a_2 > a_1$ on the set
$$ B_1 := B(x_1,\eps) \cup \ldots \cup B(x_m,\eps) \subset K$$
and $a_1 > a_2$ on the set
$$ B_2 := B(x'_1,\eps) \cup \ldots \cup B(x'_{m'},\eps) \subset K$$ 
for some $c = c(K,a_1,a_2) > 0$.

We now let $A = A(K,a_1,a_2,\eps) \gg 1$ be a large parameter to be chosen later, and set $w_1 := \exp(A a_1)$
and $w_2 := \exp(A a_2)$.  Observe that for $j=1,2$, $w_j$ has no critical points outside of $B_j$.  Furthermore,
on $K \backslash B_j$ we compute
$$ n_j^\alpha := \nabla^\alpha w_j / |\nabla w_j|_g = \nabla^\alpha a / |\nabla a|_g$$
and
$$
\Hess_{\alpha \beta}(w_j) = (A^2 \nabla_\alpha a_j \nabla_\beta a_j + A \Hess_{\alpha \beta}(a)) \exp(A a_j)
$$
so in particular
$$
\Hess_{n_j n_j}(w_j) = (A^2 |\nabla a_j|_g^2 + A \Hess_{n_j n_j}(a_j)) \exp(A a_j)
$$
and
$$
\Hess_{XX}(w_j) = (A^2 |X^\alpha \nabla_\alpha a_j|^2 + A \Hess_{XX}(a_j)) \exp(A a_j).$$
Thus if $A$ is large enough, the hypotheses of Lemma \ref{carleman-prelim} are obeyed, and we have
\begin{align*}
 \int_K (t^3 |u|^2 + t |\nabla u|^2) e^{2tw_j}\ dg &\leq C(K,B_1,B_2,w_1,w_2)\times \\
 & \times 
\left(\int_K |F|^2 e^{2tw_j}\ dg + \int_{B_j} (t^3 |u|^2 + t |\nabla u|^2) e^{2tw_j}\ dg\right)
\end{align*}
for $j=1,2$.  Adding the two inequalities together we obtain
\begin{align*}
 \int_K (t^3 |u|^2 + t |\nabla u|^2) (e^{2tw_1} + e^{2tw_2})\ dg \leq& C(K,B_1,B_2,w_1,w_2) 
\biggl(\int_K |F|^2 (e^{2tw_1} + e^{2tw_2})\ dg \\
&+ 
\int_{B_1} (t^3 |u|^2 + t |\nabla u|^2) e^{2tw_1}\ dg \\
&+
\int_{B_2} (t^3 |u|^2 + t |\nabla u|^2) e^{2tw_2}\ dg\biggr).
\end{align*}
But observe that $w_2 \geq w_1 + c$ on $B_1$ and $w_1 \geq w_2 +c$ on $B_2$ for some
$c = c(K,w_1,w_2,B_1,B_2) > 0$, and thus for $j=1,2$ we have $e^{2tw_j} \leq e^{-2ct} (e^{2tw_1} + e^{2tw_2})$
on $B_j$.  Thus if $t$ is large enough, we can absorb the last two terms on the right-hand
side into the left-hand side, and the claim follows.
\end{proof}

\begin{remark} Morally speaking, the above two-weight Carleman inequality can be viewed heuristically
as a single-weight Carleman inequality on the product manifold $M \times M$, with $u$ replaced by the tensor
product $u \otimes u$ and $w$ replaced by the tensor sum $w_1 \oplus w_2$ (and $\lambda$ replaced by $2\lambda$).  
The point is that the critical points 
$(x_i, x'_j)$ of $w_1 \oplus w_2$ lie off the diagonal, and the contribution of the weights at those points
can be dominated by the contributions of the weights at the diagonal points $(x_i,x_i)$ and $(x'_j,x'_j)$,
for instance by the Cauchy-Schwarz inequality.  Intriguingly, this perspective shares many similarities with the 
philosophy underying the interaction Morawetz inequalities, used for instance in \cite{htw}, where 
the positive commutator method of the previous section was also extended to a product setting.
\end{remark}

We can now quickly prove Proposition \ref{ucp}.

\begin{proof}[Proof of Proposition \ref{ucp}]  Let $\chi$ be a smooth cutoff function which equals $1$
on $K'$ and is supported on $K$.  Then $\chi u$ is supported on $K$ and obeys the equation
$$ (-\Delta_M - \lambda) (\chi u) = - 2 (\nabla_\alpha \chi) \nabla^\alpha u - (\Delta_M \chi) u + \chi 
(f\pm i\eps u-Vu)
+ \chi V u.$$
In particular, $(-\Delta_M - \lambda) (\chi u)$ is less than 
$\chi (|f| + O(A|u|))$ on $K'$, and obeys the bound
$$ (-\Delta_M - \lambda) (\chi u) \leq C(\chi) (A |u| + |\nabla u|_g + |f| )$$
on $K \backslash K'$.
Applying Corollary \ref{cor-carleman} to $\chi u$ we thus obtain
\begin{align*}
 \int_K (t^3 |u|^2 + t |\nabla u|_g^2) (e^{2tw_1} + e^{2tw_2})\ dg \leq&
 C(K,w_1,w_2) 
C(\chi) \int_K |f|^2 (e^{2tw_1} + e^{2tw_2})\ dg \\
&+
A C(\chi) \int_K |u|^2 (e^{2tw_1} + e^{2tw_2})\ dg\\
&+ C(\chi) \int_{K \backslash K'} ( |u|^2 + |\nabla u|_g^2 ) (e^{2tw_1} + e^{2tw_2})\ dg
\end{align*}
for $t \geq C(K,w_1,w_2) (1 + \sqrt{\lambda})$, and the claim follows by choosing
$t$ to be a large multiple of $C(A,K,\chi,w_1,w_2) (1 + \sqrt{\lambda})$ (to absorb the second term on the 
right-hand side) and noting that $w_1$ and $w_2$ are smooth and hence bounded above and below on the 
compact set $K$.
\end{proof}

\begin{remark} Proposition \ref{ucp} gives good control on solutions to the Helmholtz equation on a compact
set $K$.  To obtain a limiting absorption principle, we will have to combine this proposition with more ``global''
estimates, such as the Morawetz estimates of the previous section, or the Bessel ODE analysis in Section \ref{bessel-sec}.
\end{remark}

\section{Conservation laws and differential inequalities of Bessel type}\label{bessel-sec}

The results of Lemma \ref{morawetz-lemma}  and Proposition \ref{ucp}
provide us with the following very useful dichotomy. On one hand Lemma \ref{morawetz-lemma} 
shows that if the limiting absorption principle can be proved for the restriction,  to a certain dyadic 
region $r_0/2\le \langle x\rangle\le 2r_0$, 
of a solution $u$ of the Helmholtz equation, then it also holds on the set 
$\langle x\rangle\ge 2r_0$. Control of the region $r_0/2\le \langle x\rangle\le 2r_0$ together
with the unique continuation principle of Proposition \ref{ucp} also
imply that the limiting absorption principle can be extended to the set 
$\langle x\rangle\le r_0/2$. On the other hand  Proposition \ref{ucp} allows us an alternative
scenario. To prove the limiting absorption principle on a compact set $K$, which can be thought of 
as the region $\langle x\rangle\le r_0/2$, rather than proving unconditional control on 
the solution in the region $r_0/2\le \langle x\rangle\le 2r_0$ it would be sufficient to show 
instead  that the solution varies super-exponentially through that region, i.e., its rate of change 
is given by $e^{-C(1+\sqrt\la)}$ with a sufficiently large constant $C$. 

To show that either of these scenarios must be realized we need to perform analysis near the
 asymptotic end of the manifold $M$; this is the purpose of Lemma \ref{dichotomy}, which we shall prove in this section and in the next. 

We begin with an informal discussion.  For simplicity consider the case when $M$ is asymptotically Euclidean manifold.  Then we can heuristically approximate the Helmholtz equation by its Euclidean version
 $$ \left(\frac{\partial^2}{\partial r^2} + \frac{n-1}{r} \frac{\partial}{\partial r}
+ \frac{1}{r^2} \Delta_{S^{n-1}} - V + \lambda\right) u = -F.$$
Assuming that $V$ is radial, $V = V(r)$,
and applying the ansatz
$$
 u(r,\omega) = r^{-(n-1)/2} v(r) Y_l(\omega), \quad F(r,\omega) = -r^{-(n+1)/2} G(r) Y_l(\omega)
$$
where $Y_l$ is a spherical harmonic of order $l$ on $S^{n-1}$, normalized
to have $L^2(S^{n-1})$ norm equal to 1, the Helmholtz equation becomes the \emph{Bessel ordinary
differential equation}
$$
 v_{rr} - \frac{L(L-1)}{r^2} v - Vv + \lambda v = G,
$$
where $L := l + \frac{n-1}{2}$. As we shall see in Section \ref{bessel-matching-sec}
this equation will play an important role in a counterexample construction of Proposition 
\ref{bessel}. Despite providing good insight into behavior of solutions of the Helmholtz 
equation near infinity, the use of the Bessel equation approximation has several drawbacks. 
At first glance,
it seems that this equation only emerges when the solution $u$ has a specific structure, namely that it
decouples as the product of a radial function and a spherical harmonic.  Of course, one could orthogonally
decompose an arbitrary function $u$ into such products and work on each harmonic separately (as is 
done in a number of places in the literature, e.g. \cite{burq}), but this becomes difficult if the metric 
and potential only decay slowly at infinity (although such analysis well suited for compact perturbations
of the standard Euclidean metric).
Also, such an approach often requires detailed analysis of the asymptotics
of Bessel or Hankel functions.  Here, we present a more ``energy-based'' method to simulate differential equations 
(or differential inequalities) Bessel type for solutions to the Helmholtz equation \eqref{helmholtz}, without
requiring an explicit decomposition into spherical harmonics, and without requiring any knowledge of Bessel or Hankel
functions (although such functions are in some sense lurking in the background in what follows).

In this section we shall work purely in the asymptotic region $r > R_0$; thus the analysis here may be viewed
as a ``black box'' analysis, requiring no knowledge of the manifold, solution, or potential in the
interior region $r \leq R_0$.  Eventually we will 
combine this black box analysis with the Carleman analysis 
in the near region $r = O(1)$ from previous sections to
obtain the full limiting absorption principle.  

We write the metric in the form
$$ g = dr^2 + r^2 h_{jk}[r](\omega) d\omega^j d\omega^k$$
where
$$ h_{jk}[r](\omega) := h_{jk}(\omega) + r^{-2\sigma_0} e_{jk}(r,\omega)$$
and work on the hypersurfaces $S_r := \{ (r,\omega): \omega \in \partial M \}$, which are 
naturally endowed with the metric
$h[r]$ and the corresponding measure $dh[r] := \sqrt{h[r]} d\omega$; note that this differs from the induced
measure $dg|_{S_r}$ by a factor of $r^{n-1}$.  In particular we have the co-area formula
\begin{equation}\label{polar}
\int_{r > R_0} f\ dg = \int_{R_0}^\infty \int_{S_r} f(r,\omega)\ dh[r](\omega) r^{n-1} dr.
\end{equation}
We now rewrite the resolvent equation \eqref{resolvent} in polar co-ordinate form as
$$ \frac{\partial^2}{\partial r} u + \frac{n-1}{r} \frac{\partial}{\partial r} u
+ \frac{\frac{\partial}{\partial r} dh[r]}{dh[r]} \frac{\partial}{\partial r} u
+ \frac{1}{r^2} \Delta_{h[r]} u + \lambda u = Vu \mp i\eps u - f .$$
Note that 
$$
\theta_{jk}(r,\omega) =\frac 12 \pa_r \theta_{jk}(r,\omega)
$$
is the second fundamental form of the surface $S_r$  relative to the renormalized metric
$$
\overline g=dr^2+h_{jk}d\omega^j\,d\omega^k.
$$
with mean curvature
\begin{equation}\label{eq:mean-curv}
\theta=\frac{\frac{\partial}{\partial r} dh[r]}{dh[r]}
\end{equation}
It follows from \eqref{metric-decay} that
\begin{equation}\label{div-formula}
\theta= O( r^{-1-2\sigma_0} ), \quad |\theta|_{h[r]}=O(r^{-1-2\si_0}).
\end{equation}
We expect solutions $u$ of the Helmholtz equation to decay like $r^{-(n-1)/2}$ as $r \to \infty$.  Thus we can
renormalize $u$ by defining 
\begin{equation}\label{ansatz-2}
v := r^{(n-1)/2} u
\end{equation}
(cf. \eqref{ansatz}) and observe that $v$ obeys a Bessel-like equation
\begin{equation}\label{V-bessel}
 v_{rr} +\frac 1{r^2}\big (\Delta_{h[r]}-\frac {(n-1)(n-3)}4\big)v +\la v= - \theta v_{r} 
 + (V + \frac {n-1}r \theta) v \mp i \eps v + r^{(n-1)/2} f 
\end{equation}
with the operator $-\Delta_{h[r]}$ playing the role of the parameter $l(l+n-2)$.

We now define the ``spherical energies''
\begin{align*}
{\text{Mass}}\qquad \M[r] &:= \int_{S_r} |v|^2\ dh[r] \\
{\text{Radial energy}}\qquad \RR[r] &:= \int_{S_r} |v_r|^2\ dh[r] \\
{\text{Angular energy}}\qquad \A[r] &:= \int_{S_r} \frac{1}{r^2} (|\nabla_\omega v|_{h[r]}^2 + 
\frac{(n-1)(n-3)}{4} |v|^2) \ dh[r] \\
{\text{Mass flux}}\qquad\F[r] &:= \int_{S_r} \Re( \overline{v} v_r)\ dh[r],
\end{align*}
where $\nabla_\omega v$ is the angular gradient.  Note that the quantity $\M[r]$ was already introduced in Lemma \ref{dichotomy}.  We also need the ``forcing term''
\begin{equation}\label{g-force}
\G[r] := r^{(n-1)/2} \int_{S_r} (|v| + \frac{|\nabla v|_g}{r^{-1} + \lambda^{1/2}}) (|f| + \eps |u|)
\ dh[r],
\end{equation}
and record the following three ``equations of motion'' for the spherical energies, which are closely related
to the conservation laws \eqref{charge-gradient}, \eqref{bochner-identity}, \eqref{divx-form-2}.

\begin{lemma}[Equations of motion]\label{equations-of-motion}  
We have positivity properties
\begin{equation}\label{mra}
\M[r], \RR[r], \A[r] \geq 0
\end{equation}
the Cauchy-Schwarz estimate
\begin{equation}\label{flux-cauchy}
|\F[r]| \leq \M[r]^{1/2} \RR[r]^{1/2}
\end{equation}
and the equations of motion
\begin{align}
&\frac{d}{dr} \M = 2 \F+ O( r^{-1-2\sigma_0})\M  \label{mass-dynamics}\\
&\frac{d}{dr} \F = \RR +\A - \lambda \M
+
O( r^{-2-2\sigma_0}) \M  + O(r^{-1-\sigma_0})\lambda^{1/2} \M+ O(\G)
\label{flux-dynamics} \\
&\frac{d}{dr} (\RR + \lambda \M- \A) = \frac{2}{r} \A+ O( r^{-1-2\sigma_0})
(\A+\RR+\la\M) \nonumber \\
&\hskip 6pc\quad + O( r^{-3-2\sigma_0})\M+ O(r^{-1} + \lambda^{1/2}) \G.
\label{cons-dynamics}
\end{align}
Here the implicit constants are allowed to depend on $M$ and $A$.
\end{lemma}

\begin{proof}
While in dimensions $n\ge 3$ positivity of the spherical energies and the 
Cauchy-Schwarz estimate are obvious,
 equations \eqref{mass-dynamics}-\eqref{cons-dynamics} follow from the identities:
\begin{align*}
\frac d{dr} |v|^2&=2 \Re (v_r\overline v),\\
\frac {d}{dr}(v_r\overline v)&=|v_r|^2- \frac 1{r^2}\big (\Delta_{h[r]}-\frac {(n-1)(n-3)}4\big)v
\overline v - \la |v|^2- \theta v_r \overline v \\&+  
\Big ((V + \frac {n-1}r \theta) v \mp i \eps v + r^{(n-1)/2} f \Big )\overline v,\\
\frac {d}{dr}(|v_r|^2&+\la |v|^2 -\frac 1{r^2} |\nabla_\omega v|^2_h-\frac{(n-1)(n-3)}{4r^2}|v|^2) =
\frac 2r\left( \frac 1{r^2}|\nabla_\omega v|^2_h+\frac{(n-1)(n-3)}{4r^2}|v|^2\right) \\ &+\frac 2{r^2}\div_\omega
\Re (\nabla_\omega v\, \overline{v_r})
-2 \theta |v_r|^2 -\frac 1{r^2}
\theta_{jk}\nabla_\omega^j v \overline{\nabla_\omega^k v}
\\&+2\Re\Big ((V + \frac {n-1}r \theta) v \mp i \eps v + r^{(n-1)/2} f \Big )\overline {v_r},
\end{align*}
identity \eqref{eq:mean-curv} and the assumptions \eqref{div-formula}
and $|V(x)|\le A(r^{-2-2\si_0}+\la^{1/2} r^{-1-2\si_0})$.
\end{proof}

\begin{remark} It is helpful to keep in mind the model case
(see \eqref{ansatz}), where  $v$ solves the Bessel differential equation
$$  v_{rr} - \frac{L(L-1)}{r^2} v + \lambda v = 0,$$
in which case
\begin{align*}
 &\M[r] = |v(r)|^2; \quad \RR[r] = |v_r(r)|^2; \quad \A[r] = \frac{L(L-1)}{r^2} |v(r)|^2; 
\\&\F(r) = \Re(\overline{v(r)} v_r(r)); \quad \G[r] = 0.
\end{align*}
The reader may wish to verify the above equations of motion (with all error terms set to zero)
in this special case.  It may also be useful to keep in mind the dimensional analysis
$$ r \sim \operatorname{length}^1; \quad \M \sim \operatorname{length}^0; \quad \F \sim \operatorname{length}^{-1}; \quad 
\RR, \A, \lambda, \G \sim \operatorname{length}^{-2},$$
noting that the above equations then become dimensionally consistent up to errors involving $\sigma_0$.
\end{remark}

\begin{remark}  Note that while we have four energies, we only have three equations of motion; we do not
control the evolution of $\RR$ and $\A$ separately, but only have an equation for a certain
combination $\RR - \A$ of these two.  However, we can obtain a lower bound on $\RR$ from
\eqref{flux-cauchy}.  This system of three equations and one inequality is thus still underdetermined, but
we will still be able to extract enough control out of this system to establish all the estimates we need.
\end{remark}

\begin{remark} Of course, the three equations of motion can also be interpreted in terms of
the Friedrichs abc method, where the cutoff $\chi$ is now the surface measure on a sphere $S_r$.  We omit
the details.
\end{remark}

Now we obtain some preliminary estimates on the above energies, in the setting of Lemma \ref{dichotomy}.

\begin{proposition}[Preliminary estimates]\label{resolvent-boundary} Let the notation and assumptions be as in Lemma \ref{dichotomy}.  
 Then we have the integral estimate
\begin{equation}\label{delta}
 \int_{R_0}^\infty \G[r]\ dr = O(\delta).
\end{equation}
and the boundary condition
\begin{equation}\label{boundary}
 \lim_{r \to \infty} \frac{1}{r} \int_r^{2r} \M[s] + \RR[s] + \A[s]\ ds = 0.
\end{equation}
\end{proposition}

\begin{proof}
The boundary condition \eqref{boundary} follows from the normalization 
$\|u\|_{H^{0,-1/2-\si}(M)}=1$ and  \eqref{energy-spiral}.
To prove \eqref{delta}, we first see from \eqref{g-force}, \eqref{ansatz} that
$$
\G[r] \leq C r^{n-1} \int_{S_r} (|u| + \frac{|\nabla u|_g}{r^{-1} + \lambda^{1/2}}) (|f| + \eps |u|)
\ dh[r]
$$
and so by dyadic decomposition it would suffice to show that
\begin{equation}\label{sumr}
 \sum_{R \geq R_0} 
\int_{R \leq \langle x \rangle \leq 2R}
\left(|u| + \frac{|\nabla u|_g}{R^{-1} + \lambda^{1/2}}\right) (|f| + \eps |u|)\,dg = O(\delta).
\end{equation}
From \eqref{uf} we have the bounds
$$ \int_{R \leq \langle x \rangle \leq 2R} |u|^2\,dg \leq c_R R^{1+2\sigma}$$
and
$$ \int_{R \leq \langle x \rangle \leq 2R} |f|^2\,dg \leq \delta^2 c_R R^{-1-2\sigma}$$
where $c_R > 0$ are numbers such that $\sum_R c_R = O(1)$.  From the charge estimate (Lemma \ref{charge-est}) we also have
$$ \eps \int_M |u|^2\ dg \leq \| f \|_{H^{0,1/2+\sigma}(M)}  \| u \|_{H^{0,-1/2-\sigma}(M)}\le \de$$
and thus
$$ \int_{R \leq \langle x \rangle \leq 2R} |u|^2\,dg \leq c_R \de / \eps$$
(after adjusting $c_R$ if necessary).  Finally, from \eqref{energy-spiral} we have
$$
\int_{R\le\langle x\rangle\le 2R} \left(|u| + \frac{|\nabla u|_g}{R^{-1} + \lambda^{1/2}}\right)^2 \,dg
 \le C \int_{R/2\le \langle x\rangle\le 4R}
\Big (|u|^2 + \frac {f^2}{(\la+R^{-2})^2}\Big ) dg 
$$
and so after adjusting $c_R$ a bit more we obtain
$$
\int_{R\le\langle x\rangle\le 2R} (|u| + \frac{|\nabla u|_g}{R^{-1} + \lambda^{1/2}})^2\,dg
\leq C c_R \left( \min( R^{1+2\sigma},\de/\eps ) + \frac{\delta^2 R^{-1-2\sigma}}{(\la + R^{-2})^{2}} \right).$$
On the other hand, we have
$$ \int_{R\le\langle x\rangle\le 2R} (|f| + \eps |u|)^2\,dg \leq C c_R (\delta^2 R^{-1-2\sigma} + \min( \eps \delta,k \eps^2 R^{1+2\delta} ) ).$$
Applying Cauchy-Schwarz, we will obtain \eqref{sumr} if we can show that
$$ \left( \min( R^{1+2\sigma},\de/\eps ) + \frac{\delta^2 R^{-1-2\sigma}}{(\la + R^{-2})^{2}} \right)
(\delta^2 R^{-1-2\sigma} + \min( \eps \delta,k \eps^2 R^{1+2\delta} ) ) \leq C \delta^2$$
for all $R \geq R_0$.  This is clear for the terms involving $\min( R^{1+2\sigma},\de/\eps )$.  For the term
$$ \frac{\delta^2 R^{-1-2\sigma}}{(\la + R^{-2})^{2}} \delta^2 R^{-1-2\sigma} $$
we bound $\la + R^{-2}$ from below by $R^{-2}$ to obtain a bound of $\delta^4 R^{2-4\sigma}$, which is acceptable since $\sigma < 1/2$ and $\delta \leq C$.  Finally, for the term
$$ \frac{\delta^2 R^{-1-2\sigma}}{(\la + R^{-2})^{2}} \min( \eps \delta,k \eps^2 R^{1+2\delta} )$$
we bound $\la + R^{-2}$ from below by $\la$, and use the second term in the minimum, to obtain a bound of $\eps^2 \delta^2 / \la^2$, which is acceptable since $\eps < \la$.
\end{proof}

\section{An ODE lemma}\label{sec:ODE}

In view of Lemma \ref{equations-of-motion} and Proposition \ref{resolvent-boundary}, we see that Lemma \ref{dichotomy} will follow immediately from the following ODE lemma.

\begin{lemma}[ODE Lemma]\label{ode}  Let $C_1 \gg R_0$ be a large number, and then let $C_2 \gg C_1$ be an even larger number.
For all $r \geq R_0$, let $\M$, $\RR$, $\A$,
$\F$, $\G$ be real-valued functions obeying the differential inequalities in Lemma \ref{equations-of-motion}
as well as the properties \eqref{delta}, \eqref{boundary}.  Then if $C_1$ is sufficiently large (but not depending on $\lambda$), and $C_2$ is sufficiently large depending on $C_1$ (but not on $\lambda$), then one of the following must be true:
\begin{itemize}
\item (Boundedness) There exists a radius $C_1 \leq r_0 \leq C(C_1,C_2)$ such that \eqref{bounded-mass} holds.
\item (Exponential growth) For all $C_1 \leq r \leq 10C_1$, we have \eqref{mass-growth}.
\end{itemize}
Of course, the magnitude of $C_1$, $C_2$, $C(C_1,C_2)$ will depend on the implicit constants in Lemma \ref{equations-of-motion} and \eqref{delta} (which in practice will depend on $M$ and $A$).
\end{lemma}

The proof of this Lemma is lengthy and will occupy the remainder of this section.  

\subsection{Heuristics}

We first describe in informal terms why one would expect there to be a dichotomy of the type asserted in
Lemma \ref{ode}.  We first observe that \eqref{cons-dynamics} is an approximate monotonicity formula
for the quantity $\RR+ \lambda \M - \A$.  Since this quantity is zero at infinity
by \eqref{boundary}, we expect it to be negative (up to errors of size of $O(\delta)$, thanks to \eqref{delta})
at other values of $r$.  Thus we have a lower bound on $\A$, heuristically of the form
\begin{equation}\label{ang-grow-heuristic}
 \A \geq \RR + \lambda \M- O(\delta).
\end{equation}
This converts \eqref{flux-dynamics} into a monotonicity formula as well, roughly of the form
\begin{equation}\label{flux-grow-heuristic}
 \frac{d}{dr} \F\geq 2 \RR- O(\delta).
\end{equation}
Now the dichotomy in Lemma \ref{ode} rests on whether the forcing terms such as $\RR$ or
$\frac{2}{r}\A$ in the monotonicity formulae are large enough to dominate the error terms
such as $O(\delta)$.  If this domination never occurs (or only occurs when $r$ is relatively small),
then one ends up in the ``boundedness'' scenario \eqref{bounded-mass}.  On the other hand, if at least one of the forcing
terms becomes large, one expects that this will eventually force the other forcing term to be large as well 
(as $r$ decreases towards $C_1$), causing a positive feedback loop which will eventually lead to the
``exponential growth'' scenario \eqref{mass-growth}.  For instance, if $\RR$ gets large (compared to $\M$
and $\delta$), this should force $\M$ to be similarly large thanks to \eqref{flux-grow-heuristic}; 
from \eqref{mass-dynamics} one then expects $\M$ to grow exponentially (but slightly less fast than
$\F$) as $r$ decreases; from \eqref{flux-cauchy} one then expects $\RR$ to stay large, thus creating a 
self-sustaining feedback loop.  Similarly, if $\A$ gets large, then from \eqref{cons-dynamics}
we expect the quantity $\RR+ \lambda \M - \A$ to get large and negative, which adds an additional
positive term to the right-hand side of \eqref{flux-grow-heuristic}, which as mentioned earlier should cause
$\F$, $\M$, and $\RR$ to grow; using \eqref{ang-grow-heuristic}, this should eventually sustain
the growth of $\A$, thus creating another self-sustaining feedback loop.  If these loops start far enough
away from the origin (e.g. at $r > C(C_1,C_2)$) then one might hope to expect the growth to become exponential
with growth rate $C_2$ by the time $r$ reaches $C_1$, which is the ``exponential growth'' half of the dichotomy.

\begin{remark} An oversimplified
model of this dichotomy can be seen by considering plane wave solutions to the Helmholtz equation $\Delta u = \lambda u$ in 
the (flat) cylinder $\R/2\pi\Z \times \R^+ := \{ (\theta,r): \theta \in \R/2\pi\Z, r \in \R^+ \}$, which should be thought of
as a caricature of polar co-ordinates; we assume some boundedness on $u$ at $r=+\infty$ (e.g. $u(r) = O(\delta)$ for
sufficiently large $r$) but not when $r$ is small.  One has 
``bounded'' solutions of the form $u(\theta,r) = C e^{ia\theta} e^{ibr}$ where $a$ is an integer such that
$|a|^2 \leq \lambda$ and $a^2 + b^2 = \lambda$ and $C = O(\delta)$.  Then there are ``exponential growth'' solutions of the form 
$u(x,y) = C e^{ia\theta} e^{-br}$ where $|a|^2 > \lambda$ and $a^2 - b^2 = \lambda$ and $C$ is arbitrarily large.
Thus one expects the solution to stay bounded if the ``angular energy'' $a^2$ stays smaller than $\lambda$,
and to grow exponentially otherwise.  This dichotomy roughly corresponds in our setting to
the case when $\A$ stays controlled by $\RR + \lambda \M$ (which will basically
ensure the bounded scenario) or is larger than this quantity (which will ensure the exponential growth
scenario).  See \cite{bz2} for some rigorous formulations of these heuristics, where the cylinder has now been
replaced by a stadium.
\end{remark}

\subsection{Step 1: a Pohozaev bound}

We now begin the rigorous proof of Lemma \ref{ode}.  
We begin by giving a rigorous version of \eqref{ang-grow-heuristic}, based primarily on \eqref{cons-dynamics} and the positivity of $\A$.

\begin{lemma}[Pohozaev bound]\label{flux-ceiling}  Let $\P$ denote the ``Pohozaev flux''
$$ \P[r] := \lambda \M[r] + \RR[r] - \A[r].$$
Then for all $r \geq C_1$, we have
$$\P[r] \leq O((r^{-1} + \lambda^{1/2})) \delta + O(r^{-2-2\sigma})\M[r].$$
\end{lemma}

\begin{proof}
We rewrite \eqref{cons-dynamics} in terms of the Pohozaev flux as
$$
\frac{d}{dr}\P= \frac{2}{r} \A + O( r^{-1-2\sigma_0})
(2\A + \P) + O( r^{-3-2\sigma_0}) \M + O( (r^{-1} + \lambda^{1/2})) \G,
$$
To eliminate the $O( r^{-3-2\sigma_0}) \M$ error\footnote{In dimensions $n \geq 4$ we can use the 
$\A$
term to control the $\M$ error since $(n-1)(n-3)/4 > 0$ in that case; this leads to some minor simplifications in the proof of Lemma \ref{ode}.} we shall consider the modified Pohozaev flux
$$ \P^*:= \P- C_0 r^{-2-2\sigma_0}\M$$
for constant $C_0 \gg 1$, and observe using \eqref{mass-dynamics} that
\begin{align*}
\frac{d}{dr} \P^* &= \frac{2}{r} \A + O( r^{-1-2\sigma_0})
(2\A + \P^* + C_0 r^{-2-2\sigma_0} \M) \\
&\quad + C_0 r^{-2-2\sigma_0} (\F + O(r^{-1-2\sigma_0}) \M )\\
&\quad + (2+2\sigma) C_0 r^{-3-2\sigma_0} \M + O( r^{-3-2\sigma_0}) \M
 + O( r^{-1} + \lambda^{1/2}) \G.
\end{align*}
We can use \eqref{flux-cauchy} to bound
$$ |\F| \leq r^{-1-\sigma_0} \M + r^{1+\sigma_0} \RR
=r^{-1-\sigma_0}(1+C_0) \M + 2r^{1+\sigma_0} \A + r^{1+\sigma_0} \P^*.$$
If $C_0$ is suitably large (and $C_1 \leq r$ is also suitably large) then the net $\M$ term
on the right-hand side is positive, as is the net $\A$ term.  Thus we have
$$ \frac{d}{dr} \P^*\geq O( C_0 r^{-1-\sigma_0})\P^* +
 O(r^{-1} + \lambda^{1/2}) \G.$$
On the other hand, from \eqref{boundary} we know that $\P^*[r] \to 0$ as $r \to \infty$.  The claim
then follows from \eqref{delta} and Gronwall's inequality.
\end{proof}

\subsection{Step 2: Dimensionless formulation}

In order to analyze our system further it is convenient to make a number of changes of variable to a more scale-invariant or ``dimensionless'' formulation.
At present we have three equations of motion
\eqref{mass-dynamics}, \eqref{flux-dynamics}, \eqref{cons-dynamics} and one inequality
\eqref{flux-cauchy} (as well as the positivity properties \eqref{mra}) for four unknowns.  We shall now use
the inequality \eqref{flux-cauchy} to replace two of the unknowns $\RR$, $\A$ by a single unknown
$\P$, at the cost of replacing the equalities in \eqref{flux-dynamics}, \eqref{cons-dynamics} 
by inequalities.  Indeed, we can rewrite \eqref{flux-dynamics} using \eqref{flux-cauchy} as
\begin{align*}
\frac{d}{dr} \F &= 2 \RR- \P +
O( r^{-2-2\sigma_0}) \M + O( \lambda^{1/2} r^{-1-2\sigma_0}) \M 
+ O( \G) \\
&\geq 2 \frac{\F^2}{\M} - \P 
+
O( r^{-2-2\sigma_0})\M + O( \lambda^{1/2} r^{-1-2\sigma_0}) \M 
+O( \G).
\end{align*}
As for \eqref{cons-dynamics}, we rewrite it (for $C_1$ sufficiently large) as
\begin{align*}
\frac{d}{dr} \P &= \frac{2}{r} \A + O( r^{-1-2\sigma_0})
(2\A+ \P) + O( r^{-3-2\sigma_0})\M  + O( r^{-1} + \lambda^{1/2}) \G \\ 
&= \left(\frac{2}{r} + O( r^{-1-2\sigma_0} ) \right) \A +
O( r^{-1-2\sigma_0}) \P + O( r^{-3-2\sigma_0}) \P + O(r^{-1} + \lambda^{1/2}) \G 
\end{align*}
and then write
$$ \A= \RR+ \lambda \M - \P
\geq \frac{\F^2}{\M} + \lambda \M - \P.$$
We thus have the new equations of motion
\begin{align*}
\frac{d}{dr} \M &= 2 \F+ O( r^{-1-2\sigma_0})\M \\
\frac{d}{dr} \F &\geq 2 \frac{\F^2}{\M} - \P +
O( r^{-2-2\sigma_0})\M + O( \lambda^{1/2} r^{-1-2\sigma_0})\M  + O( \G)\\
\frac{d}{dr} \P &\geq \left(\frac{2}{r} + O(r^{-1-2\sigma_0})\right)
(\frac{\F^2}{\M} + \lambda \M - \P)
+ O( r^{-1-2\sigma_0}) \P\\&\quad + O( r^{-3-2\sigma_0}) \M + O(r^{-1} + \lambda^{1/2}) \G.
\end{align*}
To analyze these equations, we now adjust the quantities $\F$, $\P$ slightly to handle the forcing
terms involving $\G$.  Define
$$ \F^*[r] := \F[r] - C \int_r^\infty |\G[s]|\ ds = \F[r] + O(\delta)$$
and
$$ \P^*[r] := \P[r] - C \int_r^\infty (s^{-1} + \lambda^{1/2}) |\G[s]|\ ds
= \P[r] + O(r^{-1} + \lambda^{1/2}) \delta,$$
where we have used \eqref{delta}.
If the constant $C$ appearing above is large enough, then we can dominate the $\G[r]$ 
forcing terms on the right-hand sides of the above equations to obtain
\begin{align*}
\frac{d}{dr} \M &= 2 \F + O( r^{-1-2\sigma_0})\M \\
\frac{d}{dr} \F^* &\geq 2 \frac{\F^2}{\M} - \P +
O( r^{-2-2\sigma_0})\M + O( \lambda^{1/2} r^{-1-2\sigma_0})\M\\
\frac{d}{dr} \P^* &\geq \left(\frac{2}{r} + O(r^{-1-2\sigma_0})\right)
(\frac{\F^2}{\M} + \lambda \M - \P)
+ O( r^{-1-2\sigma_0}) \P + O( r^{-3-2\sigma_0}) \M.
\end{align*}
Writing $\F = \F^* + O( \delta )$ and so $\F^2 \geq (\F^*)^2
+ O( \delta)\F^*$, and similarly writing $\P =\P^*
+ O(r^{-1} + \lambda^{1/2})\delta$, we thus obtain
\begin{align}
\frac{d}{dr} \M &= 2 \F^* + O( r^{-1-2\sigma_0})\M +O(\delta)\label{mass-new}\\
\frac{d}{dr} \F^* &\geq 2 \frac{(\F^*)^2+O(\de)\F^*}{\M} - \P^* +
O( 1+\la^{1/2} r) r^{-2-2\sigma_0}\M + O(r^{-1} + \lambda^{1/2})\delta
\label{flux-new}\\
\frac{d}{dr} \P^* &\geq \left(\frac{2}{r} + O(r^{-1-2\sigma_0})\right)
\left(\frac{(\F^*)^2+O(\de)\F^*}{\M} +\lambda \M - \P^*\right)
+ O( r^{-1-2\sigma_0}) \P^* \\&+ O( r^{-3-2\sigma_0}) \M+O(r^{-2}+\la^{1/2}r^{-1})\de.
\label{moraw-new}
\end{align}
To analyze this system of equations, it is convenient to work in the ``dimensionless'' co-ordinates\footnote{Admittedly,
these co-ordinates have a singularity when $\M[r] = 0$, but this will not be relevant for us as we shall only perform
the remainder of the analysis in the case when $\mu$ is small (and hence $\M$ is large).}
\begin{align*}
\mu(r) &:= r \frac{\delta}{\M[r]} \\
\alpha(r) &:= - r \frac{\F^*[r]}{\M[r]} \\
\beta(r) &:= - r^2 \frac{\P^*[r]}{\M[r]}
\end{align*}
and to introduce the ``dimensionless'' derivative $D := -r \frac{d}{dr}$.  We then have
\begin{align*}
D\mu &= - \mu + r^2 \frac{\delta}{\M^2} \frac{d}{dr} \M \\
&= - \mu - 2r^2 \frac{\delta}{\M^2} \F + O( r^{1-2\sigma_0}) \frac{\delta}{\M} 
+ O( r^2\frac{\delta^2}{\M^2} ) \\
&= -\mu - 2 \alpha \mu + O( r^{-2\sigma_0}) \mu + O( \mu^2)
\end{align*}
and
\begin{align*}
D\alpha &= -\alpha + \frac{r^2}{\M^2} (\M \frac{d}{dr} \F^*
- \F^* \frac{d}{dr} \M ) \\
&\geq - \alpha + \frac{r^2}{\M^2}
\left(2(\F^*)^2 + O( \delta) \F^* - \M \P^* 
+ O( 1 + r \lambda^{1/2}) r^{-2-2\sigma_0} \M^2 \right.\\&\quad\left. + O(r^{-1} + \lambda^{1/2}) \de \M
- 2 (\F^*)^2 + O( r^{-1-2\sigma_0})\M\F^*
+ O( \delta) \F^*\right) \\
&= \beta-\alpha + O( \mu) \alpha + O(1 + r \lambda^{1/2}) r^{-2\sigma_0}  + O(1 + \lambda^{1/2} r )\mu 
+ O( r^{-2\sigma_0}) \alpha
\end{align*}
and
\begin{align*}
D\beta &= -2\beta + \frac{r^3}{\M^2} (\M \frac{d}{dr} \P^*
- \P^*\frac{d}{dr} \M ) \\
&\geq -2\beta + \frac{r^3}{\M^2} \left( \frac{2}{r} + O( r^{-1-2\sigma_0} )
 \left((\F^*)^2 + O( \delta) \F^*
+ \lambda\M^2 -\M \P^*\right)\right. \\
&\quad \left.+ O( r^{-1-2\sigma_0})\M\P^*+ O( r^{-3-2\sigma_0})\M^2 
+ O(r^{-2} + \lambda^{1/2} r^{-1}) \de \M\right. \\
&\quad\left. - 2\P^*\F^* + O( r^{-1-2\sigma_0})\M\P^* 
+ O( \delta)\P^*\right)\\
&= -2\beta + (2 + O( r^{-2\sigma_0})) ( \alpha^2 + O( \mu) \alpha 
+ \lambda r^2 + \beta) + O( r^{-2\sigma_0}) \beta + O( r^{-2\sigma_0} )\\& \quad
+ O(1 + \lambda^{1/2} r) \mu 
 - 2 \alpha \beta + O( r^{-2\sigma_0}) \beta + O( \mu) \beta \\
&\geq 2\alpha(\alpha-\beta) + \lambda r^2 
+ O( \mu (1 + \lambda^{1/2} r + |\alpha| + |\beta|) ) + 
O( r^{-2\sigma_0} (1 + |\beta| + \alpha^2 ) ).
\end{align*}

Meanwhile, from \eqref{flux-ceiling} we have
$$ \P^* \leq O(r^{-1} + \lambda^{1/2}) \delta + O(r^{-2-2\sigma})\M $$
and hence
$$ \beta \geq - O(1 + \lambda^{1/2}r) \mu - O(r^{-2\sigma_0}).$$
To summarize, the functions $\alpha(r), \beta(r), \mu(r)$ obey the differential inequalities
\begin{align}
D\alpha &\geq \beta-\alpha - O( (\mu + r^{-2\sigma_0}) (1 + \lambda^{1/2}r + |\alpha| ) ) 
\label{alpha-eq}\\
D\beta  &\geq 2\alpha(\alpha-\beta) + \lambda r^2 - O( \mu (1 + \lambda^{1/2} r + |\alpha| + |\beta|) ) -
O( r^{-2\sigma_0} (1 + |\beta| + \alpha^2 ) ) \label{beta-eq}\\
D\mu &= -\mu - 2 \alpha \mu + O( r^{-2\sigma_0} \mu ) + O( \mu^2 )\label{mu-eq}\\
\beta &\geq - O(1 + \lambda^{1/2}r) \mu - O(r^{-2\sigma_0}).\label{beta-floor}
\end{align}

\subsection{More heuristics} 

Recall that we are trying to establish a dichotomy between boundedness and exponential growth.
In our new co-ordinates, boundedness roughly corresponds to $\mu$ being bounded away from zero for
small values of $r$, while exponential growth corresponds to $\alpha$ being large (as can be seen either from
\eqref{mu-eq} or \eqref{mass-dynamics}).  So, heuristically speaking, we have to rule out
the scenario in which $\mu$ is small and $\alpha$ is also small.  If $\mu$ is very small, however, then we
expect to be able to ignore most of the error terms in \eqref{alpha-eq}, \eqref{beta-eq}, and thus we reduce
(heuristically) to the model equations
\begin{equation}\label{model-eqs}
 D\alpha \geq \beta - \alpha; \quad D\beta \geq 2\alpha(\alpha-\beta) + \lambda r^2; \quad \beta \geq 0.
\end{equation}
The intuition here is that as $r$ moves backwards from $\infty$ to $C_1$, the first inequality will lift up
$\alpha$ if $\alpha$ is below $\beta$, while the second inequality will lift up $\beta$ if $\beta$ is below
$\alpha$ (note that this forces both $\alpha$ and $\beta$ to be non-negative, by the third equation); indeed,
when $\lambda r^2$ is large there some additional lift applied to $\beta$ (and thus indirectly to $\alpha$,
by the first equation).  These equations then suggest that if $\alpha$ and $\beta$ are both large at some
radius $r$, then they will stay large for all smaller radii also; this is what will lead to the exponential growth
scenario.  On the other hand, if $\alpha$ and $\beta$ stay small for all time then we can hope to obtain the boundedness
scenario.

\subsection{Step 3: A condition for igniting exponential growth}

We now make the above heuristics rigorous.  

\begin{lemma}[Exponential growth is self-sustaining]\label{exp-growth}  Suppose that $r \geq C_1$ is such that
$$ \mu(r) \leq 1/C_1 \hbox{ and } \alpha(r) \geq C_2^2.$$
Then for all $C_1 \leq s \leq r$ we have
\begin{equation}\label{mush}
 \mu(s) \leq 1/C_1 \hbox{ and } \alpha(s) \ge 2C_2.
\end{equation}
\end{lemma}

\begin{proof}  
From \eqref{mu-eq} we have the somewhat crude estimate
\begin{equation}\label{mu-eq2}
D\mu(s) \leq -\frac{1}{2} \mu(s) \hbox{ whenever } \alpha(s) \geq 0 \hbox{ and } \mu(s) \leq 1/C_1.
\end{equation}
From \eqref{mu-eq2} and the continuity method it will thus suffice to show that
$$ \alpha(s) > 2C_2 \hbox{ for all } C_1< s \leq r.$$
Suppose this claim is false, then there exists $C_1 < s_* < r$ such that
\begin{equation}\label{alpha-sod}
 \alpha(s_*) = 2C_2 
\end{equation}
and
\begin{equation}\label{alpha-large}
 \alpha(s) \ge 2C_2 \hbox{ for all } s_* \leq s \leq r.
\end{equation}
From \eqref{mu-eq2} and Gronwall's inequality (and the continuity method) we conclude in particular that
\begin{equation}\label{mu-decay}
 \mu(s) \leq (s/r)^{1/2} \mu(r) \leq (s/r)^{1/2} / C_1 \leq 1/C_1 \hbox{ for all } s_* \leq s \leq r.
\end{equation}
It is convenient to introduce the quantity
$$ \kappa(s) := \alpha^2(s) + \beta(s).$$
From \eqref{alpha-eq}, \eqref{beta-eq}  we have
\begin{align*}
D\kappa(s) &= 2\alpha(s) D\alpha(s) + D\beta(s) \\
&\geq -O( (\mu(s) + s^{-2\sigma_0}) (\alpha(s) + \lambda^{1/2} s \alpha(s) + \alpha^2(s)) ) \\
&\quad + \lambda s^2 - 
O( \mu(s) (1 + \lambda^{1/2} s + \alpha(s) + |\beta(s)|) ) - O( s^{-2\sigma_0} (1 + |\beta(s)| + \alpha^2(s) ) ) \\
&\geq \lambda s^2 + O( (\mu(s) + s^{-2\sigma_0}) (1 + \alpha(s) + \alpha^2(s) + \lambda^{1/2} s+\lambda^{1/2} s \alpha(s) + |\beta(s)| ) )\\
&\geq \lambda s^2 - O( (\mu(s) + s^{-2\sigma_0}) (1 + \alpha^2(s) + \lambda s^2 + |\beta(s)| ) )\\
&\geq (1 - O(\mu(s) + s^{-2\sigma_0})) \lambda s^2 
- O( (\mu(s) + s^{-2\sigma_0}) (1 + \kappa(s) + 2\min(-\beta(s),0) ) ).
\end{align*}
But from \eqref{beta-floor} we have
$$ \min(-\beta(s),0) \leq C \mu  ( 1 + \lambda s^2 ) + C s^{-2\si_0}$$
and thus
$$D \kappa(s) \geq (1 - O((\mu + s^{-2\sigma_0})(1+\mu))) \lambda s^2 
- O( (\mu(s) + s^{-2\sigma_0}) (1 + \kappa(s) + \mu(s)) ).$$
Using \eqref{mu-decay}, we conclude
\begin{equation}\label{kappa-eq}
D\kappa(s) \geq \frac{1}{2} \lambda s^2 - O( (C_1^{-1}(s/r)^{1/2} + s^{-2\sigma_0}) (1 + \kappa(s)))
\hbox{ for all } s_* \leq s \leq r.
\end{equation}
On the other hand, from \eqref{beta-floor} we have the crude estimate
$$ \beta(r) \geq - O(1 + \lambda^{1/2}r)$$
and hence
$$ \kappa(r) \geq  C^{-1} C_2^4- C\la^{1/2} r.$$
From this, \eqref{kappa-eq}, and Gronwall's inequality we see that
\begin{equation}\label{kc}
 \kappa(s) \geq C^{-1} (\lambda r(r-s) + C_2^4- C\la^{1/2}r)
\hbox{ for all } s_* \leq s \leq r.
\end{equation}
On the other hand, from \eqref{alpha-eq} and \eqref{alpha-large}, and writing $\beta = \kappa - \alpha^2$, we
have the rather crude estimate
\begin{equation}\label{das}
D\alpha(s) \geq \kappa(s) - O(C_1^{-1}(1 +\lambda^{1/2} s)) - O(\alpha^2(s)) \hbox{ for all } s_* \leq s \leq r.
\end{equation}
From \eqref{alpha-sod} we have $\alpha(s_*) < 3C_2 < \alpha(r)$.  Thus we can find $s_* < r_* < s$ such that
$\alpha(r_*) = 3C_2$ and $2C_2 < \alpha(s) < 3C_2$ for all $s_* < s < r_*$.  Then by \eqref{kc}, \eqref{das}
we have
$$ D\alpha(s) \geq C^{-1} C_2^2 + C^{-1} \lambda r(r-s) - C \lambda^{1/2} r$$
Observe that the expression on the right is positive unless $r \geq C_2 \lambda^{-1/2}$ and
$s = r + O( C_1\lambda^{-1/2} )$, in which case it is bounded below by $- C \lambda^{1/2} r$.  From this and
the fundamental theorem of calculus we see that
$$ -C_2 = \alpha(r_*) - \alpha(s_*) \geq \int_{r_*}^{s_*} D\alpha(s)\ \frac{ds}{s} \geq - C_1,$$
a contradiction.  The claim follows.
\end{proof}

\begin{corollary}  There exists $C_3 = C(C_1,C_2) \gg C_2$ such that if 
there exists $r \geq C_3$ for which $\mu(r) \leq 1/C_1$ and $\alpha(r) \geq C_2^2$
are both true, then
we are in the exponential growth scenario \eqref{mass-growth}.
\end{corollary}

\begin{proof}  Suppose first that we are in the low energy case $\lambda \leq 1$.  Then
from the above Lemma we have $\alpha(r) \geq 2C_2$ for all $C_1 \leq r \leq 10C_1$.
From \eqref{mu-eq} we then have
$$ D\mu(s) \leq - 2C_2 \mu(s) \hbox{ for all } C_1 \leq s \leq 10C_1$$
which by definition of $\mu(s)$, yields the mass growth estimate \eqref{mass-growth}.

Now suppose that we are in the high energy case $\lambda > 1$.  In this case we observe from \eqref{das}, \eqref{kc}
(with $r$ now being replaced by $C_3$) that
$$
D\alpha(s) \geq C^{-1} (\lambda C_3(C_3-s) + C_2^4) - C C_3\lambda^{1/2}  - O(\alpha^2(s)) 
\hbox{ for all } C_1 \leq s \leq C_3$$
and in particular
$$
D\alpha(s) \geq C^{-1} \lambda C_3^2 \hbox{ for all } C_1 \leq s \leq C_3/2 
\hbox{ such that } \alpha(s) \leq C_2^2 \lambda^{1/2}.$$
We also have $\alpha(s) \geq 2C_2$ in this region.  We thus conclude (if $C_3$ is large) that
$$ \alpha(s) \geq C_2^2 \lambda^{1/2} \hbox{ for all } C_1 \leq s \leq 10C_1$$
and then by arguing as before we obtain \eqref{mass-growth}.
\end{proof}

\subsection{Step 4: The case of no exponential growth}

In light of the above corollary, we may assume without loss of generality that for any $r \geq C_3$, at least one of
\begin{equation}\label{mass-bound}
\mu(r) \geq 1/C_1 \hbox{ or } \alpha(r) \leq C_2^2
\end{equation}
is true.

We can now remove the second half of of the dichotomy \eqref{mass-bound} at large distances.

\begin{lemma}\label{mbl} Suppose that \eqref{mass-bound} holds.  Let $C_4 = C(C_1,C_2,C_3) \gg C_3$ be a sufficiently large
constant depending on $C_1$, $C_2$ and $C_3$.  Then we have $\mu(r) > 1/C^2_4$ for all
$r \ge C_4 (1 + \lambda^{-1/2})$.
\end{lemma}

\begin{proof} By \eqref{boundary} we have $\mu(r) \to \infty$ as $r \to \infty$.  Thus if the claim is false,
then we can find $r_1 \geq r_0 \geq C_4  (1 + \lambda^{-1/2})$ such that
\begin{align}
\mu(r_0) &= C_4^{-2}\label{r0-bound}\\
\mu(r_1) &= C_4^{-1} \label{r1-bound}\\
C_4^{-2} \leq \mu(r) &\leq C_4^{-1} \hbox{ for all } r_0 \leq r \leq r_1 \label{r-inter}.
\end{align}
In particular from \eqref{mass-bound} we have
\begin{equation}\label{mb-2}
\alpha(r) \leq C_2^2 \hbox{ for all } r_0 \leq r \leq r_1 
\end{equation}
From \eqref{r1-bound} we have
$$ D \mu(r_1) \leq 0$$
which by \eqref{mu-eq}, \eqref{r1-bound} forces
\begin{equation}\label{aro} \alpha(r_1) \geq -1
\end{equation}
(for instance).  Also, from \eqref{mu-eq}, \eqref{r-inter}, \eqref{mb-2} we have
$$
D \mu(r)  \ge - 2 C_2^2  \mu(r)
$$
which by Gronwall's inequality and \eqref{r0-bound}, \eqref{r1-bound} forces a certain largeness bound
in the interval $[r_0,r_1]$:
$$ \int_{r_0}^{r_1} \frac{dr}{r} \geq C^{-1} C_2^{-2} \log C_4,$$
which implies that $ r_1 \geq r_0 \, C_4^{C^{-1} C_2^{-2}}$.
In particular, if $C_4$ is large enough then we have
\begin{equation}\label{r-large} 
r_1 - r_0 \,\ge \,  C_4 (1+\la^{-1/2}),
\end{equation}
(for instance).
Now we control $\beta$.  If $r_0 \leq r \leq r_1$, then $r \geq C_4 \lambda^{-1/2}$ and
$\mu(r) \leq C_4^{-1}$, which by \eqref{beta-eq} implies the crude bound
$$
D\beta  \geq C^{-1} \alpha^2 - O( (1 + |\alpha|) (1 + |\beta|) ) + C^{-1} \lambda r^2;$$
using the crude bound $O( (1 + |\alpha|) (1 + |\beta|) ) = O( 1 + C^{-1} |\alpha|^2 + C |\beta|^2)$
we conclude
$$
D\beta  \geq C^{-1} \lambda r^2 \hbox{ whenever } |\beta(r)| \leq C^{-1} r \lambda^{1/2}.$$
On the other hand, from \eqref{beta-floor} we have
$$ \beta(r) \geq - C \lambda^{1/2}r \hbox{ whenever } r_0 \leq r \leq r_1.$$
Combining these two equations we see that
$$ \beta(r) \geq C^{-1} \lambda^{1/2} r \hbox{ whenever } r_0 \leq r \leq r_1 - C \lambda^{-1/2}.$$
On the other hand, from \eqref{alpha-eq} and the estimates $r \geq C_4 \lambda^{-1/2}$,
$\mu \leq C_4^{-1}$ we have
$$
D\alpha \geq \beta- O(\alpha) - O( C_4^{-2\sigma_0} \lambda^{1/2} r ),$$
and thus we have
$$ D\alpha(r) \geq C^{-1} r \lambda^{1/2} - O(\alpha(r)) \hbox{ whenever } r_0 \leq r \leq r_1 - C \lambda^{-1/2} $$
and
$$ D\alpha(r) \geq - O( C \lambda^{1/2} r ) - O(\alpha(r))
\hbox{ whenever } r_1 - C \lambda^{-1/2}  \leq r \leq r_1.$$
Moreover, since $r\ge C_4\la^{-1/2}$ and $\alpha(r)\le C_2^2$, we can replace the first
estimate by
$$ D\alpha(r) \geq C^{-1} r \lambda^{1/2}\hbox{ whenever } r_0 \leq r \leq r_1 - C \lambda^{-1/2} $$
From these estimates and the initial condition \eqref{aro}, 
applying Gronwall's inequality, we see  that
$$ \alpha(r) \geq - C \hbox{ whenever } r_1 - C \lambda^{-1/2} \leq r \leq r_1$$
and then by a further application of Gronwall we see that
$$  \alpha(r) \geq C^{-1} (r_1-r) \lambda^{1/2} - O(1) \hbox{ whenever } r_0 \leq r \leq r_1 - C \lambda^{-1/2}.$$
Since $r_1-r_0\ge C_4 \la^{-1/2}$ 
this contradicts \eqref{r-large} and \eqref{mb-2}, and the claim follows.
\end{proof}

In the high-energy case $\lambda \geq 1$ this lemma immediately gives the boundedness half \eqref{bounded-mass}
of the dichotomy, by choosing $r := 4C_4$ (for instance) and using the definition of $\mu$.  In the low-energy case one observes from 
\eqref{mass-bound} and \eqref{mu-eq} that
$$ D\mu(r) \ge -2C_2^2 \mu(r) $$
whenever $C_3 \leq r \leq C_4(1 + \lambda^{-1/2})$ and $\mu(r) \leq 1/C_1$.  From this, Lemma \ref{mbl}, and Gronwall's
inequality we see that
\begin{equation}\label{chump-change}
 \mu(r) \geq C(C_1,C_2,C_3,C_4)^{-1} \lambda^{C(C_1,C_2)}
\end{equation}
for all $C_3 \leq r \leq C_4(1 + \lambda^{1/2})$, and by definition of $\mu$ we are again in the boundedness
half \eqref{bounded-mass} of the dichotomy.  The claim follows.  This completes the proof of
Lemma \ref{ode}, and thus Lemma \ref{dichotomy}.
\endprf

\section{Additional equations of motion}\label{sec:add-motion}

As we have already seen in Section \ref{main-proof},
the combined results of Lemma \ref{charge-est}, Lemma \ref{elliptic-regularity}, Lemma \ref{morawetz-lemma},  Proposition \ref{ucp}, and Lemma \ref{dichotomy} are already sufficient to prove Theorem \ref{main}.  But to prove the Sommerfeld radiation conditions of Proposition \ref{qlap}, it turns out that we must study further spherical energies in addition to the energies $\M,\RR,\A,\F$ introduced in  Section \ref{bessel-sec}. We take a quick detour to define these energies, study 
additional equations of motions and derive their consequences, in particular establishing Lemma \ref{lem:improvement}. The results of this
section  will also be useful in the proof of Proposition \ref{lowenergy-pos}.

Using the notations of Section \ref{bessel-sec}, we define the following additional energies:
\begin{align*}
{\text{Outgoing null energy}}\qquad &\N[r] := \int_{S_r} |v_r \mp i \k v|^2\ dh[r],\\
{\text{Complex flux}}\qquad & \Zflux[r] := \int_{S_r} v_r \overline v \, dh[r],
\end{align*}
where $\k=a\pm ib$ is the complex number such that
 $\k^2 = \lambda\pm i\eps$ and $\Im(\k)=\pm b > 0$. 
 
\begin{lemma}[Additional equation of motion] \label{lem:Out}
Let $u \in H^2(M)$ be a solution to the resolvent equation
$$ (H - (\lambda \pm i\eps)) u = f$$
for some $\lambda, \eps > 0$ with the Hamiltonian $H=-\Delta_M +V$,
satisfying assumptions of Proposition \ref{qlap}.
Then (with the notation of the previous two sections)
\begin{equation}\label{eq:add-motion}
\begin{split}
\frac{d}{dr} (\N - \A) &= \frac{2}{r} \A+ 
2 b \,\,(\N+\A) +O(r^{-2-2\si_0}) b\M\\ &+  O\Big (r^{-1-2\si_0} 
(\RR+\A+|z|^2\M)\Big ) + O(F[r]^{1/2} \N^{1/2})\\
&+O\big(r^{-2-2\si_0}+\la^{1/2} r^{-1-2\si_0}\big)\M^{1/2} \N^{1/2}
\end{split}
\end{equation}
where $F[r] := r^{n-1} \int_{\pa M} |f|^2 dh[r]$.
\end{lemma}

\begin{remark}
A similar equation of motion appeared in \cite{iu} for the proof of a qualitative
limiting absorption principle for a Sch\"odinger operator with magnetic potential in
$\R^n$. This equation can be interpreted as a special 
case of the identity \eqref{eq:radiat}.
\end{remark}

\begin{proof}  
For simplicity assume that $(H-\la-i\eps) u =f$. Note that $\F=\Re\, \Zflux$. 
We have as in Lemma \ref{equations-of-motion}
\begin{align*}
&\frac d{dr}\M = 2 \Re \Zflux+ O(r^{-1-2\si_0} \M),\\
&\frac d{dr}\Zflux = \RR + \A-
z^2\M +O(r^{-2-2\si_0}) \M + 
 r^{\frac {n-1}2}\int_{S_r} (f+V v) \overline v dh[r]
\end{align*}
Note that 
\begin{align*}
|v_r-izv|^2 &= |v_r|^2+|z|^2 |v|^2 + 2\Im (z v \overline {v_r})\\
&=
|v_r|^2+z^2 |v|^2+ (|z|^2 -z^2)|v|^2 + 2\Im (z v \overline {v_r}).
\end{align*}
Thus,
\begin{align*}
&\N = \RR + |z|^2 \M + 2\Im (z \overline \Zflux).
\end{align*}
Since 
\begin{align}
\frac d{dr}\Big (\RR-\A\Big ) &= \frac 2r \A - 2\Re (z^2\overline \Zflux)  + 
O(r^{-1-2\si_0})(\RR+\A) \nonumber\\ 
&\quad + 2\int_{S_r} 
\Re \big((f+Vv)\overline {v_r}\big)  dh[r]\label{eq:motion}
\end{align}
we derive the following equation for $\N-\A$:
\begin{align*}
\frac d{dr}\Big (\N-\A\Big )  &
= \frac 2r \A + 2\Re \left( (|z|^2 - z^2)\overline \N \right)+ 
2\Im (z) (\RR+\A) \\ 
&\quad + 2|z|^2\, \Im (z) \M  + 
O(r^{-1-2\si_0})(\RR+\A+|z|^2\M) + O(F[r]^{1/2} \N^{1/2})\\
&\quad +O(r^{-2-2\si_0}) \Im (z)\M+O\big(r^{-2-2\si_0}+\la^{1/2} r^{-1-2\si_0}\big)\M^{1/2} \N^{1/2}\\
&= \frac 2r \A+ 2\Re \left((|z|^2-z^2) \overline \Zflux\right)- 4\Im (z)\, \Im (z\overline \Zflux)\\ 
&\quad + 
2\Im (z)\, (\N+\A) 
+O(r^{-1-2\si_0}) (\RR+\A+|z|^2\M)+O(r^{-2-2\si_0}) \Im (z)\M \\ 
&\quad + 
O(F[r]^{1/2} \N^{1/2})
+O\big(r^{-2-2\si_0}+\la^{1/2} r^{-1-2\si_0}\big)\M^{1/2} \N^{1/2}
\end{align*}
Since
$$
2\Re \left((|z|^2-z^2)\overline \Zflux\right)=4\Im (z)\,\Re (-iz\overline \Zflux) =
4\Im (z)\,\Im (z\overline \Zflux), 
$$
we have
\begin{align*}
\frac d{dr}(\N-\A) &= \frac 2r \A+ 
2 b \,\,(\N+\A) +  O(r^{-1-2\si_0}) (\RR+\A+|z|^2\M)+O(r^{-2-2\si_0}) b\M\\& + O(F[r]^{1/2} \N^{1/2})
+O\big(r^{-2-2\si_0}+\la^{1/2} r^{-1-2\si_0}\big)\M^{1/2} \N^{1/2}.
\end{align*}
\end{proof}

\subsection{Proof of Lemma \ref{lem:improvement}}\label{improvement-sec}

We are now ready to establish Lemma \ref{lem:improvement}.  Let the notation and assumptions be as in that lemma.

Multiplying equation \eqref{eq:add-motion} for 
$(\N-\A)$ by $r^{2\si}$ we obtain 
\begin{align*}
\frac d{dr} r^{2\si}(\N-\A) &= \frac 1{r^{1-2\si}} 
\big ( (2-2\si) 
\A+\N \big )+
2 b \, r^{2\si}\,\,(\N+\A) \\ &+  O(r^{-1-2\si_0+2\si}) (\RR+\A+|z|^2\M) + 
r^{2\si} O(F[r]^{1/2} \N^{1/2})\\ 
&+O(r^{-2-2\si_0+2\si}) b\M+
O\big(r^{2\si}(r^{-2-2\si_0}+\la^{1/2} r^{-1-2\si_0})\big)\M^{1/2} \N^{1/2}.
\end{align*}
The presence of the $2 b \, r^{2\si}\,\,(\N+\A)$ term will allow us to ignore 
the boundary term at $r=\infty$ arising after integration in $r$. Therefore,
for some universal positive constant $c$.
\begin{align}
c\int_{2r_0}^\infty (r^{-1+2\si}&+b\, r^{2\si})\,\,(\N +\A)\, dr  
\le   \int_{r_0}^\infty r^{-1-2\si_0+2\si} 
 \big (\RR+|z|^2 \M \big )\, dr \nonumber\\ &
+ r_0^{-1+2\si} \int_{r_0}^{4r_0} \A\,dr+
 \int_{r_0}^\infty r^{-3-2\si_0+2\si} \M\, dr +
  \int_{r_0}^\infty r^{1+2\si} F[r]\, dr. \label{eq:out-energy}
\end{align}
Note that in view of Lemma \ref{morawetz-lemma}, the bound \eqref{eq:out-energy}
already gives the estimate \eqref{eq:urad-better} in the region $\eps\le C\la$.

We now claim that $u$ verifies the Poincar\'e type inequality
\begin{equation}\label{eq:poinc-u}
\|u\|^2_{H^{0,-3/2+\si}(M_{2r_0})}+b \|u\|^2_{H^{0,-1+\si}(M_{2r_0})}\le C(r_0)
\left (\int_{2r_0}^\infty \N[r] + \|u\|^2_{L^2(M_{2r_0}\setminus M_{4r_0})}\right).
\end{equation}
Note that \eqref{eq:poinc-u} together with \eqref{eq:out-energy} and 
Lemma \ref{morawetz-lemma} imply all the statements of Lemma \ref{lem:improvement}
for $\eps\le C\la$.

To establish \eqref{eq:poinc-u}, observe that 
$$
\N[r]=\int_{S_r}|v_r-iz v|^2 dh[r]= \int_{S_r}e^{-2br} |\pa_r(e^{-izr}v)|^2 dh[r]
$$
Our assumptions on the metric $g$ imply that the area form 
$dh[r]$, which corresponds to the metric 
$(h_{ab}(\omega)+r^{-2\si_0} e_{ab}(r,\omega))d\omega^a\,d\omega^b$, is 
equivalent to the area form $dh$ of the metric $h_{ab}(\omega)d\omega^a\,d\omega^b$.
Integrating by parts the expression
\begin{align*}
\int_{2r_0}^\infty dr \,\,r^{-2+2\si}((1-\si)r^{-1}&+b)
\int_{S_r}|v|^2 dh \\
&= \int_{2r_0}^\infty dr \,\,r^{-1+2\si}e^{-2br}
\int_{S_r} \Re \left (\pa_r(e^{-izr}v) (\overline{e^{-izr}v})\right) dh\\ 
&\quad + 
\frac 12 (2r_0)^{-2+2\si}((1-\si)(2r_0)^{-1}+b)\int_{S_{2r_0}}|v|^2 dh\\
&\le 2(1-\si)^{-1}\int_{2r_0}^\infty dr \,\,r^{-1+2\si}e^{-2br}
\int_{S_r} |\pa_r(e^{-izr}v)|^2 dh \\ 
&\quad +\frac 12(1-\si)
\int_{2r_0}^\infty dr \,\,r^{-3+2\si}
\int_{S_r}|v|^2 dh \\ 
&\quad + \frac 12 (2r_0)^{-2+2\si}((1-\si)(2r_0)^{-1}+b)\int_{S_{2r_0}}|v|^2 dh
\end{align*}
and averaging over $r_0$ we immediately obtain \eqref{eq:poinc-u}.

It remains to consider the case $\eps\ge C\la$. Under this assumption we have 
$|z|\le 2\eps$ (for instance). In view of \eqref{eq:poinc-u} and a trivial inequality 
$$
\N\ge \frac 12\RR-7 |z|^2\M
$$
the bound \eqref{eq:out-energy} can be reduced to the estimate
\begin{align*}
\|\nabla u\|_{H^{0,-1/2+\si}(M_{2r_0})}&+ \|u\|_{H^{0,-3/2+\si}(M_{2r_0})}
\le C |z|  \|u\|_{H^{0,-1/2-\si_0+\si}(M_{2r_0})} \\&+ C(r_0)\left( \|u\|_{L^2(M_{r_0}\setminus M_{4r_0})}
+  \|\nabla u\|_{L^2(M_{r_0}\setminus M_{4r_0})}+\|f\|_{H^{0,1/2+\si}(M_{r_0})}\right).
\end{align*}
According to \eqref{eq:eps-charge-weight}
$$
\eps  \|u\|_{H^{0,-1/2-\si_0+\si}(M_{2r_0})}
 \leq   C  \|f\|_{H^{0,-1/2-\si_0+\si}(M_{r_0})} + C
\|\nabla u \|_{H^{0,-3/2-\si_0+\si}(M_{r_0})}, 
$$
which implies that 
\begin{align*}
\|\nabla u\|_{H^{0,-1/2+\si}(M_{2r_0})}&+ \|u\|_{H^{0,-3/2+\si}(M_{2r_0})}
\leq C |z|  \|u\|_{H^{0,-1/2-\si_0+\si}(M_{2r_0})} \\&+ C(r_0)\left( \|u\|_{L^2(M_{r_0}\setminus M_{4r_0})}
+  \|\nabla u\|_{L^2(M_{r_0}\setminus M_{4r_0})}+\|f\|_{H^{0,1/2+\si}(M_{r_0})}\right).
\end{align*}
provided that $r_0$ is sufficiently large. The bounds \eqref{eq:u-better}, \eqref{eq:unab-better}
follow immediately, while \eqref{eq:urad-better} can be recovered from \eqref{eq:out-energy}.
The proof of Lemma \ref{lem:improvement} is now complete.

\section{The low energy regime}\label{sec:low-energy}

We now prove Proposition \ref{lowenergy-pos}. We assume that the 
potential $V$ obeys the bounds 
$$
|V(x)|\le A \langle x\rangle^{-2-2\si_0},\qquad 
\int_{M} |V_-(x)|^{\frac n2} \le \beta
$$
with a small constant $\beta=\beta(M)$
and show that for all sufficiently small $\eps, |\la|<\la_0(M)$
and $\si<\min(1,\si_0)$,
\begin{align}
&\|R(\la\pm i\eps) f\|_{H^{2,-1/2-\si}(M)}\le C(M,A) \la^{-1/2} \|f\|_{H^{0,1/2+\si}(M)},
\label{eq:low-quant}\\
&\|R(\la\pm i\eps) f\|_{H^{2,-3/2+\si}(M)}\le C(M,A)  \|f\|_{H^{0,1/2+\si}(M)}.
\label{eq:low-quant2}
\end{align}
We start by recalling the \emph{Friedrichs inequality}:
\begin{equation}\label{fried}
\|u\|_{L^2(M)}\le  C_s |{\text{supp}} u|^{\frac 1n} \||\nabla u\|_{L^2(M)},
\end{equation}
which holds for any smooth function of compact support on $M$.
The Friedrichs inequality is a direct consequence (by H\"older's inequality)
of the Sobolev inequality (with the same constant $C_s$):
\begin{equation}\label{eq:Sobolev}
\|u\|_{L^{\frac {2n}{n-2}}(M)}\le C_s \|\nabla u\|_{L^2(M)},
\end{equation}
which holds for all smooth functions of compact support on $M$ with 
the constant $C_s=C_s(M)$ is related to the the isoperimetric constant ${\mathcal I}(M)$
for $M$:
$$
{\mathcal I}(M) = \inf_N \frac {A(\partial N)^n}{|N|^{n-1}},\qquad
C_s^{-1}= {\mathcal I}(M)^{\frac 2n} \Big (\frac {n-2}{2(n-1)}\Big )^2.
$$
The infimum above ranges over all open submanifolds $N$ with compact closure 
and smooth boundary 
$\pa N$, and $A(\pa N)$ denotes the area of $\pa N$.  A proof of this Sobolev inequality can be found for instance in \cite{chavel}.

Let $u:=R(\la \pm i\eps) f$, thus
$$
(-\Delta_M + V(x) -\la\mp i\eps) u =f.
$$
Let $v=\Re \, u$ then 
\begin{equation}\label{eq:for-u}
(-\Delta_M + V(x))v =F = \Re (f+(\la\pm i\eps) u).
\end{equation}
We normalize 
\begin{equation}\label{eq:normalization}
\|u\|_{H^{0,-3/2+\si}(M)}=1,\qquad 
\|f\|_{H^{0,1/2+\si}(M)}=\de
\end{equation}
with a small constant $\de$.  We will then show that a sufficiently small $\la_0$ and 
all $|\la|, \eps<\la_0$
\begin{equation}\label{eq:sufficient}
\|u\|_{H^{0,-3/2+\si}(M)}\le \frac12 + C(M,A) \de,
\qquad \|u\|_{H^{0,-1/2-\si}(M)}\le C(M,A)\la^{-1/2}  \de.
\end{equation}
By Lemma \ref{lem:improvement} it will suffice to show that for a sufficiently 
large\footnote{In what follows the set $K$ will be fixed and the small constants $\la_0, \de$ will
be allowed to depend on $K$.} compact set $K$
$$
\|\nabla u\|_{L^2(K)}+ \|u\|_{L^2(K)}\le (2C_K)^{-1}+ C(M,A) \de.
$$
with the constant $C_K$ from \eqref{eq:u-better}.
Let $\mu>0$ be a small constant to be determined later and 
set
$$
\mu_\ell=\mu+\frac 12 (1+\ga)^{-\ell}\mu  
$$
Define a sequence of increasing nested sets  
$$
K_{\ell}:=B(2^\ell R_0)\cap \{x:\,\,v(x)\ge \mu_\ell\}
$$
Let $\chi_\ell$ be smooth cut-off functions on $M$ adapted to $B(2^\ell R_0)$, i.e.,
$\chi_\ell(x)=1$ for $x\in B(2^\ell R_0)$ and $\chi_\ell(x)=0$ for $x\in B^c(2^{\ell+1}R_0)$. 

Multiplying \eqref{eq:for-u} by\footnote{Here of course we use the notation $x_+ := \max(x,0)$.} $\chi_\ell^2 (v-\mu_\ell)_+$ and integrating by parts 
we obtain
\begin{align*}
\int_M \chi_\ell^2 \big (|\nabla (v-\mu_\ell)_+|^2 + V(x) u (v-\mu_\ell)_+\big )dg &= 
-2 \int_M \chi_\ell \nabla\chi_\ell  \cdot \nabla (v-\mu_\ell)_+ (v-\mu_\ell)_+\ dg \\ &+ 
 \int_M \chi_\ell^2  F (v-\mu_\ell)_+\ dg
\end{align*}
By Cauchy-Schwarz and a simple rearrangement, 
\begin{align*}
\int_M \chi_\ell^2 \big (|\nabla (v-\mu_\ell)_+|^2 &+ V_+(x) v (v-\mu_\ell)_+\big )dg \le 
4\int_M |\nabla\chi_\ell|^2  (v-\mu_\ell)^2_+\ dg \\ &+ 
2 \int_M \chi_\ell^2 |F| (v-\mu_\ell)_+\ dg
 \\&+ 2\mu_\ell \int_M \chi_\ell^2  V_-(x) (v-\mu_\ell)_+ \ dg + 
 2\int_M \chi_\ell^2 V_-(x) (v-\mu_\ell)_+^2 dg
\end{align*}
Using the smallness assumption on the negative part of the potential $V$ we have 
\begin{align*}
 \int_M \chi_\ell^2 V_-(x) (v-\mu_\ell)_+^2 dg &\le \beta
\Big ( \int_M  (\chi_\ell  (v-\mu_\ell)_+^2)^{\frac{2n}{n-2}}\Big )^{\frac {n-2}n}\\&\le
\beta C \Big (\int_M\bigl( \chi_\ell^2 |\nabla (v-\mu_\ell)_+|^2 + |\nabla\chi_\ell|^2  (v-\mu_\ell)^2_+\bigr) \ dg\Big ),
\end{align*}
where the last line follows from the Sobolev inequality \eqref{eq:Sobolev}.
Also, choosing a sufficiently small (universal) constant $a$,
\begin{align*}
\mu_\ell \int_M& \chi_\ell^2 V_-(x) (v-\mu_\ell)_+ dg \le \mu_\ell\beta |K_{\ell+1}|^{\frac {n-2}{2n}}\
\Big ( \int_M  (\chi_\ell  (v-\mu_\ell)_+^2)^{\frac{2n}{n-2}}\Big )^{\frac {n-2}{2n}}\\&\les
a^{-1}C \mu_\ell^2 \beta^2 |K_{\ell+1}|^{\frac {n-2}{n}}+a
 \Big( \int_M (\chi_\ell^2 |\nabla (v-\mu_\ell)_+|^2 + |\nabla\chi_\ell|^2  (v-\mu_\ell)^2_+)\ dg\Big ).
\end{align*}
Similarly, by Cauchy-Schwarz,
$$
 \int_M \chi_\ell^2 |F| (v-\mu_\ell)_+ dg \le 
 a^{-1} |K_{l+1}|^{\frac 2n} \int_M \chi_\ell^2 |F|^2 dg +
a |K_{l+1}|^{-\frac 2n} \int_M \chi_\ell^2 (v-\mu_\ell)_+^2 dg
$$
Thus, assuming that $\beta, a$ are sufficiently small,
\begin{equation}\label{eq:iterate}
\begin{split}
c\int_M \chi_\ell^2 |\nabla (v-\mu_\ell)_+|^2  dg &\le 
\int_M |\nabla\chi_\ell|^2  (v-\mu_\ell)^2_+\ dg+
a |K_{l+1}|^{-\frac 2n} \int_M \chi_\ell^2 (v-\mu_\ell)_+^2 dg \\ &+ 
 a^{-1} |K_{l+1}|^{\frac 2n} \int_M \chi_\ell^2 |F|^2 dg 
 +a^{-1}\mu_\ell^2 \beta^2 |K_{\ell+1}|^{\frac {n-2}{n}}.
\end{split}
\end{equation}
for some universal constant $c>0$, which depends only on the Sobolev constant $C_s$.

By Friedrichs' inequality \eqref{fried}, we have
$$
\int_M \chi_{\ell}^2  (v-\mu_{\ell})^2_+ \les
|K_{\ell+1}|^{\frac 2n}\left(
\int_M \chi_\ell^2 |\nabla (v-\mu_\ell)_+|^2 +
\int_M  |\nabla\chi_\ell |^2 (v-\mu_\ell)_+^2\right).
$$
Therefore, substituting the bound for $\int_M \chi_\ell^2 |\nabla (v-\mu_\ell)_+|^2  dg $
from \eqref{eq:iterate} and using smallness of $a$, we obtain
\begin{align*}
c \int_M \chi_\ell^2 (v-\mu_\ell)_+^2  dg &\le 
 |K_{l+1}|^{\frac 2n} \int_M |\nabla\chi_\ell|^2  (v-\mu_\ell)^2_+\ dg \\ &+ 
  |K_{l+1}|^{\frac 4n} \int_M \chi_\ell^2 |F|^2 dg 
 +\mu_\ell^2 \beta^2 |K_{\ell+1}|.
\end{align*}
We can further simplify this by using the bound $|K_\ell |\le C R_0^{n} 2^{n\ell}$.
\begin{align*}
c(R_0) \int_M \chi_\ell^2 (v-\mu_\ell)_+^2  dg &\le
 |K_{l+1}|^{\frac 2n} \int_M |\nabla\chi_\ell|^2  (v-\mu_\ell)^2_+\ dg \\ &+ 
  2^{4\ell} \int_M \chi_\ell^2 |F|^2 dg 
 +\mu_\ell^2 \beta^2 2^{n\ell}.
\end{align*}
It is important however to keep the factor $ |K_{l+1}|^{\frac 2n}$ in front of the 
first term!

By the Chebyshev inequality we have
$$
|K_{\ell+1}|\le \frac 1{(\mu_{\ell+1}-\mu_{\ell+2})^2}
 \int_M  \chi_{\ell+2}^2 (v-\mu_{\ell+2})_+^2.
$$
Observe that 
$$
|\nabla\chi_\ell|\le 2^{-\ell} R_0^{-1} \chi_{\ell+2},\qquad 
(\mu_{\ell+1}-\mu_{\ell+2})= \frac 12 \mu \ga (1+\ga)^{-\ell-2}, 
\qquad (v-\mu_{\ell+2})_+\ge (v-\mu_\ell)_+.
$$
We then obtain
\begin{align*}
c(R_0)\int_M \chi_\ell^2 (v-\mu_\ell)_+^2  dg&\le C(\mu,\ga)
(1+\ga)^{\frac {4\ell }n} {2^{-2\ell}} 
\Big (\int_M \chi_{\ell+2}^2 (v-\mu_{\ell+2})_+^2  dg \Big )^{1+\frac 2n}\\ &+
  2^{4\ell} \int_M \chi_\ell^2 |F|^2 dg 
 +\mu^2 \beta^2 2^{n\ell}.
\end{align*}
The constant $C(\mu,\ga)$ is essentially $(\mu\ga)^{-4/n}$. In particular, $C(\mu,\ga)$
becomes large as $\mu,\ga\to 0$.

We now define 
$$
A_\ell= 2^{-(3-2\si)\ell} \int_M \chi_\ell^2 (v-\mu_\ell)_+^2  dg
$$
so that, with the normalization \eqref{eq:normalization} 
$$
A_\ell\le c(R_0)\|u\|_{H^{0,-3/2+\si}(M)}^2=c(R_0)\le 1.
$$
Then
$$
A_\ell\leq C(\mu,\ga,R_0)\left (
(1+\ga)^{\frac {4\ell }n} {2^{(-2+\frac 2n(3-2\si))\ell}} A_{\ell+2}^{1+\frac 2n}+
  2^{(1+2\si)\ell} \int_M \chi_\ell^2 |F|^2 dg 
 +\mu^2 \beta^2 2^{ (n-3+2\si)\ell}\right).
$$
Observe that for $n\ge 3$
$$
-2+\frac 2n(3-2\si)\le -\frac {2\si}{3n},
$$
which
allows us to conclude that with an appropriate choice of a sufficiently small 
$\ga=\ga(\si)$ (explicitly, we need 
$6\ln(1+\ga)<\si \ln 2$) there exists a positive constant $\omega>0$ such that 
$$
A_\ell\le C(\mu,\ga,R_0)\left (
2^{-2\omega \ell} A_{\ell+2}^{1+\frac 2n}+
 2^{(1+2\si)\ell} \int_M \chi_\ell^2 |F|^2 dg 
 +\mu^2 \beta^2 2^{ (n-3+2\si)\ell}.\right)
$$
We now recall our normalizations \eqref{eq:normalization}. 
In particular,
$$
 \int_M \chi_\ell^2 |F|^2 dg \le  \int_M \chi_\ell^2 (|f|^2+ (\la^2+\eps^2) |u|^2) dg
 \leq C \de^2+ 2^{(3-2\si)\ell}C(R_0) (\la^2+\eps^2)
$$
and $A_\ell\le 1$. Thus
\begin{equation}\label{eq:Al}
A_\ell\le C(\mu,\ga,R_0)\left(
2^{-2\omega \ell} A_{\ell+2} +
 2^{(1+2\si)\ell}  \left( \de^2+ 2^{(3-2\si)\ell}(\la^2+\eps^2)\right)
 +\mu^2 \beta^2 2^{ (n-3+2\si)\ell}\right).
\end{equation}
Iterating \eqref{eq:Al} we obtain that for any $k \ll L$
$$
A_k\le C(\mu,\ga,R_0,k)^L(2^{-\omega L^2} +  \de^2+\la^2+\eps^2+\beta^2).
$$
The constants $\ga$ and $R_0$ have been fixed, independently
of $\la,\beta,\de$. Constant $\mu$ will be chosen to depend only on
the compact set $K$. Therefore, using the smallness of $\la,\de$ and $\beta$ 
we can find a large integer 
$L_0=L_0(\mu,\ga,R_0)$ 
such that for all $k\le k_0 \ll L_0$
$$
(\eps^2+\la^2+\beta^2) \le C(\mu,\ga,R_0,k)^{L_0} 2^{-100 n L_0}, \qquad
C(\mu,\ga,R_0,k)^{L_0}2^{-\omega L_0^2}\le 2^{-100 n L_0}.
$$
Thus, for all $k\le k_0$
$$
A_k\le 2^{-100 n L_0} + C(L_0) \de^2.
$$
We can assume that the set $K$ is contained in the ball $B(2^{k_0} R)$. 
Therefore,
$$
\int_{K} (v-\mu_{k_0})_+^2\le 2^{(3-2\si)k_0}(2^{-100 n L_0} + C(L_0) \de^2).
$$
The above inequality implies that 
$$
\int_{K} v_+^2\le 4 \mu^2 |K|+ 2^{(3-2\si)k_0}(2^{-100 n L_0} + C(L_0) \de^2)
\le (8C_K)^{-2}+ C \de^2,
$$
provided that $\mu$ is chosen to be sufficiently small relative to the size of $K$.
Repeating the above arguments for the function $-v=-\Re u$ and similarly for 
$\Im u$ we conclude that 
$$
\int_K |u|^2\le (2C_K)^{-2} + C \de^2.
$$
as desired. The remaining estimate for $\nabla u$ is straightforward and follows
immediately by integrating the equation $(-\Delta_M+V(x)-\la\mp i\eps) u =f$
against $\bar u$ over the set $K$. We omit the details.

\section{The quasimode counterexample}\label{quasimode-sec}

We now present the (standard) counterexample in Proposition \ref{pseudo} which shows that
the losses of $\exp( C \sqrt{\lambda} )$ which arise for instance in the unique continuation estimates
in Section \ref{carleman-sec} are in fact sharp.  For more sophisticated versions of this type of construction in the more general context of an elliptic trapped geodesic, see \cite{quasivodev}.

We begin by reviewing some basic facts about spherical harmonics.
Let $(S^n,g_n) \subset \R^{n+1}$ be the unit sphere with the standard metric, which we parameterize in Euler polar co-ordinates
as 
$$ S^n := \{ ((\sin \theta) \omega, \cos \theta): 0 \leq \theta \leq \pi; \omega \in S^{n-1} \}$$
Observe that the Laplace-Beltrami operators on $S^n$ and on $S^{n-1}$ are related by the formula
$$ \Delta_{S^n} = \frac{\partial^2}{\partial \theta^2} + \frac{(n-1) \cos \theta}{\sin \theta}
\frac{\partial}{\partial \theta} +  \frac{1}{\sin^2 \theta} \Delta_{S^{n-1}},$$
where $\Delta_{S^{n-1}}$ is of course applied to the $\omega$ variable.  Thus if $l \geq 0$ is an integer, and
$Y_l(\omega)$ is a spherical harmonic of order $l$ on $S^{n-1}$, i.e.
$$ \Delta_{S^{n-1}} Y_l(\omega) = -l(l+n-2) Y_l(\omega),$$
then the \emph{sectorial harmonic}
$$ U_l(\theta,\omega) := \sin^l(\theta) \cos(\theta) Y_l(\omega)$$
is a spherical harmonic on $S^n$, in fact we have
$$ (-\Delta_{S^n}-\lambda) U_l = 0, \hbox{ where } \lambda := (l+1)(l+n).$$
We shall think of being $l$ as being much larger than $n$, so $\lambda \sim l^2$.  Because of the
$\sin^l(\theta)$ factor, $U_l$ will be highly concentrated near the equator $\theta = \pi/2$ (which is a stable
trapped geodesic of the sphere $S^n$).  Indeed,
if we normalize $Y_l$ to have $L^2$ norm equal to 1 on $S^{n-1}$, then a simple computation shows that
$$ \int_{S^n: |\theta - \pi/2| < \pi/4} |U_l|^2 \sim l^{-1} \sim \lambda^{-1/2},$$
while
$$ \int_{S^n: |\theta - \pi/2| > \pi/4} |U_l|^2 + |\nabla U_l|^2 = O( e^{-cl} ) = O( e^{-c\sqrt{\lambda}})$$
for some constant $c = c_n > 0$ depending only on the dimension.

We now transfer this phenomenon to the setting in Proposition \ref{pseudo}.  Let $M \subset \R^{n+1}$ be any smooth
$n$-dimensional manifold isometrically embedded in $\R^{n+1}$, which is equal to the plane $\R^n \times \{0\} \subset
\R^{n+1}$ outside of a compact set, and contains the portion of the sphere
$$ S^n_{\geq \pi/8} := \{ ((\sin \theta) \omega, \cos \theta): \theta \geq \pi/8; \omega \in S^{n-1} \} \subset S^n \subset \R^{n+1}.$$
In other words, $M$ is formed by gluing a large portion of the sphere (which contains the equator $\theta = \pi/2$)
to the Euclidean space $\R^n$.  Note that $M$ thus inherits the equator of $S^n$ as a stable trapped geodesic.
Now let $l$ be a large integer parameter, set $\lambda = \lambda_l := (l+1)(l+n)$, and
consider the ``quasimode'' $u: M \to \C$ defined by $u := U_l \chi$, where $\chi$ is a smooth cutoff supported
on the set $S^n_{\geq \pi/8}$ which equals one on $S^n_{\geq \pi/4}$.  Then from the previous computations we see
that
$$ \| u \|_{H^{0,-1/2-\sigma}(M)} \sim \lambda^{-1/2}$$
and
$$ \|  (-\Delta_M - \lambda) u \|_{H^{0,1/2+\sigma}(M)} = O( e^{-c \sqrt{\lambda}} ),$$
and thus for $\eps = \eps_l$ sufficiently small
$$ \|  (-\Delta_M - (\lambda \pm i\eps)) u \|_{H^{0,1/2+\sigma}(M)} = O( e^{-c \sqrt{\lambda}} ).$$
Proposition \ref{pseudo} follows (with a small value of $C_0$) by 
setting $f_l := (-\Delta_M - (\lambda \pm i\eps)) u$.  To obtain a larger value of $C_0$, we simply scale this example,
replacing the unit sphere $S^n$ by a sphere of larger radius; we omit the standard details.
\endprf

\section{The Bessel matching counterexample}\label{bessel-matching-sec}

We now present the proof of Proposition \ref{bessel}.  We begin by considering the Helmholtz 
equation $(H-\lambda)u = f$ in polar co-ordinates $(r,\omega)$ on $\R^n$, thus
$$ (\frac{\partial^2}{\partial r^2} + \frac{n-1}{r} \frac{\partial}{\partial r}
+ \frac{1}{r^2} \Delta_{S^{n-1}} - V + \lambda) u = -F,$$
where we ignore the singularity at $r=0$ for now.  If we now assume $V$ to be radial, $V = V(r)$,
and apply the ansatz
\begin{equation}\label{ansatz}
 u(r,\omega) = r^{-(n-1)/2} v(r) Y_l(\omega), \quad F(r,\omega) = -r^{-(n+1)/2} G(r) Y_l(\omega)
\end{equation}
where $l \geq 0$ is a large even integer parameter and $Y_l$ is a spherical harmonic of order $l$ on $S^{n-1}$, normalized
to have $L^2(S^{n-1})$ norm equal to 1, then the Helmholtz equation becomes the \emph{Bessel ordinary
differential equation}
\begin{equation}\label{bessel-ode}
 v_{rr} - \frac{L(L-1)}{r^2} v - Vv + \lambda v = G,
\end{equation}
where $L := l + \frac{n-1}{2}$.  Also, we observe that
$$ \| u \|_{H^{0,-1/2-\sigma}(M)} \sim \int_0^\infty r^{-2\sigma} |v(r)|^2\ dr$$
and
$$ \| F \|_{H^{0,1/2+\sigma}(M)} \sim \int_0^\infty r^{2\sigma} |G(r)|^2\ dr.$$
Let us temporarily ignore the contributions of the $\lambda$, $V$, and $G$ factors, and consider the ODE
\begin{equation}\label{approx-ode}
 v_{rr} - \frac{L(L-1)}{r^2} v = 0.
\end{equation}
This equation has two linearly independent solutions, $r^L$ and $r^{-L+1}$.  The former function decays
quickly as $r \to 0$, and the latter decays quickly as $r \to \infty$.  To exploit this, let us define $v$
by \emph{fiat} to be a smooth function on $(0,\infty)$ which equals $r^L$ when $r < 1/2$, equals
$r^{-L+1}$ when $r > 1$, and is smooth and positive in between (of course, the smoothness bounds on $v$ will
depend on $L$).  Then the function $v_{rr} - \frac{L(L-1)}{r^2} v$ is smooth and supported on the annulus
$\{1/2 \leq r \leq 1\}$.  Note also that $u$ is then equal to $r^l Y_l(\omega)$ near the origin, which is a smooth function
(indeed, it is a harmonic polynomial of degree $l$).

Now let us define the spherically symmetric potential $V = V_l$ by
$$ V := \frac{v_{rr} - \frac{L(L-1)}{r^2} v}{v}.$$
By construction, $V$ is smooth and supported on the annulus $\{1/2 \leq r \leq 1\}$ (though the bounds on this potential
will get worse as $l$ increases), and we have
$$ v_{rr} - \frac{L(L-1)}{r^2} v - Vv = 0.$$
Reversing the above steps, this means that the corresponding function $u$ is an eigenfunction of the Schr\"odinger
operator $H := -\Delta + V$, with eigenvalue zero, thus $\Delta u = Vu$.  Also, $u$ decays at infinity like
$|u(r,\omega)| = O(r^{-(n-1)/2 - L+1}) = O(r^{-l-n})$, and in particular will be bounded in
the space $H^{0,1/2+\sigma}(M)$ for $l$ large enough.  

Next, we let $m \geq 1$ be an arbitrary integer and set $u_m := u \chi_m$, where $\chi_m$ 
is a smooth cutoff to the region $\{ r \leq 2m\}$ which equals one when $\{ r \leq m\}$.
We also set $\lambda= \lambda_m := m^{-l/10}$, and
set $V_m := V + \lambda \chi_m$.  Then we see that
$$ (-\Delta-V_m - \lambda) u_m$$
vanishes outside of the annulus $m \leq r \leq 2m$, and has magnitude $O_l( m^{-l + O(1)} )$ on this annulus, thus
$$ \| (-\Delta + V_m - \lambda_m) u_m \|_{H^{0,1/2+\sigma}(M)} \leq O_l( m^{-l + O(1)} ).$$
In particular for $\eps_m$ sufficiently small
$$ \| (-\Delta + V_m - (\lambda_m \pm i\eps_m)) u_m \|_{H^{0,1/2+\sigma}(M)} \leq O_l( m^{-l + O(1)} ).$$
Also, by construction of $u_m$ and $V_m$ we see that (if $l$ is sufficiently large)
$$ \| u_m \|_{H^{0,-1/2-\sigma}(M)} \geq c_l > 0.$$
and
$$ \sup_m \| \langle x \rangle^{1+\sigma_0} \nabla_x^\alpha V_m \|_{L^\infty(\R^n)} < \infty \hbox{ for all }
\alpha \geq 0,$$
and the claim follows (by taking $l$ large enough depending on $C_0$).  To complete the construction
we need to show that the hamiltonian $H_m=-\Delta + V_m$ does not contain a small negative 
eigenvalue or an  eigenfunction or resonance at zero. (Of course, the potentials
$V_m$ are very close to a fixed potential $V$ which \emph{does} have an eigenfunction at zero, 
which is indeed the cause
of the bad behavior of low frequency limiting absorption for the perturbed potentials.)
By the Cwikel-Rozenblum bound, 
the Hamiltonian $H=-\Delta +V$ has only finitely many negative eigenvalues labeled 
$\lambda_1\le \ldots \le\lambda_k<c<0$. Their number and location is independent of $m$. 
The operator norm of the perturbation $H_m-H$ is bounded by 
$m^{-l/10}$, which implies, by perturbation theory, that for sufficiently large values of $m$ the 
hamiltonian $H_m$ has at least $k$ negative eigenvalues $\lambda'_1\le \ldots \le\lambda'_k<
c+m^{-l/10}<0$. Moreover, denoting the linear span the associated eigenfunctions of $H_m$ 
by $P_m$, we have 
$$
\int_{\R^n} H\phi\, \bar\phi \ge 0
$$ 
for any $\phi\in H^2(\R^n)$ in the orthogonal complement of $P_m$. 
Recall that the potential $V$ is supported in the annulus $1/2\le r\le 1$. We can also assume 
without loss of generality that $|V(x)|\le 2$.
As a consequence, for any $\phi\in H^2(\R^n)$ in the orthogonal complement of $P_m$,
\begin{align}
\int_{\R^n} H_m\phi\, \bar\phi &= \frac \lambda 4 \int_{\R^n} H_m\phi\, \bar\phi+
(1-\frac \lambda 4)
\int_{\R^n} (H+\lambda\chi_m)\phi\, \bar\phi\nonumber\\ & \ge  
\frac \lambda 4 \int_{\R^n} \left (|\nabla\phi|^2+
V|\phi|^2 +\lambda\chi_m |\phi|^2\right)+ (1-\frac \lambda 4)
\int_{\R^n} \lambda\chi_m |\phi|^2\nonumber\\ &\ge  
\int_{\R^n} \left (\frac \lambda 4 |\nabla\phi|^2+\frac \lambda 2 \chi_m |\phi|^2\right)\ge 
C^{-1} \lambda \int_{\R^n} \frac {|\phi(x)|^2}{|x|^2}.\label{eq:res-eig}
\end{align}
This estimate already implies that $H_m$ has precisely $k$ negative eigenvalues 
$\lambda'_1\le \ldots \le\lambda'_k<
c+m^{-l/10}<0$ and that $0$ is not an eigenvalue. On the other hand if 
$\phi\in \cap_{\alpha>0} H^{2,-\frac 12-\alpha}(\R^n)$ is a zero resonance it is formally 
(the corresponding eigenfunctions are exponentially decaying) orthogonal to the subspace 
$P_m$. Moreover,  $\nabla\phi\in \cap_{\alpha>0} H^{1,\frac 12-\alpha}(\R^n)$,
$\Delta\phi\in \cap_{\alpha>0} H^{0,\frac 32-\alpha}(\R^n)$, which implies 
that \eqref{eq:res-eig} holds and thus $\phi$ vanishes identically. 
\endprf

\begin{remark} It is possible to sharpen this example a little bit by making the potentials $V_m$
compactly supported in the annulus $1/2 \leq r \leq 1$.  
This is achieved by replacing the approximating ODE \eqref{approx-ode} by
the minor variant $v_{rr} - \frac{L(L-1)}{r^2} v + (\lambda_m \pm i\eps_m) v = 0$, and replacing
$r^L$ and $r^{-L-1}$ by Hankel functions (which have similar decay behavior at zero and at $r \sim m$
respectively).  We omit the details.
\end{remark}

\section{Applications}

\subsection{Spectral applications}\label{spectral-sec}

We now prove Proposition \ref{absent} and Proposition \ref{singular}.  

\begin{proof} (Proof of Proposition \ref{absent})
Suppose first that we have an outgoing resonance $u$ for some $\lambda = \kappa^2$ for some $\kappa > 0$, thus
$u\in H^{0,-1/2-\sigma}(M)$, $H u=\la u$, and $(\pa_r u-i\kappa \, u)\in H^{0,-1/2+\sigma'}(M \backslash K_0)$ for some $\sigma' > 0$ and all $\sigma > 0$.  By shrinking $\sigma, \sigma'$ if necessary we can take $0 < \sigma < \sigma' < 1/2$.  By elliptic regularity, $u\in H^{2,-1/2-\sigma}(M)$.

Write $\kappa := \sqrt{\la}$, let $0<\de\ll \min(1,\kappa)$, and let $z := \kappa + i \delta$.
Consider the function 
$$
u_\de(x) :=e^{-\de \langle x\rangle} u
$$
where $\langle r \rangle := (1+r^2)^{1/2}$.

Since $Hu = \lambda u$, one has
$$ (H-z^2) u = - 2 i \delta \kappa u + \delta^2 u,$$
while from \eqref{deltam} one has
$$ \Delta_M e^{-\de \langle x \rangle} = -\delta^2 e^{-\de \langle x \rangle} + O( \frac{1}{\langle x \rangle} \delta e^{-\de \langle x \rangle} )$$
in the exterior of $K_0$. A straightforward application of the product rule then gives the equation
$$ (H-z^2)u_\delta = \delta e^{-\de \langle x \rangle} f_\delta$$
where $f_\delta$ takes the form
$$ f_\delta = 2 (\pa_r u - i \kappa u) + O\left( \frac{1}{\langle x \rangle} |u| \right) + O\left( \frac{1}{\langle x \rangle} |\partial_r u| \right)$$
in the exterior region $M \backslash K_0$, and takes the form
$$ f_\delta = O( u ) + O( |\nabla u| )$$
in the compact region $K_0$; here all implied constants are allowed to depend on $M$ and $\kappa$.
Applying the limiting absorption principle (Theorem \ref{main}) we conclude that
\begin{align*}
 \| u_\de \|_{H^{0,-1/2-\sigma}(M)} &\leq C(\lambda) \delta
( \| e^{-\de \langle x \rangle} (\pa_r u - i \kappa u) \|_{H^{0,1/2+\sigma}(M \backslash K_0)}\\
&\quad\quad + \| e^{-\de \langle x \rangle} u \|_{H^{1,-1/2+\sigma}(M)} ).
\end{align*}
Now observe that
$$ \| e^{-\de \langle x \rangle} u \|_{H^{1,-1/2+\sigma}(M)} \leq C \delta^{-2\sigma} \| u \|_{H^{1,-1/2-\sigma}(M)}$$
and
$$ \| e^{-\de \langle x \rangle} (\pa_r u - i \kappa u) \|_{H^{0,1/2+\sigma}(M \backslash K_0)}
\leq C \delta^{\sigma'-\sigma-1} \| \pa_r u - i \kappa u \|_{H^{0,-1/2+\sigma'}(M \backslash K_0)}.$$
We thus have
$$
 \| u_\de \|_{H^{0,-1/2-\sigma}(\R^n)} \leq C(\lambda) 
 ( \delta^{1-2\sigma} \| u \|_{H^{1,-1/2-\sigma}(M)} + \delta^{\sigma'-\sigma} \| \pa_r u - i \kappa u \|_{H^{0,-1/2+\sigma'}(M \backslash K_0)} ).$$
Taking limits as $\delta \to 0$ and using monotone convergence we conclude that $u$ is identically zero, which is absurd.
 
The same argument rules out incoming resonances and eigenfunctions at any positive energy $\lambda > 0$.
\end{proof}

\begin{remark}  If the Hamiltonian $H$ obeyed the limiting absorption principle \eqref{cmvl2} (note in particular that the constant here does not blow up as $\lambda \to 0$), a modification of above argument would also rule out an eigenvalue or resonance at zero; we omit the details\footnote{In fact, by taking $z$ to be a number such as $2\delta+i\delta$, it suffices to have the limiting absorption principle \eqref{cmvl2} in a sector such as $\{ z = \kappa + i\eps: 0 < \eps < \kappa \}$.}.  This provides a converse to Proposition \ref{lowenergy-jk}.  
\end{remark}

\begin{proof}[Proof of Proposition \ref{singular}]
The absence of singular continuous spectrum in an interval $(a,b)$ is guaranteed by the condition that 
$$
\sup_{0<\epsilon<1} \int_a^b |\Im \langle R(\lambda+i\epsilon)\phi, \phi\rangle_{L^2(M)}|^2\, d\lambda<\infty
$$
for any function $\phi\in C_0^\infty(M)$, see \cite[Theorem XIII.20]{Rsimon2}.  (Indeed, the spectral measure associated to $\phi$ will have an $L^2$ density with respect to Lebesgue measure.)  The result now easily follows from the
resolvent estimates established in Proposition \ref{main}.
\end{proof} 

\subsection{Local smoothing estimates and integrated local energy  decay}\label{ls-sec}

We now prove Propositions \ref{glos} and \ref{prop:int-decay}.  As the arguments are standard (dating back to \cite{kato1}; see also \cite{burq-2004}, \cite{c08}, \cite{c08b}), we shall be somewhat brief in our discussion.  In particular, we will work formally (assuming that all exchanges of integrals, etc. are justified); one can make these formal computations rigorous by standard approximation arguments which we omit here.

We begin with the proof of \eqref{glos-1} from Proposition \ref{glos}.  Henceforth we fix $M,V,A$ and allow all constants to depend on these quantities.

From Duhamel's formula we have
$$ H^{1/2} P_H u(t) = H^{1/2} P_H e^{itH} u(0) - i \int_0^t H^{1/2} P_H e^{i(t-t')H} \nabla^* F(t')\ dt'.$$
If we can establish the retarded estimate
\begin{equation}\label{retard} \| \int_{t' < t} H^{1/2} P_H e^{i(t-t')H} F(t') \|_{L^2_t H^{0,-1/2-\sigma}(\R \times M)} \leq C 
\|F\|_{L^2_t H^{0,1/2+\sigma}(\R \times M)}
\end{equation}
then by time reversal we also have the corresponding advanced estimate in which the constraint $t'<t$ in \eqref{retard} is replaced by $t'>t$.  Summing, we obtain
$$ \| \int_\R H^{1/2} P_H e^{i(t-t')H} F(t') \|_{L^2_t H^{0,-1/2-\sigma}(\R \times M)} \leq C 
\|F\|_{L^2_t H^{0,1/2+\sigma}(\R \times M)};$$
taking the inner product of the expression inside the left-hand norm with $F$ and rearranging, we obtain the inhomogeneous estimate
$$ \| \int_\R H^{1/4} P_H e^{-it' H} F(t')\ dt' \|_{L^2(M)} \leq C \|F\|_{L^2_t H^{0,1/2+\sigma}(\R \times M)}.$$
which by duality gives
$$ \| \nabla H^{1/2} P_H e^{itH} u_0 \|_{L^2_t H^{0,-1/2-\sigma}(\R \times M)} \leq C \|H^{1/4} P_H u_0\|_{L^2(M)}$$
for any test function $u_0$.  From all these estimates we see that to prove \eqref{glos-1} it suffices to establish the retarded estimate \eqref{retard} and the standard elliptic estimate
$$ \| H^{1/4} P_H u_0 \|_{L^2(M)} \leq C \|u_0\|_{H^{1/2}(M)}.$$
To prove the latter estimate, we observe from the $TT^*$ method that it is equivalent to $H^{1/2} P_H$ mapping $H^{1/2}(M)$ to $H^{-1/2}(M)$, which by interpolation and duality follows from $H^{1/2} P_H$ mapping $H^1(M)$ to $L^2(M)$, which by another application of the $TT^*$ method (and the boundedness of $P_H$ on $L^2(M)$) follows from the Dirichlet form $\langle Hu, v \rangle$ being bounded in $H^1(M)$.

Now we turn to \eqref{retard}.
By a limiting argument, it suffices to prove the damped retarded estimate
\begin{equation}\label{retard-damp} 
\liminf_{\eps \to 0^+} \| \int_{t' < t} H^{1/2} P_H e^{i(t-t')(H+i\eps)} F(t') \|_{L^2_t H^{0,-1/2-\sigma}(\R \times M)} \leq C 
\|F\|_{L^2_t H^{0,1/2+\sigma}(\R \times M)}.
\end{equation}

Following Kato \cite{kato1}, we perform a Fourier transform in the time variable
$$ F(t) = \int_\R e^{i \lambda t} \hat F(\lambda)\ d\lambda.$$
The expression inside the left-hand norm of \eqref{retard-damp} can then be expressed as\footnote{One can also shift the contour, taking advantage of the projection away from any negative eigenvalues to work with resolvents in a neighbourhood of the positive real axis, if desired.}
$$ \int_\R e^{i\lambda t} H^{1/2} P_H R(\lambda-i\eps) \hat F(\lambda)\ d\lambda.$$
Applying Plancherel's theorem (and Fatou's lemma) to both sides, we see that it suffices to establish the estimates
$$ \limsup_{\eps \to 0^+} \| H^{1/2} P_H R(\lambda-i\eps) f \|_{H^{0,-1/2-\sigma}(M)} \leq C \|f\|_{H^{0,1/2+\sigma}(M)}
$$
for all $\lambda \in \R$ and all test functions $f$.  If $\lambda$ is negative and bounded away from the origin, then $H^{1/2} P_H R(\lambda-i\eps) $ is bounded on $L^2(M)$ (thanks to the spectral theorem), so we may assume that $\lambda$ is either positive, or close to the origin.

Using the elliptic estimate
$$ \| H^{1/2} u \|_{H^{0,-1/2-\sigma}(M)} \leq C \| u \|_{H^{1,-1/2-\sigma}(M)}$$
it suffices to show that
$$ \| P_H R(\lambda-i\eps) f \|_{H^{1,-1/2-\sigma}(M)} \leq C \|f\|_{H^{0,1/2+\sigma}(M)}
$$
whenever $\lambda \in \R$ and $\eps$ is sufficiently small.  

For $\lambda$ sufficiently close to the origin, this follows from Proposition \ref{lowenergy-jk}.  For $\lambda$ positive and bounded away from the origin, the claim follows instead from Theorem \ref{nontrap}.  This completes the proof of Proposition \ref{glos}.

\begin{remark}
It is clear from the above argument that the global-in-time local smoothing estimates are in fact \emph{equivalent} to the limiting absorption principle (with the appropriate decay in $\lambda$ in the constants).
\end{remark}

\begin{proof}[Proof of Proposition \ref{prop:int-decay}]
As before, we allow all constants to depend on $C,M,A$.

We rewrite the wave equation in the form 
$$
\Phi_t - \A\Phi = \tilde F,
$$
where $\Phi:= \begin{pmatrix} u \\ u_t \end{pmatrix}$, $\tilde F := \begin{pmatrix} 0 \\ F \end{pmatrix}$, and $\A$ is the matrix operator
$$\A:=\left(
     \begin{array}{ccc}
        0& 1&\\
         -H&0
     \end{array}\hskip -.7pc\right)$$
We begin with the proof of \eqref{int-decay-1}.
As in the proof of Proposition \ref{glos}, it suffices to consider the solution of the retarded inhomogeneous problem 
 $$
 \Phi(t)=\int_{s<t} e^{(t-s)A} \tilde F(s)\,ds.
 $$    
Applying Plancherel as before, we see that bounds on the retarded solution are equivalent to bounds on the 
resolvent $(i\A-\mu-i0^-)^{-1}$. More precisely, if $\Psi=(\psi_1,\psi_2)$ is a solution 
of the equation
$$
(i\A-\mu+i\eps)\Psi=(0,g)
$$
then we need to show that uniformly in $\eps>0$ and $\mu \in \R$ that
$$
\|\nabla\psi_1\|_{H^{0,-1/2-\sigma}(\R^n)}+\|\psi_2\|_{H^{0,-1/2-\sigma}(\R^n)}\le 
C \| g\|_{H^{0,1/2+\sigma}(\R^n)}.
$$
Since
\begin{equation}\label{oo} 
(i\A-\mu+i\eps)^{-1} =
R((\mu-i\eps)^2)
\left(
     \begin{array}{ccc}
        \mu-i\eps& i&\\
        -iH& \mu-i\eps
     \end{array}\hskip -.7pc\right)
\end{equation}
we have
$$ \psi_1 = i R((\mu-i\eps)^2) g; \quad \psi_2 = (\mu-i\eps) R((\mu-i\eps)^2) g$$
and the desired estimate \eqref{int-decay-1} follows immediately from Theorem \ref{main}.

The proof of \eqref{int-decay-2} follows from a similar argument using Lemma \ref{lem:improvement}; we omit the details.
\end{proof}

\subsection{The RAGE theorem}\label{rage-sec}

We now prove Proposition \ref{rage}.  We establish the claim just for the Schr\"odinger equation; the claim for the wave equation is analogous and is left to the reader.

By $L^2$ stability of the Schr\"odinger equation 
it suffices to establish \eqref{eq:rage} on a dense subset 
of $L^2(M)$. Thus we may assume that $u(0)\in H^2(M)$. Moreover, Proposition \ref{absent} 
together with the assumption that $u(0)=f$ is orthogonal to all eigenfunction of $H$ implies continuity 
of the measure $d\mu_f(\la)=(dE_\la(H)f,f)$ defined from the spectral measure $dE_\lambda(H)$ associated to $H$.
As a consequence, by a density argument, we can assume that $u(0)$ is spectrally supported on an interval $\lambda\in (a,b)$ 
with $0<a<b<\infty$ of energies $\lambda$. For such functions we have both the limiting absorption principle and 
the local smoothing estimate
$$
\int_0^\infty \|u(t)\|_{H^{2,-1/2-\si}(M)}^2\ dt < \infty
$$
for some $\si > 0$.

Let $\eps>0$.  Applying  the monotone convergence theorem, we can find $T_0 > 0$ such that
$$
\int_{T_0}^\infty \|u(t)\|_{H^{2,-1/2-\si}(M)}^2\ dt < \eps.
$$
Now for any $T > T_0+1$ and
 $g \in H^{0,1/2+\si}(M)$ we have 
\begin{align*}
\int_{\R^n} \overline {u(T,x)} g(x) &= \int_{T-1}^T \int_{M} \frac {d}{dt}  
\big (\overline{u(t,x)} (t-T+1) g(x) \big )\ dx dt \\ & 
=\int_{T-1}^T \int_{\R^n} 
\big (i H \overline{u(t,x)} (t-T+1) g(x) + \overline{u(t,x)} g(x)\big )\ dx dt 
\\ &\le C \|g\|_{H^{0,1/2+\si}(M)}\,\left(\int_{T-1}^T \|u(t)\|_{H^{2,-1/2-\si}(M)}^2\ dt \right)^{1/2} 
\\ &\le C \eps \|g_0\|_{H^{0,1/2+\si}(M)}.
\end{align*}
This implies that $u(t)\to 0$ in $H^{0,-1/2-\si}(M)$ for some $\sigma >1/2$, and
the result follows.

\subsection{The limiting amplitude principle}\label{lamp-sec}

We now prove Proposition \ref{lamp}.   We will follow an approach of Eidus \cite{eidus} and deduce the limiting amplitude principle from the limiting absorption principle, combined with a H\"older continuity property of the resolvent.  More precisely, we will use the following fact:

\begin{proposition}[H\"older continuity]\label{holder}  Let the assumptions be as in Theorem \ref{main}.  Let $K$ be a compact subset of $M$, and let $F$ be a compact subset of the right half-plane $\{ z: \Re z > 0 \}$ that avoids $0$.  Then there exists $C > 0$ and $\sigma > 0$ such that one has the H\"older continuity bound
$$ \| R(z) f - R(z') f \|_{H^1(K)} \leq C |z-z'|^\sigma \|f\|_{L^2(K)}$$
for all $z,z' \in F$ with $\Im z, \Im z' > 0$, and all $f \in L^2(K)$.
\end{proposition}

Let us assume this proposition for the moment and prove Proposition \ref{lamp}.  For simplicity of notation we will assume that $H$ has no eigenfunctions; the general case is analogous and proceeds by inserting a projection $P_H$ to the absolutely continuous portion of the spectrum of $H$ (which commutes with all spectral multipliers of $H$, or of related operators such as the matrix operator $\A$ introduced below) throughout the argument.
By Proposition \ref{rage} and linearity, it suffices to handle the homogeneous case $u_0=u_1=0$.  
As in the proof of Proposition \ref{prop:int-decay}, we introduce the vector $\Phi=\begin{pmatrix} u \\ u_t\end{pmatrix}$ and rewrite 
the wave equation in the form 
$$
\Phi_t - \A\Phi = e^{i\mu t} F,
$$
where $F :=\begin{pmatrix} 0 \\ f \end{pmatrix}$ and $\A$ is the matrix operator 
$$\A:=\left(
     \begin{array}{ccc}
        0& 1&\\
         -H&0
     \end{array}\hskip -.7pc\right).$$
It will suffice to show that
$$ 
e^{-i\mu t} \Phi(t) \to \begin{pmatrix} v \\ i\mu v \end{pmatrix}
$$
in $H_1(K) \times L^2(K)$ as $t \to +\infty$.
 
By the Duhamel formula, we have
\begin{align*}
e^{-i\mu t} \Phi(t) & =\int_0^t e^{(t-s)(\A-i\mu)} ds F \\
&= \lim_{\eps \to 0^+} \int_0^t e^{(t-s)(\A-i\mu-\eps)} ds F \\
&= e^{t(\A - i\mu)} (\A-i\mu-0^+)^{-1} F - (\A-i\mu-0^+)^{-1} F,
\end{align*}
writing $(\A-i\mu-0^+)^{-1}$ as the weak limit of the $(\A-i\mu-\eps)^{-1}$ as $\eps \to 0^+$.
From \eqref{oo} we have
$$ - (\A-i\mu-0^-)^{-1} F = R((\mu-i0^+)^2) \begin{pmatrix} u \\ i \mu u \end{pmatrix}$$
and so from Proposition \ref{qlap} (and elliptic regularity) one has
$$
(\A-i\mu-0^+)^{-1} F \to \begin{pmatrix} v \\ i\mu v \end{pmatrix}$$
in $H^1(K) \times L^2(K)$.  Discarding the bounded phase $e^{-it\mu}$, it thus remains to establish that
$$ e^{t\A} (\A-i\mu-0^+)^{-1} F $$
converges strongly in $H^1(K) \times L^2(K)$ to zero as $t \to +\infty$.  

Let $C$ be a semicircular contour of the form
$$ \{ i y: |y-\mu| \leq r \} \cup \{ \mu + r e^{i\theta}: \pi/2 \leq \theta \leq 3\pi/2 \}$$
for some radius $0 < r < \mu$ (e.g. one can select $r := \mu/2$).  From the Cauchy integral formula and the spectral theorem (using the fact that the spectrum of $\A$ is the imaginary axis) we have
\begin{align*}
 e^{t\A} (\A-i\mu-0^+)^{-1} F &= \lim_{\eps \to 0^+} \lim_{\eps' \to 0^+} \frac{1}{2\pi i} \int_C \frac{e^{tz}}{z-i\mu-\eps} (\A - z - \eps')^{-1} f\ dz \\
 &\quad + e^{t\A} 1_{|\A-i\mu| \geq r} (\A-i\mu-0^+)^{-1} F
\end{align*}
strongly in $H^1(K) \times L^2(K)$.  From the spectral theorem, $1_{|\A-i\mu| \geq r} (\A-i\mu-0^+)^{-1} F$ lies in $H^1(M) \times L^2(M)$, and so by the RAGE theorem, the second term on the right-hand side goes to zero strongly in $H^1(K) \times L^2(K)$ as $t \to \pm \infty$.  It thus suffices to show that
$$ \lim_{t \to +\infty} \lim_{\eps \to 0^+} \lim_{\eps' \to 0^+} \int_C \frac{e^{tz}}{z-i\mu-\eps} (\A - z - \eps')^{-1} f\ dz = 0$$
strongly in $H^1(K) \times L^2(K)$.

From the Cauchy integral formula we have
$$ \int_C \frac{e^{tz}}{z-i\mu+\eps}\ dz = 0$$
for $\eps$ small enough, and so it suffices to show that
\begin{equation}\label{aloe}
\lim_{t \to +\infty} \lim_{\eps \to 0^+} \lim_{\eps' \to 0^+} \int_C \frac{e^{tz}}{z-i\mu-\eps} G_{\eps'}(z)\ dz = 0
\end{equation}
strongly in $H^1(K) \times L^2(K)$, where
$$ G_{\eps'}(z) := (\A - z - \eps')^{-1} F - (\A - i\mu - \eps')^{-1} F.$$
From \eqref{oo} and the H\"older continuity property in Proposition \ref{holder}, $G_{\eps'}$ is H\"older continuous in $H^1(K) \times L^2(K)$, uniformly in $\eps'$; in particular, as $G_\eps'$ vanishes at $i\mu$, we have the bound
\begin{equation}\label{zmu}
\| G_{\eps'}(z) \|_{H^1(K) \times L^2(K)} = O( |z-i\mu|^\delta )
\end{equation}
for some $\delta>0$.  This is already enough, when combined with Minkowski's integral inequality and the dominated convergence theorem, to control the semicircular portion of the contour $C$; it remains to demonstrate that
$$
\lim_{t \to +\infty} \lim_{\eps \to 0^+} \lim_{\eps' \to 0^+} \int_{-r}^r \frac{e^{ity}}{iy-\eps} G_{\eps'}(i\mu+iy)\ dy = 0$$
strongly in $H^1(K) \times L^2(K)$.  From \eqref{zmu} and Minkowski's integral inequality we have
$$
\left\| \int_{-s}^s \frac{e^{ity}}{iy-\eps} G_{\eps'}(i\mu+iy)\ dy \right\|_{H^1(K) \times L^2(K)} = O( s^\delta )$$
for any $0 < s < r$, so it suffices to show that
$$
\lim_{t \to +\infty} \lim_{\eps \to 0^+} \lim_{\eps' \to 0^+} \int_I \frac{e^{ity}}{iy-\eps} G_{\eps'}(i\mu+iy)\ dy = 0$$
for any compact interval $I$ in $[-r,r]$ avoiding the origin.  However, from \eqref{zmu} we see that for any unit vector $w$ in $H^1(K) \times L^2(K)$, the scalar function
$$ y \mapsto \left\langle \frac{1}{iy-\eps} G_{\eps'}(i\mu+iy), w \right\rangle_{H^1(K) \times L^2(K)}$$
is uniformly H\"older continuous in $\eps, \eps'$ on $I$, and in particular (by the Arzela-Ascoli theorem) is precompact in the uniform norm.  From the Riemann-Lebesgue lemma we thus have
$$
\lim_{t \to +\infty} \lim_{\eps \to 0^+} \lim_{\eps' \to 0^+} \left\langle \int_I \frac{e^{ity}}{iy-\eps} G_{\eps'}(i\mu+iy)\ dy, w \right\rangle_{H^1(K) \times L^2(K)} = 0$$
uniformly in $w$, and the claim follows by duality.

It remains to prove Proposition \ref{holder}.  We first observe from the limiting absorption principle that
$$ \| R(z) f \|_{H^2(K)} \leq C \|f\|_{L^2(K)}$$
uniformly for all $z$ in a compact set avoiding the origin.  Thus, by interpolation, it suffices to obtain an $L^2(K) \to L^2(K)$ H\"older continuity bound, thus we need to show that
$$ \| R(z) f - R(z') f \|_{L^2(K)} \leq C |z-z'|^\sigma \|f\|_{L^2(K)}$$
for all $z, z'$ in $F$ with positive imaginary part, and some sufficiently small $\sigma > 0$ (which may be different from the one in Proposition \ref{holder}).  Clearly, we may assume that $|z-z'|$ is smaller than any given absolute constant.

Let $\sigma > 0$ be a small exponent to be chosen later.
By the triangle inequality, it suffices to prove this claim under the additional assumption that $\Im(z') \geq 0.1 |z-z'|$ (say).  Write $\eta := \Im(z')$, then $\eta > 0$ is small and $|z-z'| = O(\eta)$, with $|z|, |z'|$ comparable to $1$, and with $\Re(z), \Re(z')$ positive and bounded away from the origin, and our task is to now show that
$$ \| R(z) f - R(z') f \|_{L^2(K)} \leq C \eta^\sigma \|f\|_{L^2(K)}.$$
Using the resolvent identity
$$ R(z) - R(z') = (z-z') R(z) R(z')$$
it thus suffices to show that
$$ \| R(z) R(z') f \|_{L^2(K)} \leq C \eta^{-1+\sigma} \|f\|_{L^2(K)},$$
so by duality it suffices to show that
$$ |\langle R(z') f, R(\overline{z}) h \rangle_{L^2(M)}| \leq C \eta^{-1+\sigma}$$
whenever $f, h \in L^2(K)$ have unit norm.

Fix $f, h$.
Recall that $z'$ has imaginary part $\eta$.  This effectively localises $R(z') f$ to the region $\{ x: \langle x \rangle = O( 1/\eta ) \}$.  Indeed, if we write $z' := -\omega^2$, where $\Im \omega > 0$, then $|\omega|$ is comparable to $1$ and $\Im \omega$ is comparable to $\eta$ (with constants depending on the compact region $F$).  Using the elementary identity
$$ \frac{1}{x^2 + \omega^2} = \frac{1}{\omega} \int_0^\infty e^{-\omega t} \cos tx\ dt$$
we see that
$$ R(z') f = \frac{1}{\omega} \int_0^\infty e^{-\omega t} \cos(t \sqrt{-H}) f\ dt.$$
The wave propagators $\cos(t \sqrt{-H})$ are contractions on $L^2(M)$, and $\cos(t \sqrt{-H}) f$ is supported in the region where $\langle x \rangle \leq t + O(1)$.  From the exponential decay in the $e^{-\omega t}$ factor, we thus see that
\begin{equation}\label{xo}
\int_{|z| \geq \eta^{-1+\sigma}} \langle x \rangle^{100} |R(z') f(x)|^2\ dg \leq C \eta^{100} 
\end{equation}
(say), for any fixed $\sigma > 0$ (allowing $C$ to depend on $\sigma$).  Thus, it will suffice to show that
\begin{equation}\label{psi}
 |\int_M \psi (R(z') f) \overline{R(\overline{z}) h}\ dg| \leq C \eta^{-1+\sigma}
\end{equation}
for a smooth cutoff $\psi$ to the region $\langle x \rangle \leq 2\eta^{1-\sigma}$ that equals $1$ when $\langle x \rangle \leq \eta^{1-\sigma}$, as the error term caused by $1-\psi$ can be estimated by \eqref{xo} and the limiting absorption principle (with plenty of powers of $\eta$ to spare).

Write $u := R(z') f$ and $v := R(\overline{z} h)$.  From the limiting absorption principle we have
\begin{equation}\label{uvlap}
\| u\|_{H^{0,-1/2-\sigma}(M)}, \| v\|_{H^{0,-1/2-\sigma}(M)} \leq C.
\end{equation}
If we apply these bounds and Cauchy-Schwarz to estimate \eqref{xo} directly, we obtain a bound of $O( \eta^{-1-2\sigma})$, which barely fails to be adequate for our purposes.  To obtain the additional powers of $\eta$ needed to close the argument, we take advantage of the fact that $z'$ and $\overline{z}$ lie on different sides of the real axis, and so $R(z') f$ and $R(\overline{z}) h$ obey opposing radiation conditions.    Indeed, if we write $u := R(z') f$ and $v := R(\overline{z} h)$, then from Proposition \ref{qlap} we have
\begin{equation}\label{uv-1}
\| (\partial_r - i z') u\|_{H^{0,-1/2+3\sigma}(M \backslash K_0)} \leq C
\end{equation}
and
\begin{equation}\label{uv-2}
 \| (\partial_r + i \overline{z}) v\|_{H^{0,-1/2+3\sigma}(M \backslash K_0)} \leq C
\end{equation}
(say), if $\sigma$ is small enough.
To use these facts, denote the integral in \eqref{psi} by $I$.  From \eqref{uvlap} and Cauchy-Schwarz, we can write
$$ I = \int_M (\psi-\psi_0)  u \overline{v}\ dg + O(1),$$
where $\psi_0$ is a cutoff to a fixed compact region that equals $1$ on $K_0$.  Applying \eqref{uv-1} (and using \eqref{uvlap} to estimate the error), we can then write
$$ I = \frac{1}{iz'} \int_M (\psi-\psi_0) (\partial_r u) \overline{v}\ dg + O(\eta^{-1+\sigma}).$$
We integrate by parts to then obtain
\begin{align*}
I &= -\frac{1}{iz'} \int_M (\psi-\psi_0) u \overline{(\partial_r v)}\ dg + O(\eta^{-1+\sigma}) \\
&\quad + O( \int_M \frac{1}{\langle x \rangle} \psi |u| |v|\ dg ).
\end{align*}
Using \eqref{uvlap} and Cauchy-Schwarz, the final error term is $O(\eta^{-1+\sigma})$ for $\sigma$ small enough.  For the main term, we use \eqref{uv-2}, estimating the error again using \eqref{uvlap} and Cauchy-Schwarz, to obtain
$$
I = -\frac{1}{iz'} \int_M (\psi-\psi_0) u \overline{(-i\overline{z}v)}\ dg + O(\eta^{-1+\sigma}) 
$$
which simplifies (using \eqref{uvlap} and Cauchy-Schwarz one last time to remove the $\psi_0$ cutoff) to
$$ I = - \frac{\overline{z}}{z'} I + O( \eta^{-1+\sigma} ).$$
As $z, z'$ are both in the upper half-plane, have magnitude comparable to $1$, and are within $\eta$ of each other, we see that $-\frac{\overline{z}}{z'}$ is a bounded distance away from $1$ for $\eta$ small enough.  We thus conclude that $I = O( \eta^{-1+\sigma} )$ as required.  This proves Proposition \ref{holder} and thus Proposition \ref{lamp}.

\begin{remark}  It should be clear from the above argument that one also has a similar limiting absorption principle for the Schr\"odinger evolution (replacing $H^1(K) \times L^2(K)$ by $L^2(K)$).  We leave the details to the interested reader.  It should also be clear from the argument that one can strengthen the convergence in $H^1(K) \times L^2(K)$ to convergence in $H^{1,-s}(M) \times H^{0,-s}(M)$ for some sufficiently large $s$, and dually one can also relax the hypothesis $f \in L^2(K)$ to $f \in H^{0,s}(M)$.
\end{remark}

\begin{remark} The hypothesis in Proposition \ref{holder} that $F$ lie in the right half-plane and avoid zero was needed in order to invoke the limiting absorption principle.  If one assumes that $H$ has no eigenfunctions (or if $f$ is assumed to be orthogonal to such eigenfunctions), then one can extend this proposition to the left half-plane as well (and indeed the claim follows easily from the spectral theorem in that caase), and if $H$ has no eigenfunction or resonance at zero, then one no longer needs to avoid the origin (thanks to Proposition \ref{lowenergy-jk}).  It is likely that one can upgrade the H\"older continuity bound to a stronger bound, such as differentiability, under further regularity and decay hypotheses on the metric $g$ and potential $V$, but we will not pursue this matter here.
\end{remark}


\begin{remark}  We sketch here an alternate approach to the limiting amplitude principle that does uses wave equation energy estimates instead of H\"older continuity properties of the resolvent, but requires the $\sigma_0$ parameter to be large (in particular, potential needs to be strongly short-range), and also requires $H$ to have no bound states.  The main task, as noted above, is to establish decay of $e^{t\A} (\A-i\mu-0^+)^{-1} F$ in $H^1(K) \times L^2(K)$ as $t \to +\infty$.  If the data $\begin{pmatrix} v \\ i\mu v \end{pmatrix} = (\A-i\mu-0^+)^{-1} F$ had finite total energy, then this would follow from the RAGE theorem; however, from the spectral theorem we see that we do not expect this data to have finite energy.  However, one can use Lemma \ref{lem:improvement} (or more precisely, a variant of this lemma) to show (if $\sigma_0$ is large enough) that $\begin{pmatrix} v \\ i\mu v \end{pmatrix}$ has finite \emph{incoming} energy, in the sense that $r^{-1} \nabla_\omega v$ and $v_r + i \mu v$ lie in $L^2(M)$.  It turns out that these types of estimates, together with wave equation energy estimates formed by contracting the stress-energy tensor against a well-chosen vector field (essentially outgoing vector field $\partial_t+\partial_r$), shows that the energy of $e^{t\A} (\A-i\mu-0^+)^{-1} F$ on a forward light cone is bounded, which by energy estimates implies that $e^{t\A} (\A-i\mu-0^+)^{-1} F$ stays bounded in $H^1(K) \times L^2(K)$.  By truncating away a compactly supported component of $(\A-i\mu-0^+)^{-1} F$ (whose contribution decays by the RAGE theorem), one can upgrade this boundedness to decay.
\end{remark}

\section{Decay estimates for the time-dependent Schr\"odinger equation}\label{schro-decay}

In this section we prove Proposition \ref{prop:decay-S}.  We may assume that $t \geq 1$, since the case $0 \leq t \leq 1$ follows from Sobolev embedding and $H^2(M)$ energy estimates.  Our initial analysis will be valid in all dimensions three or greater, but we will eventually specialize to the three-dimensional case for sake of concreteness.

Let $\psi$ be a solution of the time-dependent Schr\"odinger equation
\begin{align*}
&i\pa_t\psi +\Delta_M \psi=0,\\
&\psi|_{t=0}=\psi_0.
\end{align*}
Consider the following second order self-adjoint operator
$$
P:=t^2\Delta_M -it\chi(r\pa_r+\frac n2 +\frac r2\theta) - \frac 14 \chi^2 r^2
$$
with a smooth cut-off function $\chi(r)$ supported in the region $\langle x\rangle\ge r_0$ 
and equal to $1$ for $r\ge 2r_0$ for some sufficiently large $r_0$.  As we shall see later, this operator can essentially be viewed as a conjugate of $t^2 \Delta_M$, via the heuristic\footnote{In the Euclidean case $H = \Delta_{\R^n}$, setting $\chi$ equal to $1$, $P$ is the Laplacian conjugated by the pseudoconformal transformation $u(t,x) \mapsto \frac{1}{t^{n/2}} e^{i r^2/4t} \overline{u(\frac{1}{t},\frac{x}{t})}$, which explains why we expect $P$ to approximately commute with the Schr\"odinger operator $i \pa_t + \Delta_M$.}
\begin{equation}\label{Heuristic}
P \approx e^{i \chi r^2/4t} t^2 \Delta_M e^{-i \chi r^2/4t}.
\end{equation}
We will establish the decay estimate by commuting the Schr\"odinger equation with $P$.

Recall that in the region $\langle x\rangle\ge r_0$ the metric $g$ on $M$ takes the 
form 
$$
g=dr^2+ r^2 h[r]_{ab} d\omega^a\, d\omega^b,\qquad h[r]_{ab}=h_{ab}(\omega) + r^{-2\si_0}
e_{ab}(r,\omega)
$$
and $\theta=\frac 12 h[r]^{ab}\pa_r h[r]_{ab}$.
By hypothesis, we are assuming that 
$$
|(r\pa_r)^k (\nabla_\omega^\alpha) h|\le C_{k\alpha},\qquad k\le 3, \,\,|\alpha|\le 2.  
$$
Let $\phi:=P\psi$, then $\phi$ solves the forced Schr\"odinger equation
$$
i\pa_t\phi +\Delta_M \phi=F
$$
where
$$ F:= (i\pa_t P +[\Delta_M,P])\psi.
$$
To obtain bounds for $F$, we compute
$$
i\pa_t P= 2it\Delta_M +\chi (r\pa_r+\frac n2+\frac r2\theta)$$
and
\begin{align*}
[\Delta_M,P]&= -it [\Delta_M, \chi(r\pa_r+\frac n2+\frac r2\theta)] - \frac 14 [\Delta_M,\chi^2 r^2]=
-it (\Delta_M \chi) (r\pa_r+\frac n2+\frac r2\theta)\\& - 2\chi' it  \pa_r -2\chi'it r\pa_r^2-
\chi'it \pa_r(r\theta)-2\chi' it (r\pa_r+\frac n2+\frac r2\theta)\pa_r
-\chi it  [\Delta_M, r\pa_r+\frac r2\theta] \\ &- 
\frac 14 (\Delta_M \chi^2) r^2 - 2r \chi \chi' -\chi\chi' r^2\pa_r-\frac 14\chi^2[\Delta_M, r^2]
\end{align*}
with the understanding that all expressions involving $\chi$ or its derivatives vanish on $K_0$.

In polar coordinates (outside of $K_0$), the Laplace-Beltrami operator has the following representation
$$
\Delta_M =\pa_r^2 + \left(\frac {n-1} r+\theta\right) \pa_r+\frac 1{r^2} \Delta_{h[r]}.
$$
As a consequence, we may compute commutators:
\begin{align*}
[\Delta_M, r\pa_r] := 2 \Delta_M -(r \pa_r\theta+\theta)\pa_r +\frac 1{r} \pa_r (  \Delta_{h[r]}),\\
[\Delta_M, r^2]:= 4 r \pa_r + 2 + 2(n-1) + 2\theta r,\\
[\Delta_M,r\theta]:=2(\theta+r\pa_r\theta)\pa_r+\frac 2r \nabla_\omega\theta \nabla_\omega+
\left(\frac{n-1}r+\theta\right)\theta +2\pa_r\theta +r(\Delta_M\theta).
\end{align*}
Putting all these estimates together, we obtain the pointwise bound
\begin{align*}
|F |&\le C \left ((1+t)\,\zeta\, |\pa_r\psi| + t \,\zeta |\pa^2_r\psi|\right.\\
&\quad+\left.
r^{-2-2\si_0} t \,(|\nabla_\omega^2 \psi|+|\nabla_\omega\psi|)
+(1+t) r^{-2-2\si_0}\, |\psi|\right ),
\end{align*}
where $\zeta$ is a smooth cut-off 
supported in the region $r_0\le \langle x\rangle\le 2r_0$.

Fix $T>0$ and let $\eta_T$ be the characteristic function of the interval $[1,T]$. Then 
$$
(i\pa_t +\Delta_M) \eta_T \phi = \eta_T F + i (\phi(1) \de(t-1)-\phi(T)\de(t-T))
$$
By the global in time local smoothing estimate for $\si>1/2$,
$$
\|\eta_T \phi\|_{L^2_t H^{0,-1/2-\si}(M)}\le C \left (\| \eta_T F\|_{L^2_t H^{0,1/2+\si}(M)} +
\|\phi(1)\|_{L^2(M)} + \|\phi(T)\|_{L^2(M)}\right),
$$
which means that 
$$
\|\phi\|_{L^2_{[1,T]} H^{0,-1/2-\si}(M)}\le C \left (\|F\|_{L^2_{[1,T]} H^{0,1/2+\si}(M)} +
\|\phi(1)\|_{L^2(M)} + \|\phi(T)\|_{L^2(M)}\right).
$$
We also have the standard $L^2$ estimate, which implies that for any 
$t\in [1,T]$,
$$
\|\phi(t)\|_{L^2(M)}\le C \left (\|F\|_{L^2_{[0,T]} H^{0,1/2+\si}(M)} +
\|\phi(1)\|_{L^2(M)}\right).
$$
Adding the two estimates above we obtain the standard bound
\begin{equation}\label{eq:smooth-est}
\|\phi(t)\|_{L^2(M)}+\|\phi\|_{L^2_{[1,T]} H^{0,-1/2-\si}(M)}
\le C \left (\|F\|_{L^2_{[1,T]} H^{0,1/2+\si}(M)} +
\|\phi(1)\|_{L^2(M)}\right).
\end{equation}

To relate $\phi=P\psi$ back to $\psi$, we now work on developing the heuristic \eqref{Heuristic}.  We compute
$$
\Delta_M \left(e^{-i\chi \frac{r^2}{4t}}\psi\right)= e^{-i\chi \frac{r^2}{4t}}\left(\Delta_M 
-\frac i{2t}(2\chi r +\chi' r^2)\pa_r\right)\psi + \psi \Delta_M \left (e^{-i\chi \frac{r^2}{4t}}\right).$$
We expand the final term $\Delta_M \left (e^{-i\chi \frac{r^2}{4t}}\right)$:
\begin{align*}
 \Delta_M \left (e^{-i\chi \frac{r^2}{4t}}\right)&=-\frac i{4t} \pa_r\left 
(e^{-i\chi \frac{r^2}{4t}} ({2r}\chi+\chi' {r^2})\right)
-\frac i{4t} (\frac {n-1}r+\theta) ({2r}\chi+\chi' {r^2}) e^{-i\chi \frac{r^2}{4t}}\\ &
=-\frac i{4t}  e^{-i\chi \frac{r^2}{4t}} \left (-\frac{i}{4t} (2r\chi+\chi' r^2)^2 + 4\chi'r + 2\chi+\chi''r^2 \right. \\
&\quad\quad \left.+
 (\frac {n-1}r+\theta) ({2r}\chi+\chi' {r^2})\right)
\end{align*}
We thus obtain the pointwise estimate
$$
\left|\phi - e^{i\chi \frac{r^2}{4t}}\, t^2 \Delta_M \left (e^{-i\chi \frac{r^2}{4t}}\psi\right)\right|\le 
C(t\zeta |\pa_r\psi| + (1+t) \zeta |\psi|).
$$
Inserting this bound into \eqref{eq:smooth-est} and  letting 
$$
u:= e^{-i\chi \frac{r^2}{4t}}\psi
$$
we obtain
\begin{align*}
\|t^2 \Delta_M u \|_{L^2(M)} &+\|t^2 \Delta_M u \|_{L^2_{[1,T]} H^{0,-1/2-\si}(M)}\le C \left 
(\|F\|_{L^2_{[1,T]} H^{0,1/2+\si}(M)}\right. \\ &+ \left.t\, \|\zeta |\pa_r\psi| + (1+t)  
\zeta |\psi|\, \|_{L^2_{[1,T]} H^{0,-1/2-\si}(M)}+\|\phi(1)\|_{L^2(M)}\right)
\end{align*}
We now observe that
\begin{align*}
&|\psi|=|u|,\quad |\nabla_\omega\psi|=|\nabla u|,\quad 
|\pa_r\psi|\le |\pa_r u| +\frac r{t} |u|,\\
&|\nabla_\omega^2\psi|=|\nabla_\omega^2 u|,\quad |\pa_r^2\psi|\le 
|\pa_r^2 u|+ \frac rt |\pa_r u| + \frac {r^2}{t^2} |u|^2.
\end{align*}
Therefore,
\begin{align*}
|F|+t\, \zeta |\pa_r\psi|+ t 
\zeta |\psi|&\le C \,t\, \left (\zeta |\pa^2_r u| + r^{-2-2\si_0}  |\nabla_\omega^2u|\right.\\&+\left. \zeta |\pa_r u|+ 
r^{-2-2\si_0} 
|\nabla_\omega u| + r^{-2-2\si_0} |u|\right) 
\end{align*}
We thus have 
\begin{align*}
\|t^2 \Delta_M u &\|_{L^2(M)}+\|t^2 \Delta_M u \|_{L^2_{[1,T]} H^{0,-1/2-\si}(M)}\\
 &\le C \Big
(\|t\pa^2_r u\|_{L^2_{[1,T]} H^{0,-1/2-\si}(M)}+ \|t\pa_r u\|_{L^2_{[1,T]} H^{0,-1/2-\si}(M)} \\ 
&\quad +
\|t r^{-1}\nabla_\omega u\|_{L^2_{[1,T]} H^{0,-1/2-2\si_0+\si}(M)} +
\|t r^{-2}\nabla^2_\omega u\|_{L^2_{[1,T]} H^{0,1/2-2\si_0+\si}(M)} 
\\ 
&\quad +\|t u\|_{L^2_{[1,T]} H^{0,-3/2-2\si_0+\si}(M)}+
\|\phi(1)\|_{L^2(M)} \Big).
\end{align*}
We can couple this bound with the local smoothing estimate for the original solution $\psi$, which implies (as $u$ has the same magnitude as $\psi$) that
$$
\|u \|_{L^2(M)}+\|u \|_{L^2_{[1,T]} H^{0,-1/2-\si}(M)}\le C \|u(1)\|_{L^2(M)}.
$$
By combining these two estimates, we arrive at the estimate 
\begin{equation}\label{eq:tron}
\begin{split}
\|t^2 \Delta_M u \|_{L^2(M)}&+\|t^2 \Delta_M u \|_{L^2_{[1,T]} H^{0,-1/2-\si}(M)}\\ &\le C \Big
(\|t r^{-2}\nabla^2_\omega u\|_{L^2_{[1,T]} H^{0,1/2-2\si_0+\si}(M)} +
\|t u\|_{L^2_{[1,T]} H^{0,-3/2-2\si_0+\si}(M)}\\ &+
 \|u(1)\|_{L^2(M)}+\|\phi(1)\|_{L^2(M)} \Big)
\end{split}
\end{equation}
The first two terms on the right hand side require special care. We introduce a radius parameter $R$ and split
\begin{align*}
\|t r^{-2}\nabla^2_\omega u\|_{L^2_{[1,T]} H^{0,1/2-2\si_0+\si}(M)}& \le 
\|t r^{-2}\nabla^2_\omega u\|_{L^2_{[1,T]} H^{0,1/2-2\si_0+\si}(M_R)}\\ &+
\|t r^{-2}\nabla^2_\omega u\|_{L^2_{[1,T]} H^{0,1/2-2\si_0+\si}(M_R^c)}\\&\le 
C(R) \|t r^{-2}\nabla^2_\omega u\|_{L^2_{[1,T]} H^{0,-1/2-\si}(M)}\\&+
C R^{1/2-2\si_0+\si}\,  \sup_{t\in [1,T]}  \|t^2 \,r^{-2}\nabla^2_\omega u\|_{L^2(M)}
\end{align*}
For $\si_0>\si>1/2$ the second term can be absorbed by the $\|t^2 \Delta_M u \|_{L^2(M)}$
in the left-hand side of \eqref{eq:tron}, while the first term, via Bochner identities, can be partially absorbed
by the second term\footnote{Note that Bochner identity and our assumptions on the metric
$g$ imply that 
$$
\|\nabla^2 u\|_{L^2(M)} + \|\langle x\rangle^{-1} |\nabla u|_g\|_{L^2(M)}\le C\|\Delta_M u\|_{L^2} .
$$ The dependence of constant $C$ on the manifold $M$ is implicit as this inequality proved 
with the help of compactness arguments. This bound allows us to prove Sobolev multiplicative 
inequalities of the form 
$$
\|u\|_{L^p(M)}\le C \|\Delta_M u\|^\alpha_{L^2(M)} \|u\|^{1-\alpha}_{L^2(M)},  \quad
2\alpha =(\frac n2 -\frac np), \,\, 0\le\alpha\le 1, \,\, (\alpha,p)\ne (1,\infty).
$$
} on the left hand side of \eqref{eq:tron} with the remaining residual term
of the same form as  $\|t u\|_{L^2_{[1,T]} H^{0,-3/2-2\si_0+\si}(M)}$.

Therefore, we have that 
\begin{equation}\label{eq:stul}
\begin{split}
\|t^2 \Delta_M u \|_{L^\infty_{[1,T]} L^2(M)}&+\|t^2 \Delta_M u \|_{L^2_{[1,T]} H^{0,-1/2-\si}(M)}\le C \Big
(\|t u\|_{L^2_{[1,T]} H^{0,-3/2-2\si_0+\si}(M)}\\ &+
 \|u(1)\|_{L^2(M)}+\|\phi(1)\|_{L^2(M)} \Big)
\end{split}
\end{equation}
In dimension $n=3$ we now proceed as follows. Using a local smoothing estimate for $u$
we get the bound 
$$
\|t^2 \Delta_M u(t) \|_{L^2(M)}\le C t( \|u(1)\|_{L^2(M)}+\|\phi(1)\|_{L^2(M)} ),
$$
which gives the preliminary bound
\begin{equation}\label{preli}
\|\Delta_M u(t) \|_{L^2(M)}\le \frac {C} t( \|u(1)\|_{L^2(M)}+\|\phi(1)\|_{L^2(M)} )
\end{equation}
for all $t\geq 1$.
By the Sobolev estimate, 
$$
\|u(t)\|_{L^\infty(M)} \le C \|\Delta_M u(t)\|^{\frac 34}_{L^2(M)} \|u(t)\|^{\frac 14}_{L^2(M)}\le
 \frac {C}{t^{\frac 34}} ( \|u(1)\|_{L^2(M)}+\|\phi(1)\|_{L^2(M)} )
$$
Therefore,
$$
\|t u\|_{L^2_{[1,T]} H^{0,-3/2-2\si_0+\si}(M)}\le T^{\frac 12} \|t u\|_{L^\infty_{[1,T]} L^\infty(M)}
\le T^{\frac 14}  ( \|u(1)\|_{L^2(M)}+\|\phi(1)\|_{L^2(M)} ).
$$
We can thus bootstrap the preliminary bound \eqref{preli} to the improved bound 
$$
\|\Delta_M u(t) \|_{L^2(M)}\le \frac {C}{t^{\frac 54}} ( \|u(1)\|_{L^2(M)}+\|\phi(1)\|_{L^2(M)} )
$$
for any $t \geq 1$.

We now iterate this process.  If on the $n^{\operatorname{th}}$ step the decay rate of  $\|\Delta_M u \|_{L^2(M)}$
is $t^{-\alpha_n}$, then 
$$
\alpha_{n+1}=\frac 34\alpha_n+\frac 12
$$
and thus $\alpha_n\to 2$. It then follows that 
$$
\|\psi(t)\|_{L^\infty(M)} \le C_\eps t^{-\frac 32+\eps}  ( \|\psi(1)\|_{L^2(M)}+\|\phi(1)\|_{L^2(M)} )
$$
for any $\epsilon>0$ and $t \geq 1$. 

In higher dimensions the norm $\|t u(t)\|_{H^{0,-3/2-2\si_0+\si}(M)}$ is $L^2$ integrable and 
can be dealt with directly in one step by interpolating between $\|\Delta_M u \|_{L^2(M)}$
and $\|u \|_{L^2(M)}$. We omit the details.

\section{Decay estimates for the wave equation}\label{wave-decay}
\newcommand\g{{\bf g}}

In this section we prove\footnote{The argument below is an optimal but has been retained for illustrative purposes.
The following decay rates can be improved via an alternative approach as  in  \cite{dr3}.}
Proposition \ref{prop:wave-S}.  For $t \leq 2$ the claim is easily established from $H^2 \times H^1$ energy estimates and Sobolev embedding, so we shall limit ourselves to the case when $t>2$.

Let $u$ be a solution of the wave equation
\begin{align*}
&\Box_Mu: =u_{tt}-\Delta_M u=0,\\
&u|_{t=0}=u_0,\,\,\,u_{t}|_{t=0}=u_1.
\end{align*}
In Section \ref{sec:conserv} we defined the energy-momentum tensor $\Q_{mk}$
associated with the Helmholtz equation $(H-z^2)u=0$. Here we use its spacetime counterpart\footnote{Greek
indices $\a,\b=0,\ldots,n$ with index $0$ corresponding to the $t$ coordinate. Operations of raising and 
lowering of indices are done with respect to the space-time metric $\g=-dt^2+g$. Finally, $D$
will denote the Levi-Civita connection of metric $\g$.} 
$$
\Q_{\a\b} := \pa_\a u \pa_\b u
- \frac{1}{2} \g_{\a\b}\,  \g^{\mu\nu} \pa_\mu u\,\pa_\nu u.
$$
As is well known, the energy-momentum tensor $\Q_{\a\b}$ is divergence free:
$$
D^\b \Q_{\a\b}=0.
$$
Let $K=K^\a \pa_\a$ be an arbitrary (smooth) vectorfield. We form the quantity 
$$
P_\a :=  \Q_{\a\b} K^\b
$$  
with the property that 
$$
D^\a P_\a = \Q_{\a\b} \pi^{\a\b}, \quad \pi_{\a\b} = \frac 12 (D_\a K_\b + D_\b K_\a),
$$
where $\pi_{\a\b} := \frac 12 {\mathcal L}_K g_{\a\b}$ is the deformation tensor of $K$. 
We note two more identities:
\begin{align*}
&\pa_\a (u^2) =2 \pa_\a u\, u,\\
&D^\a ( \pa_\a u\, u)= \pa^\a u\, \pa_\a u.
\end{align*}
Then choosing additional smooth function $a^\a$ and $b$, we have that 
$$
D^\a \left  (a_\a u^2 + b\, \pa_\a u\, u+ \Q_{\a\b} K^\b\right) =
D^\a a_\a u^2 + (2 a_\a + \pa_\a b) + b  \pa^\a u\, \pa_\a u +
 \Q_{\a\b} \pi^{\a\b}
$$
or more explicitly
\begin{align*}
D^\a \left  (a_\a u^2 + b\, \pa_\a u\, u+ \Q_{\a\b} K^\b\right) &=
D^\a a_\a u^2 + (2 a_\a + \pa_\a b)  \pa^\a u\, u  \\ &+
\left(\pi^{\a\b}+(b-\frac 12 {\text{tr}} \pi) \g^{\a\b}\right) \pa_\a u\, \pa_\b u,
\end{align*}
where $ {\text{tr}} \pi=\g^{\mu\nu} \pi_{\mu\nu}$.

It is clear that one of the choices of $a^\a, b$ to simplify the right hand side is 
$$
a_\a=-\frac 12 \pa_\a b
$$
whence
$$
D^\a \left  (-\frac 12\pa_\a b\, u^2 + b\, \pa_\a u\, u+ \Q_{\a\b} K^\b\right) =
\square_M b\, u^2 +
\left(\pi^{\a\b}+(b-\frac 12 {\text{tr}} \pi) \g^{\a\b}\right) \pa_\a u\, \pa_\b u
$$
Integrating this expression in the space-time slab $[t,0]\times M$ we obtain
\begin{equation}\label{eq:stokes}
\begin{split}
\int_{M\times \{0\}}&\left(-\frac 12 \pa_t b\, u^2 + b\, \pa_t u\, u+ \Q_{0\b} K^\b\right)\, dg-
\int_{M\times \{t\}}\left(-\frac 12 \pa_t b\, u^2 + b\, \pa_t u\, u+ \Q_{0\b} K^\b\right)\, dg\\ &=
\int_{0}^{t} \int_M \left (\square_M b\, u^2 +
\left(\pi^{\a\b}+(b-\frac 12 {\text{tr}} \pi) \g^{\a\b}\right) \pa_\a u\, \pa_\b u\right)\ dg\, dt
\end{split}
\end{equation}

We now make the remaining choices of the vector field $K$ and function $b$. 
For $K$ we choose a modification of the Morawetz vectorfield $(t^2+r^2) \pa_t +2tr \pa_r$,
known to play an important role in the study of the decay properties of solutions of the
wave equation in Minkowski space and obstacle problems (see \cite{morawetz}).
Define
$$
K:=(t^2+\chi r^2) \pa_t +2tr \chi \pa_r
$$
with a cut-off function $\chi$ of the previous section supported in $\langle x\rangle \le r_0$
and equal to 1 for $\langle x\rangle \ge 2r_0$. The deformation tensor of $K$ can be computed 
as follows:
\begin{align*}
&\pi_{00}=-2t,\quad \pi_{rr}=2t \chi + 2t r\chi',\\
&\pi_{ab}= 2t \chi (h[r]_{ab}+r \theta_{ab}),\quad \pi_{0r}=-\chi'r^2,\\
&{\text{tr}}\pi= 2t + 2t\chi (n-1) + 2t\chi r \theta+2tr \chi' 
\end{align*}
We then choose 
$$
b=(n-1) t + 2t\chi r \theta
$$
so that 
\begin{equation}\label{eq:pi}
\left(\pi^{\a\b}+(b-\frac 12 {\text{tr}} \pi) \g^{\a\b}\right) \nabla_\a u\, \nabla_\b u=
2tr^{-1}\chi \theta_{ab} \nabla^a_\omega u\, \nabla_\omega^b u + O(\zeta) (1+t) |\nabla u|^2,
\end{equation}
where $\zeta$ is a smooth cut-off supported in $r_0\le \langle x\rangle \le 2r_0$. Furthermore,
\begin{equation}\label{eq:squareb}
\square_M b = 2t \Delta_g (\chi r\theta)=t \,\,O(r^{-2-2\si_0}) 
\end{equation}
We now analyze the expression $\left(-\frac 12 \pa_t b\, u^2 + b\, \pa_t u\, u+ \Q_{0\b} K^\b\right)$.
\begin{lemma}\label{lem:AB}
Let $n\ge 3$. Then for sufficiently large $r_0$ and all $t\gg r_0$
\begin{equation}\label{eq:QM}
\begin{split}
\int_{M}&\left(-\frac 12 \pa_t b\, u^2 + b\, \pa_t u\, u+ \Q_{0\b} K^\b\right)\, dg \ge 
C(r_0)\,\, t^2\,\,\int_{M_{4r_0}^c} (|\pa_t u|^2+|\nabla u|^2 + u^2)\ dg \\&+ C(r_0) \int_{M_{4r_0}} 
\left ((t+r)^2|(\pa_t+\pa_r) u|^2+(t-r)^2|(\pa_t-\pa_r) u|^2\right)\ dg\\ &+ C(r_0) \int_{M_{4r_0}} 
\left((t^2+r^2)
|r^{-1}\nabla_\omega u|^2 + (1+ t^2 r^{-2}) u^2\right)\ dg\\ &-
C(r_0)\,\, \,\,\int_{M_{2r_0}} O(r^{-2\si_0}) (|\pa_t u|^2+|\nabla u|^2)\ dg.
\end{split}
\end{equation}
\end{lemma}
\begin{remark}
Similar arguments can be found in \cite{kl} for the wave equation in Minkowski
space and \cite{dr1} for the wave equation on Schwarzschild background.
\end{remark}

\begin{proof}
First observe that with our choices of $b, K^\a$
\begin{align*}
Q: =-\frac 12 \pa_t b\, u^2 + b\, \pa_t u\, u+ \Q_{0\b} K^\b&=\frac 12 (t^2+\chi r^2) ( |\pa_t u|^2 +|\nabla u|_g^2)
+2tr \chi \pa_t u\, \pa_r u \\ &+ t(n-1+2\chi r\theta) \pa_t u \, u -\frac 12 (n-1+2\chi r\theta ) u^2.
\end{align*}
On the complement of the set $M_{2r_0}$ we can thus bound $Q$ as follows.
\begin{equation}\label{eq:Qcom}
Q\ge \, \frac {t^2}4  ( |\pa_t u|^2 +|\nabla u|_g^2) - C(r_0) u^2,\qquad{\text{on}}\,\, M_{4r_0}^c
\end{equation}
provided that $t\gg r_0$. Let $\tilde \chi$ be a cut-off function supported in $M_{2r_0}$ and 
equal to 1 on $M_{4r_0}$. To prove the desired result it will be sufficient to show that
\begin{align*}
\int_M \tilde\chi &\left (\frac 12 (t^2+r^2) ( |\pa_t u|^2 +|\nabla u|_g^2)
+2tr \pa_t u\, \pa_r u + t(n-1+2 r\theta) \pa_t u \, u -\frac 12 (n-1+2 r\theta ) u^2\right)\\ &\ge 
 C(r_0) \int_{M_{4r_0}} 
\left ((t+r)^2|(\pa_t+\pa_r) u|^2+(t-r)^2|(\pa_t-\pa_r) u|^2\right)\ dg\\ &+ C(r_0) \int_{M_{4r_0}} 
\left((t^2+r^2)
|r^{-1}\nabla_\omega u|^2 + (1+ t^2 r^{-2}) u^2\right)\ dg\\ &-
C(r_0)\,\, \,\,\int_{M_{2r_0}} O(r^{-2\si_0}) (|\pa_t u|^2+|\nabla u|^2)\ dg
\end{align*}
Note that the $t^2 u^2$ term on $M_{4r_0}^c$ can be obtained via a Poincar\'e inequality 
from the estimate above and \eqref{eq:Qcom}.

Define 
$$
S u = (t\pa_t+r\pa_r) u,\qquad \underline S = (t \pa_r +r\pa_t u).
$$
Then 
$$
\frac 12 (t^2+r^2) ( |\pa_t u|^2 +|\nabla u|_g^2)
+2tr \pa_t u\, \pa_r u =\frac 12 \left((Su)^2+(\underline S u)^2\right)+\frac 12 (t^2+r^2) 
|r^{-1}\nabla_\omega u|_g^2.
$$
Furthermore,
$$
t\pa_t u \, u= Su u - \frac 12 r \pa_r (u)^2,\quad t\pa_t u \, u= \frac tr \underline S u u-\frac {t^2}{2r}
\pa_r(u^2).  
$$
As a consequence,
\begin{align*}
&t(n-1) \int_M \tilde\chi \pa_t u \, u=(n-1) \int_M \tilde\chi S u \, u + \frac {n-1}2
\int_M \left (\tilde \chi(n+r\theta) + r\tilde\chi'\right) u^2,\\
&t(n-1) \int_M \tilde\chi \pa_t u \, u=(n-1) \frac tr\int_M \tilde\chi \underline S u \, u + \frac {n-1}2\, t^2
\int_M \left (\tilde \chi(n-2+r\theta) + r\tilde\chi'\right) \frac {u^2}{r^2}.
\end{align*}
It is important to note that $\tilde\chi'\ge 0$ and $(n-2)\gg r\theta$ on the support of $\tilde\chi$
provided that $r_0$ is sufficiently large. It is not difficult to show that for $n\ge 3$ one can find
constants $A, B$ such that 
$$
A+B=n-1,\quad A^2+B^2< (n-\frac 52)(n-3) + \frac 54(n-1) A.
$$
(in dimension $n=3$ the choice of $A=3/2, B=1/2$ is sufficient).
As a consequence,
\begin{align*}
\int_M \tilde\chi &\left (\frac 12 (t^2+r^2) ( |\pa_t u|^2 +|\nabla u|_g^2)
+2tr \pa_t u\, \pa_r u + t(n-1+2 r\theta) \pa_t u \, u -\frac 12 (n-1+2 r\theta ) u^2\right)\\ &\ge 
\frac 12\int_{M}\tilde\chi  \left ((Su+Au)^2+(\underline S u + \frac tr B u)^2
+(t^2+r^2) 
|r^{-1}\nabla_\omega u|_g^2\right)\ dg \\ &+
\frac 12\int_{M}\tilde\chi  \left(\left((n-2)(n-3) + (n-1)A-A^2-B^2+ \frac {n-1}{2r^2} t^2 B\right)
u^2\right)\ dg \\ &- \int_{M}\tilde\chi  O(r^{-2\si_0} )(u^2+ t |\pa_t u|\, |u|)\ dg.
\end{align*}
Running the argument again we obtain
that
\begin{align*}
\int_M \tilde\chi &\left (\frac 12 (t^2+r^2) ( |\pa_t u|^2 +|\nabla u|_g^2)
+2tr \pa_t u\, \pa_r u + t(n-1+2 r\theta) \pa_t u \, u -\frac 12 (n-1+2 r\theta ) u^2\right)\\ &\ge 
c \int_{M}\tilde\chi  \left ((Su)^2+(\underline S u)^2+(t^2+r^2) 
|r^{-1}\nabla_\omega u|_g^2+
(1+ \frac {t^2}{r^2})u^2\right)\ dg - 
\int_{M}\tilde\chi  O(r^{-2\si_0}) (\pa_t u)^2\ dg.
\end{align*}
The result now follows immediately from the identity
$$
(t+r)^2 ((\pa_r+\pa_r)u)^2 +(t-r)^2((\pa_t-\pa_r)u)^2 =  \left((Su)^2+(\underline S u)^2\right)
$$
\end{proof}
We introduce notations for the energy associated with the vectorfield $K$ and 
a conserved energy associated with the vectorfield $\pa/\pa t$:
\begin{align}
E_K(t) = E_{K,u}(T)&:= t^2\,\,\int_{M_{4r_0}^c} (|\pa_t u|^2+|\nabla u|^2 + u^2)\ dg \\&+\int_{M_{4r_0}} 
\left ((t+r)^2|(\pa_t+\pa_r) u|^2+(t-r)^2|(\pa_t-\pa_r) u|^2\right)\ dg\\ &+  \int_{M_{4r_0}} 
\left((t^2+r^2)
|r^{-1}\nabla_\omega u|^2 + (1+ t^2 r^{-2}) u^2\right)\ dg\label{eq:conformal}\\ 
E(t) = E_u(T) :=&\int_{M} (|\pa_t u|^2+|\nabla u|^2)\ dg.\nonumber
\end{align}
Combining Lemma \ref{lem:AB} and \eqref{eq:stokes}, we derive the following 
\begin{proposition}\label{prop:first}
Let $u$ be a solution of $\square_g u=0$ and let $T$ be sufficiently large. Then there
exists a sufficiently large $R\ll T$ and a 
smooth cut-off function $\zeta$ supported in the region 
$M_{r_0}\setminus M_{R}=\{x:\,r_0\le \langle x\rangle \le R\}$ such that
$$
E_K(T)\le C(r_0)\left (E_K(0) +
\int_0^T (1+t) \int_{M} \tilde \zeta |\nabla u|^2\ dg\, dt\right).
$$ 
\end{proposition}
\begin{proof}
 Lemma \ref{lem:AB} together with \eqref{eq:stokes}, \eqref{eq:pi} and \eqref{eq:squareb} imply
 that all $T\gg r_0\gg 1$ 
 \begin{align*}
E_K(T)&\le C(r_0) (E_K(0) + E(T))\\ &+ 
C(r_0)
\int_0^T \int_{M_{r_0}} \left (O(r^{-2\si_0})\,t\, \left (|r^{-1}\nabla_\omega u|_g^2+ r^{-2} u^2\right)+
O(\zeta) (1+t) |\nabla u|^2\right)\ dg\, dt.
\end{align*}
 Since the expression for the energy $E_K(t)$ contains both $t^2 |r^{-1}\nabla_\omega u|^2$ and
 $t^2 r^{-2} u^2$ we can find a sufficiently large $R$ such that  the above inequality can be 
simplified to 
$$
E_K(T)\le C(r_0) (E_K(0) + E(0))+ 
C(r_0)
\int_0^T (1+t) \int_{M} \tilde \zeta |\nabla u|^2\ dg\, dt
$$ 
with $\tilde \zeta$ supported in $M_{r_0}\setminus M_{R}$. Note that to get to the last inequality 
we also used a Poincar\'e inequality to convert the $u^2$ term into $|\nabla u|_g^2$, and 
conservation of the energy $E(t)$. To conclude the proof of the proposition it remains to observe
that the energy $E_K(0)$ easily dominates $E(0)$.
\end{proof}

Henceforth we specialize to the three-dimensional case $n=3$.

The integrated local energy decay estimate of Proposition \ref{prop:int-decay} gives the bound
$$
\int_0^T (1+t) \int_{M} \tilde \zeta |\nabla u|^2\ dg\, dt\le (1+T) \int_0^T 
\int_{M} \tilde \zeta |\nabla u|^2\ dg\, dt \le C (1+T) E(0).
$$
This immediately implies that 
$$
E_{K}(T)\le C(1+T) E_K(0).
$$
In particular, as can be seen from \eqref{eq:conformal}, 
inequality above leads to the $1/t$ decay of the local energy. Precisely, 
for any compact set $B\subset M$ 
\begin{equation}\label{eq:loc1-decay}
\int_B ((\pa_t u)^2 + |\nabla u|_g^2)\ dg \le \frac {C(B)}t E_K(0).
\end{equation}
Applying in addition the same argument to the function $v=\pa_t u$ we then have
$$
\int_B ((\pa_t^2 u)^2 + |\pa_t \nabla u|_g^2)\ dg \le \frac {C(B)}t E_{K, \pa_t u}(0).
$$
From the Gagliardo-Nirenberg inequality we thus obtain uniform decay for $u$:
$$
\|u(t)\|_{L^\infty_x} \le C_\eps t^{-\frac 12+\eps} \left (E_{K,u}(0)+ E_{K,\pa_t u} (0)\right).
$$
for any $\eps>0$. 

This however is not an optimal result (in particular, as suggested by our experience with 
the wave equation in Minkowski space in dimension $n=3$). To improve on it one observes
that Proposition \ref{prop:first} can be iterated, \cite{dr1}. Heuristically (for more details
see \cite{dr1}), the argument is as follows.
First partition the time interval $[0,T]$ dyadically into subintervals of the form 
$[t,\Gamma t]$ with some $\Gamma>1$ sufficiently close to $1$. By finite speed of 
propagation, on the set $[t,\Gamma t]\times M_{R}^c$, solution $u$ is completely determined
by its values at time $t$ on the set $M_{R+C(\Gamma-1)t}^c$ for some universal constant $C$
dependent on the manifold $M$. In particular, $\Gamma$ can be chosen in such a way that 
$R+C(\Gamma-1)t<1/2 t$. Now, the estimate $E_K(t)\le C(1+t) E_K(0)$ in fact gives more than
just the local energy decay. One can easily show that \eqref{eq:loc1-decay} holds for any
set $B$ such that $B\subset M_{1/2t}^c$ with a constant $C(B)$ independent of $B$
(in particular of $t$). In view of the above discussion the integrated local energy decay estimate
for $u$ on the interval $[t,\Gamma t]$ should take the form 
$$
\int_{t}^{\Gamma t} \int \tilde \zeta ((\pa_t u)^2+|\nabla u|_g^2)\ dg\, dt \le 
C \int_{\{t\}\times M_{R+C(\Gamma-1)t}^c} ((\pa_t u)^2+|\nabla u|_g^2)\ dg\le 
\frac Ct E_K(0).
$$
Adding these estimates over all such time subintervals we obtain that 
$$
\int_0^T (1+t) \int \tilde \zeta ((\pa_t u)^2+|\nabla u|_g^2)\ dg\, dt \le \log T E_K(0).
$$
Using Proposition \ref{prop:first} now implies that 
$$
E_K(T)\le C(r_0) \log T E_K(0)
$$
and 
$$
\int_B ((\pa_t u(t))^2 + |\nabla u(t)|_g^2)\ dg \le \frac {C(B)\log t}{t^2} E_K(0)
$$
for all $t \leq T$.  Repeating the whole argument one more time eliminates the remaining $\log t$ term and 
shows that 
$$
E_K(T)\le C(r_0) E_K(0)
$$
Finally, combining the bound for $u$ with the similar bound for $v=\pa_t u$ and using Sobolev
inequalities we arrive at the desired estimate
$$
\|u(t)\|_{L^\infty_x} \le C_\eps {t}^{-1+\eps} \left (E_{K,u}(0)+ E_{K,\pa_t u}(0)\right).
$$

\end{document}